\documentclass[a4paper, reqno, 12pt]{amsart}

\usepackage[usenames,dvipsnames]{color}
\usepackage{amsthm,amsfonts,amssymb,amsmath,amsxtra}
\usepackage[all]{xy}
\SelectTips{cm}{}
\usepackage{xr-hyper}
\usepackage[colorlinks=
   citecolor=Black,
   linkcolor=Red,
   urlcolor=Blue]{hyperref}
\usepackage{verbatim}

\usepackage[margin=1.25in]{geometry}
\usepackage{mathrsfs}

\RequirePackage{xspace}
\RequirePackage{etoolbox}
\RequirePackage{varwidth}
\RequirePackage{enumitem}
\RequirePackage{tensor}
\RequirePackage{mathtools}
\RequirePackage{longtable}
\RequirePackage{multirow}

\usepackage[url=false,doi=false,isbn=false,sorting=nyt,giveninits=true,maxbibnames=4]{biblatex}
\def\le{\leqslant}
\def\a{\alpha}

\def\L{\Lambda}

\def\o{\omega}
\def\p{\pi}

\def\s{\sigma}

\def\l{\lambda}

\def\<{\langle}
\def\>{\rangle}

\newcommand{{\BG}}{\ensuremath{\mathbb {G}}\xspace}

\newcommand{{\BK}}{\ensuremath{\mathbb {K}}\xspace}

\newtheorem{theorem}{Theorem}
\newtheorem{proposition}[theorem]{Proposition}
\newtheorem{lemma}[theorem]{Lemma}

\newtheorem{corollary}[theorem]{Corollary}

\theoremstyle{definition}

\newtheorem{remark}[theorem]{Remark}

\numberwithin{equation}{section}
\numberwithin{theorem}{section}

\setitemize[0]{leftmargin=*,itemsep=\the\smallskipamount}
\setenumerate[0]{leftmargin=*,itemsep=\the\smallskipamount}

\renewcommand{\to}{%
   \ifbool{@display}{\longrightarrow}{\rightarrow}%
   }
\let\shortmapsto\mapsto
\renewcommand{\mapsto}{%
   \ifbool{@display}{\longmapsto}{\shortmapsto}%
   }
\newlength{\olen}
\newlength{\ulen}
\newlength{\xlen}
\newcommand{\xra}[2][]{%
   \ifbool{@display}%
      {\settowidth{\olen}{$\overset{#2}{\longrightarrow}$}%
       \settowidth{\ulen}{$\underset{#1}{\longrightarrow}$}%
       \settowidth{\xlen}{$\xrightarrow[#1]{#2}$}%
       \ifdimgreater{\olen}{\xlen}%
          {\underset{#1}{\overset{#2}{\longrightarrow}}}%
          {\ifdimgreater{\ulen}{\xlen}%
             {\underset{#1}{\overset{#2}{\longrightarrow}}}
             {\xrightarrow[#1]{#2}}}}%
      {\xrightarrow[#1]{#2}}
   }
\makeatother
\newcommand{\xyra}[2][]{%
   \settowidth{\xlen}{$\xrightarrow[#1]{#2}$}%
   \ifbool{@display}%
      {\settowidth{\olen}{$\overset{#2}{\longrightarrow}$}%
       \settowidth{\ulen}{$\underset{#1}{\longrightarrow}$}%
       \ifdimgreater{\olen}{\xlen}%
          {\mathrel{\xymatrix@M=.12ex@C=3.2ex{\ar[r]^-{#2}_-{#1} &}}}%
          {\ifdimgreater{\ulen}{\xlen}%
             {\mathrel{\xymatrix@M=.12ex@C=3.2ex{\ar[r]^-{#2}_-{#1} &}}}
             {\mathrel{\xymatrix@M=.12ex@C=\the\xlen{\ar[r]^-{#2}_-{#1} &}}}}}%
      {\mathrel{\xymatrix@M=.12ex@C=\the\xlen{\ar[r]^-{#2}_-{#1} &}}}%
   }
\makeatletter
\newcommand{\xla}[2][]{%
   \ifbool{@display}%
      {\settowidth{\olen}{$\overset{#2}{\longleftarrow}$}%
       \settowidth{\ulen}{$\underset{#1}{\longleftarrow}$}%
       \settowidth{\xlen}{$\xleftarrow[#1]{#2}$}%
       \ifdimgreater{\olen}{\xlen}%
          {\underset{#1}{\overset{#2}{\longleftarrow}}}%
          {\ifdimgreater{\ulen}{\xlen}%
             {\underset{#1}{\overset{#2}{\longleftarrow}}}
             {\xleftarrow[#1]{#2}}}}%
      {\xleftarrow[#1]{#2}}
   }
\newcommand{\isoarrow}{%
   \ifbool{@display}{\overset{\sim}{\longrightarrow}}{\xrightarrow\sim}%
   }

\def\blambda{\boldsymbol{\lambda}}

\begin{document}

\title[]{Thickening realization and positivity properties of canonical bases}

\author[Jiepeng Fang]{Jiepeng Fang}
\address{Department of Mathematics and New Cornerstone Science Laboratory, The University of Hong Kong, Pokfulam, Hong Kong, Hong Kong SAR, China}
\email{fangjp@hku.hk}

\author[Xuhua He]{Xuhua He}
\address{Department of Mathematics and New Cornerstone Science Laboratory, The University of Hong Kong, Pokfulam, Hong Kong, Hong Kong SAR, China}
\email{xuhuahe@hku.hk}

\thanks{}




\begin{abstract}
Let $\mathbf{U}$ be a quantum group associated with a symmetric Cartan datum, let $\dot{\mathbf{U}}$ be its modified form, and let $\dot{\mathbf{B}}$ be the canonical basis of $\dot{\mathbf{U}}$. Lusztig conjectured that the structure constants of the multiplication, comultiplication, and bilinear form in $\dot{\mathbf{U}}$ with respect to $\dot{\mathbf{B}}$ belong to $\mathbb{N}[v,v^{-1}]$.

We introduce the \emph{thickening realization}, which relates $\dot{\mathbf{B}}$ to the canonical basis of the negative part of a larger quantum group $\tilde{\mathbf{U}}^-$. More precisely, it identifies the relevant structure constants in $\dot{\mathbf{U}}$ with the structure constants in $\tilde{\mathbf{U}}^-$. As consequences, we prove, for arbitrary symmetric Cartan datum, $\dot{\mathbf{B}}$ has the positivity properties for the comultiplication and bilinear form, as well as the positivity for the multiplication whenever one factor is spherical parabolic. In particular, Lusztig's conjecture holds for simply-laced finite type. We also prove the canonical bases of a broad class of tensor products of integrable modules have the positivity properties for the transition matrices and actions by $\dot{\mathbf{B}}$.
\end{abstract}

\maketitle


\section*{Introduction}

\subsection{Lusztig's positivity conjecture}

The theory of canonical bases, first established by Lusztig \cite{Lusztig-1990, Lusztig-1991} and further studied by Kashiwara \cite{Kashiwara-1991}, reveals profound and elegant structures on quantum groups and their integrable modules. 

A cornerstone of this theory is the \textbf{positivity property}, established by Lusztig \cite{Lusztig-1993}. Let $\mathbf{U}$ be the quantum group associated with a symmetric Cartan datum. For the canonical bases of the negative part $\mathbf{U}^-$ and positive part $\mathbf{U}^+$, Lusztig proved that they have the positivity properties with respect to the multiplication, comultiplication and bilinear form \cite[Theorem 14.4.13]{Lusztig-1993}. For the canonical bases of the simple
integrable highest and lowest weight modules, Lusztig proved that they have the positivity properties with respect to the actions by the Chevalley generators  \cite[Theorem 22.1.7]{Lusztig-1993}.

To extend canonical-basis theory to the full quantum group, Lusztig
introduced the modified quantum group $\dot{\mathbf{U}}$, in which
the Cartan part is replaced by a system of mutually orthogonal
idempotents. The category of unital
$\dot{\mathbf{U}}$-modules is equivalent to the category of weight
$\mathbf{U}$-modules, and $\dot{\mathbf{U}}$ admits a canonical basis
$\dot{\mathbf{B}}$ \cite[Part IV]{Lusztig-1993}. In
\cite[Conjecture 25.4.2]{Lusztig-1993}, Lusztig conjectured that
$\dot{\mathbf{B}}$ has the following positivity properties:
\begin{itemize}
\item \textbf{(Multiplication)}
$\dot{\mathbf{B}}\cdot\dot{\mathbf{B}}
\subset \mathbb{N}[v,v^{-1}][\dot{\mathbf{B}}]$;
\item \textbf{(Comultiplication)}
$\dot{\Delta}(\dot{\mathbf{B}})
\subset \mathbb{N}[v,v^{-1}]
[\dot{\mathbf{B}}\hat{\otimes}\dot{\mathbf{B}}]$;
\end{itemize}
A closely related positivity question concerns the bilinear form on $\dot{\mathbf{U}}$:
\[
(\dot{\mathbf{B}},\dot{\mathbf{B}})_{\dot{\mathbf{U}}}
\subset\mathbb{N}[[v^{-1}]].
\]
These questions ask whether the rigid positivity structures of $\mathbf{U}^-$ and 
$\mathbf{U}^+$ persist after they are assembled in $\dot{\mathbf{U}}$.

\subsection{Previous work}
Geometric and categorical realizations provide the principal source of
positivity results for canonical bases. 

Lusztig realized the negative part $\mathbf{U}^-$ and its canonical basis using perverse sheaves on varieties of quiver representations \cite{Lusztig-1990,Lusztig-1991}. Khovanov-Lauda \cite{Khovanov-Lauda-2009} and Rouquier \cite{Rouquier-2008,Rouquier-2012} subsequently categorified $\mathbf{U}^-$ using quiver Hecke algebras, and Varagnolo--Vasserot \cite{Varagnolo-Vasserot-2011} identified the canonical basis elements with corresponding indecomposable projective modules. Related constructions for simple integrable highest weight modules were given in \cite{Kang-Kashiwara-2012,Zheng-2014}.

For the modified quantum group, Beilinson--Lusztig--MacPherson \cite{Beilinson-Lusztig-MacPherson-1990} constructed a geometric realization in finite type $A$ using partial flag varieties. This framework was extended to affine type $A$ by Ginzburg--Vasserot \cite{Ginzburg-Vasserot-1993} and Lusztig \cite{Lusztig-1999,Lusztig-2000}. Schiffmann-Vasserot \cite{Schiffmann-Vasserot-2000} and McGerty \cite{McGerty-2012} identified the canonical basis elements with corresponding geometric objects. These constructions established the multiplication positivity \cite{Beilinson-Lusztig-MacPherson-1990,Ginzburg-Vasserot-1993,Lusztig-1999,Lusztig-2000,Schiffmann-Vasserot-2000,McGerty-2012,Fu-Shoji-2018}, comultiplication positivity \cite{Fan-Li-2021,Fu-2021}, and bilinear-form positivity \cite{McGerty-2012} in type $A$. In finite type, Webster's $2$-categorical construction gave the multiplication positivity \cite{Webster-2015}; alternative geometric approaches and partial extensions appeared in \cite{Fang-Lan-2023,Fang-Lan-2025}.

These methods depend either on type $A$ geometry or on features specific to
finite type. Before the present work, the multiplication positivity was open
outside finite type and affine type $A$, while the comultiplication and
bilinear-form positivity were not known beyond the finite and affine type
$A$ settings. The main purpose of this paper is to provide a uniform
algebraic mechanism valid for arbitrary symmetric Cartan data.

\subsection{Main results}
We introduce an algebraic construction, called the \emph{thickening
realization}, that relates the tensor products of certain modules of $\mathbf{U}$ to the negative part of a larger quantum group. Let ${^{\omega}\Lambda_{\lambda_1}}\otimes \Lambda_{\lambda_2}$ be the tensor product of a simple integrable lowest weight module and a simple integrable highest weight module. For each Demazure submodule ${^{\omega}V_w(\lambda_1)}\subset
{^{\omega}\Lambda_{\lambda_1}}$, we realize
$${^{\omega}V_w(\lambda_1)}\otimes\Lambda_{\lambda_2}$$
as a subquotient of a Verma module of an enlarged quantum group
$\tilde{\mathbf{U}}$, compatibly with canonical bases. Since
${^{\omega}\Lambda_{\lambda_1}}$ is the direct limit of its Demazure
submodules, these realizations control the canonical basis of the full
tensor product ${^{\omega}\Lambda_{\lambda_1}}\otimes \Lambda_{\lambda_2}$. 
Through Lusztig's stable-limit construction of
$\dot{\mathbf{B}}$, these realizations yield identities relating
the relevant structure constants in $\dot{\mathbf{U}}$ to those
in the negative parts of enlarged quantum groups.

Here a canonical basis element is called \emph{spherical parabolic}
if it belongs to a parabolic subalgebra in which the positive
generators are restricted to a finite-type subsystem and the negative
generators are unrestricted, or vice versa.

The structural results may be summarized as follows.

\begin{theorem}[Theorem \ref{thm:structure-multiplication}, \ref{thm:str-const-comult}]
For a quantum group associated with a symmetric Cartan datum, the following
statements hold.
\begin{enumerate}
\item \textbf{(Spherical multiplication)} If
$\dot{b},\dot{b}'\in\dot{\mathbf{B}}$ and $\dot{b}$ belongs to a
spherical parabolic subalgebra of $\dot{\mathbf{U}}$, then the structure constants for $\dot{b}\dot{b}'$ are given by the multiplication structure constants for the canonical basis of the negative part of a further enlargement
$\tilde{\tilde{\mathbf{U}}}$.
\item \textbf{(Comultiplication)} The structure constants for $\dot{\Delta}(\dot{b}),\dot{b}\in\dot{\mathbf{B}}$ are given, up to explicit powers of $v$, by the comultiplication structure constants for the canonical basis of $\tilde{\mathbf{U}}^-$.
\end{enumerate}
\end{theorem}

Together with the compatibility of canonical bases and bilinear
forms under thickening, these identities allow us to apply Lusztig's
positivity theorem for the negative part \cite[Theorem 14.4.13]{Lusztig-1993} and obtain the following
results.

\begin{theorem}[Theorems \ref{thm:trans-pos}, \ref{positivity of bilinear form}, \ref{thm:comult-positivity}, \ref{thm:mult}, Proposition \ref{prop:omega}]
For a quantum group associated with a symmetric Cartan datum, the following
positivity properties hold.
\begin{enumerate}
\item \textbf{(Comultiplication)} For any
$\dot{b}\in\dot{\mathbf{B}}$,
\[
\dot{\Delta}(\dot{b})
\in \mathbb{N}[v,v^{-1}]
[\dot{\mathbf{B}}\hat{\otimes}\dot{\mathbf{B}}].
\]
\item \textbf{(Bilinear form)} For any
$\dot{b},\dot{b}'\in\dot{\mathbf{B}}$,
\[
(\dot{b},\dot{b}')_{\dot{\mathbf{U}}}
\in\delta_{\dot{b},\dot{b}'}+v^{-1}\mathbb{N}[[v^{-1}]].
\]
\item \textbf{(Spherical multiplication)} If one of
$\dot{b},\dot{b}'\in\dot{\mathbf{B}}$ is spherical parabolic, then
\[
\dot{b}\dot{b}'
\in\mathbb{N}[v,v^{-1}][\dot{\mathbf{B}}].
\]
\item \textbf{(Transition)} For dominant weights
$\lambda_1,\lambda_2$,
\[
\mathbf{B}({^{\omega}\Lambda_{\lambda_1}}
\otimes\Lambda_{\lambda_2})
\subset
\mathbb{N}[v^{-1}]
[\mathbf{B}({^{\omega}\Lambda_{\lambda_1}})
\otimes\mathbf{B}(\Lambda_{\lambda_2})].
\]
\item \textbf{($\omega$-invariance)} The involution $\omega$ preserves
the canonical basis: $\omega(\dot{\mathbf{B}})=\dot{\mathbf{B}}$.
\end{enumerate}
\end{theorem}

The comultiplication and bilinear form assertions prove the corresponding parts of  Lusztig's conjecture for arbitrary symmetric Cartan datum. The multiplication assertion proves the other part whenever one factor is spherical parabolic. In particular, in finite type, every canonical basis element is spherical parabolic, and hence all parts of Lusztig's conjecture hold. Multiplication positivity in finite type was previously obtained by Webster \cite{Webster-2015}; the present argument gives a uniform algebraic proof and simultaneously establishes the
comultiplication and bilinear-form statements. Multiplication positivity for arbitrary elements of $\dot{\mathbf{B}}$ remains open in general symmetric Kac--Moody type.

We record several consequences for later use.

\begin{corollary}\label{cor:applications}
For a quantum group associated with a symmetric Cartan datum, the following
positivity properties hold.
\begin{enumerate}
\item For finite simply-laced type, Lusztig's positivity conjecture holds. The canonical basis $\dot{B}$ has the positivity properties with respect to the multiplication, comultiplication and bilinear form.
\item For every dominant weight $\lambda$ and
$\dot{b}\in\dot{\mathbf{B}}$,
\[
\dot{b}(\mathbf{B}(\Lambda_\lambda))
\subset\mathbb{N}[v,v^{-1}][\mathbf{B}(\Lambda_\lambda)].
\]
\item Let $\lambda_1,\ldots,\lambda_{m+n}$ be dominant, and let $\dot{b}\in\dot{\mathbf{B}}$ be spherical parabolic. Then the action of
$\dot{b}$ on the tensor product
\[
{^{\omega}\Lambda_{\lambda_1}}\otimes\cdots\otimes
{^{\omega}\Lambda_{\lambda_m}}\otimes
\Lambda_{\lambda_{m+1}}\otimes\cdots\otimes
\Lambda_{\lambda_{m+n}}
\]
has coefficients in $\mathbb{N}[v,v^{-1}]$ with respect to the canonical
basis. The same conclusion holds for arbitrary
$\dot{b}\in\dot{\mathbf{B}}$ if $m=0$ or $n=0$.
\end{enumerate}
\end{corollary}

Corollary \ref{cor:applications}(2) extends Lusztig's positivity theorem
\cite[Theorem 22.1.7]{Lusztig-1993} from the Chevalley generators to
arbitrary elements of $\dot{\mathbf{B}}$. Corollary
\ref{cor:applications}(3) includes, in particular, positivity for the
action of the Chevalley generators on canonical bases of tensor products. At $v=1$, these algebraic positivity results enter the study of total positivity on double flag varieties in the joint work of Xie and the second-named author \cite{HX}.

\subsection{The thickening realization}
The thickening construction is motivated by total positivity. Lusztig related the
totally nonnegative part of a flag variety to nonnegative linear combinations
of canonical basis elements \cite{Lusztig-1994}, and extended this framework
to Kac--Moody groups in \cite{Lusztig-2019}; see also
\cite{Bao-He-2021}. In \cite{Bao-He-2022}, Bao and the second-named author
studied twisted products of flag varieties by embedding them into a flag
variety for a larger Kac--Moody group. The thickening realization developed
here is an algebraic counterpart of that construction.

The enlarged quantum group $\tilde{\mathbf{U}}$ is associated with a framed
Cartan datum related to Nakajima's framed quiver construction
\cite{Nakajima-1998,Nakajima-2001,Li-2014,Fang-Lan-2025}. For a Demazure
submodule ${^\omega V_w(\lambda_1)}$, we construct the diagram
\begin{equation}\label{*}
\tag{$*$}\xymatrix@C=1cm@R=1cm{
\mathbf{B}((\mathbf{f}\theta_{\lambda_2}\mathbf{f})_{-w\lambda_1\odot\lambda_2})
\ar@<.5ex>[r] \ar@{^{(}->}[d]
&\mathbf{B}(M_{-w\lambda_1}\otimes \Lambda_{\lambda_2})
\ar@<.5ex>[l] \ar@{^{(}->}[d] \ar@{->>}[r]
&\mathbf{B}({^{\omega}V_w(\lambda_1)}\otimes
\Lambda_{\lambda_2})\sqcup\{0\} \ar@{^{(}->}[d]\\
(\mathbf{f}\theta_{\lambda_2}\mathbf{f})_{-w\lambda_1\odot\lambda_2}
\ar@<.5ex>[r]^-{\varphi_{-w\lambda_1,\lambda_2}}
&M_{-w\lambda_1}\otimes \Lambda_{\lambda_2}
\ar@<.5ex>[l]^-{\psi_{-w\lambda_1,\lambda_2}}
\ar@{->>}[r]^-{a_{-w\lambda_1}\otimes \mathrm{Id}}
&{^{\omega}V_w(\lambda_1)}\otimes \Lambda_{\lambda_2}.}
\end{equation}


The canonical basis of the lower-left space is obtained from a subset of the canonical basis of $\tilde{\mathbf{U}}^-$ acting on the highest weight vector of the ambient Verma module. The maps $\varphi_{-w\lambda_1,\lambda_2}$ and $\psi_{-w\lambda_1,\lambda_2}$ are inverse $\mathbf{U}$-module isomorphisms, while $a_{-w\lambda_1}\otimes\mathrm{Id}$ is a $\mathbf{U}^-$-linear quotient map. The top row records that these maps preserve the indicated canonical bases, with some basis elements mapping to zero under the quotient. Every canonical basis vector of ${^\omega\Lambda_{\lambda_1}}\otimes\Lambda_{\lambda_2}$ occurs at a Demazure stage. These realizations therefore apply to the full tensor product, and Lusztig's stable-limit construction relates them to $\dot{\mathbf{B}}$.

To obtain the structure-constant identities, we establish compatibility
of the thickening realization with multiplication and comultiplication.
For multiplication by a spherical parabolic element, we choose a
Demazure submodule for which the projection in \eqref{*} intertwines
the relevant parabolic action. A further thickening then identifies
the multiplication coefficients with those in the negative part of
an enlarged quantum group. For comultiplication, we compare the maps
between tensor products under thickening, obtaining the corresponding
coefficient identities up to explicit powers of $v$.


\subsection{Organization of the paper}
\begin{itemize}
\item Section \ref{sec:thick-module} constructs the thickening realization for the tensor products.
\item Section \ref{Thickening realization of canonical basis} proves the
compatibility of thickening realization and canonical bases.
\item Section \ref{Structure constant of multiplication} compares
spherical-parabolic multiplication in $\dot{\mathbf{U}}$ with
multiplication in the negative part of a further enlarged quantum group.
\item Section \ref{sec:comult} compares the comultiplication of
$\dot{\mathbf{U}}$ with that of $\tilde{\mathbf{U}}^-$.
\item Section \ref{Positivity of canonical basis of modified quantum group}
deduces the positivity and $\omega$-invariance results for
$\dot{\mathbf{B}}$.
\item Section \ref{Positivity of canonical basis of tensor product}
establishes the corresponding positivity statements for arbitrary tensor
products of simple lowest and highest weight modules.
\end{itemize}

{\bf Acknowledgements} JF is partially supported by the New Cornerstone Science Foundation through the New Cornerstone Investigator Program awarded to XH and the National Natural Science Foundation of China (Grant No. 12471030). XH is partially supported by the Hong Kong RGC grant 14300023 and by the New Cornerstone Science Foundation through the New Cornerstone Investigator Program. We thank Huanchen Bao, Yiqiang Li, George Lusztig and Weiqiang Wang for useful comments. JF would like to thank Jie Xiao and Yixin Lan for their long-term discussion and collaboration.

\section{Preliminary}\label{Preliminary}

\subsection{Cartan datum and root datum}\label{Cartan datum and root datum}

A symmetric Cartan datum $(I,\cdot)$ consists of a finite set $I$ and a symmetric bilinear form $\mathbb{Z}[I]\times \mathbb{Z}[I]\rightarrow \mathbb{Z}$ denoted by $(\nu,\nu')\mapsto \nu\cdot \nu'$ such that $i\cdot i=2$ for any $i\in I$, and $i\cdot j\leqslant 0$ for any $i\not=j$ in $I$.

A root datum $(Y,X,\langle\,,\,\rangle,\ldots)$ of type $(I,\cdot)$ consists of two finitely generated free abelian groups $Y,X$, a perfect bilinear pairing $\langle\,,\,\rangle:Y\times X\rightarrow \mathbb{Z}$ and two imbeddings $I\hookrightarrow Y,i\mapsto i_Y;\, I\hookrightarrow X,i\mapsto i_X$ such that $\langle i_Y,j_X\rangle=i\cdot j$ for any $i,j\in I$. The imbeddings $I\hookrightarrow Y$ and $I\hookrightarrow X$ induce group homomorphisms $\mathbb{Z}[I]\rightarrow Y,\nu\mapsto \nu_Y$ and $\mathbb{Z}[I]\rightarrow X,\nu\mapsto \nu_X$ respectively. 

The root datum is called $Y$-regular if the image of the imbedding $I\hookrightarrow Y$ is linearly independent in $Y$, and is called $X$-regular if the image of the imbedding $I\hookrightarrow X$ is linearly independent in $X$.

The simply connected root datum  of type $(I,\cdot)$ is $(\mathbb{Z}[I],\mathrm{Hom}_{\mathbb{Z}}(\mathbb{Z}[I],\mathbb{Z}),\langle\,,\,\rangle,\ldots)$, where $\langle\,,\,\rangle$ is the natural evaluation, $I\hookrightarrow \mathbb{Z}[I]$ and $I\hookrightarrow \mathrm{Hom}_{\mathbb{Z}}(\mathbb{Z}[I],\mathbb{Z})$ are defined by $i\mapsto i$ and $i\mapsto (j\mapsto i\cdot j)$ respectively. It is $Y$-regular.

Throughout this paper, we fix a symmetric Cartan datum $(I,\cdot)$, and denote by $(Y,X,\langle\,,\,\rangle,\ldots)$ the simply connected root datum of type $(I,\cdot)$.

\subsection{The algebra $\mathbf{f}$}\label{The algebra f}

We follow \cite[Chapter 1]{Lusztig-1993}.

Let $\mathbb{Q}(v)$ be the field of rational functions in $v$ with coefficients in $\mathbb{Q}$. We denote
$$[n]=\frac{v^n-v^{-n}}{v-v^{-1}},\ [m]!=\prod^m_{k=1}[k]\in \mathbb{Q}(v)\ \textrm{for any}\ n\in\mathbb{Z},m\in \mathbb{N}.$$

For the Cartan datum $(I,\cdot)$, let $\mathbf{f}$ be the $\mathbb{Q}(v)$-algebra defined in \cite[\S 1.2.5]{Lusztig-1993}. Equivalently, it is the $\mathbb{Q}(v)$-algebra with unit generated by $\theta_i$ for any $i\in I$, subject to the quantum Serre relations
$$\sum_{n=0}^{1-i\cdot j}(-1)^n\theta_i^{(n)}\theta_j\theta_i^{(1-i\cdot j-n)}=0\ \textrm{for any $i\not=j$ in $I$},$$
where $\theta_i^{(n)}=\theta^n_i/[n]!$ for any $i\in I$ and $n\in \mathbb{N}$.

The algebra $\mathbf{f}=\bigoplus_{\nu\in \mathbb{N}[I]}\mathbf{f}_\nu$ is $\mathbb{N}[I]$-graded such that $\theta_i\in \mathbf{f}_i$ for any $i\in I$. For any homogeneous $x\in \mathbf{f}_\nu$, we denote $|x|=\nu$. The tensor product $\mathbf{f}\otimes \mathbf{f}$ over $\mathbb{Q}(v)$ is an algebra with the twisted multiplication 
$$(x_1\otimes x_2)(y_1\otimes y_2)=v^{|x_2|\cdot |y_1|}x_1y_1\otimes x_2y_2\ \textrm{for any homogeneous}\ x_1,x_2,y_1,y_2\in \mathbf{f}.$$

The comultiplication of $\mathbf{f}$ is the algebra homomorphism $r:\mathbf{f}\rightarrow \mathbf{f}\otimes \mathbf{f}$ defined by $r(\theta_i)=\theta_i\otimes 1+1\otimes \theta_i$ for any $i\in I$.

\subsection{Quantum group}\label{Quantized enveloping algebra}

We refer to \cite[Chapter 3]{Lusztig-1993} for details.

For the root datum $(Y,X,\langle\,,\,\rangle,\ldots)$ of type $(I,\cdot)$, the quantum group $\mathbf{U}$ is the $\mathbb{Q}(v)$-algebra with unit generated by $E_i,F_i$ and $K_\mu$ for any $i\in I$ and $\mu\in Y$, subject to the relations in \cite[\S 3.1.1]{Lusztig-1993}.

Let $\mathbf{U}^0$ be the subalgebra of $\mathbf{U}$ generated by $K_{\mu}$ for any $\mu\in Y$, and let $\mathbf{U}^+,\mathbf{U}^-$ be the subalgebras of $\mathbf{U}$ generated by $E_i,F_i$ for any $i\in I$ respectively. Then there is an algebra automorphism $\omega:\mathbf{U}\rightarrow \mathbf{U}$ defined by $\omega(E_i)=F_i, \omega(F_i)=E_i$ and $\omega(K_{\mu})=K_{-\mu}$ for any $i\in I$ and $\mu\in Y$.
There are algebra isomorphisms $\mathbf{f}\rightarrow \mathbf{U}^+, x\mapsto x^+$ and $\mathbf{f}\rightarrow \mathbf{U}^-, x\mapsto x^-$
such that $E_i=\theta_i^+$ and $F_i=\theta_i^-$ for any $i\in I$. We have $\omega(x^+)=x^-$ and $\omega(x^-)=x^+$ for any $x\in \mathbf{f}$.

The algebra $\mathbf{U}=\bigoplus_{\nu\in \mathbb{Z}[I]}\mathbf{U}(\nu)$ is $\mathbb{Z}[I]$-graded such that $E_i\in \mathbf{U}(i),F_i\in \mathbf{U}(-i)$ and $K_\mu\in \mathbf{U}(0)$ for any $i\in I$ and $\mu\in Y$.

The tensor product $\mathbf{U}\otimes \mathbf{U}$ over $\mathbb{Q}(v)$ is an algebra with the multiplication
$$(u_1\otimes u_2)(u'_1\otimes u'_2)=u_1u'_1\otimes u_2u'_2\ \textrm{for any}\ u_1,u_2,u'_1,u'_2\in \mathbf{U}.$$

The comultiplication of $\mathbf{U}$ is the algebra homomorphism $\Delta:\mathbf{U}\rightarrow \mathbf{U}\otimes \mathbf{U}$ defined by $\Delta(E_i)=E_i\otimes 1+K_{i_Y}\otimes E_i, \Delta(F_i)=1\otimes F_i+F_i\otimes K_{-i_Y}$ and $\Delta(K_\mu)=K_\mu\otimes K_\mu$ for any $i\in I$ and $\mu\in Y$.

\subsection{Weight module}\label{U-module}

We refer to \cite[\S 3.4]{Lusztig-1993} for details.

A weight $\mathbf{U}$-module is a left $\mathbf{U}$-module over $\mathbb{Q}(v)$ whose underlying vector space has a direct sum decomposition into weight spaces. Throughout this paper, all $\mathbf{U}$-modules are weight modules, and all $\mathbf{U}$-module homomorphisms respect the weight space decompositions.

For any $\mathbf{U}$-module $M$, we define a new $\mathbf{U}$-module ${^{\omega}M}$. Its underlying vector space is $M$. For any $m \in M$, we denote by $^{\omega} m\in {^{\omega}M}$ the corresponding element. The module structure is defined by $u ({}^{\omega} m)={}^{\omega}(\omega(u) m)$ for any $u\in \mathbf{U}$ and $m\in M$. 

For any $\zeta\in X$, we denote by $M_{\zeta}$ the Verma module of $\mathbf{U}$ with the highest weight $\zeta$. Its underlying vector space is $\mathbf{f}$, and the module structure is defined by $E_i1=0, F_ix=\theta_ix, K_{\mu}x=v^{\langle\mu,\zeta-|x|_X\rangle}x$ for any $i\in I,\mu\in Y$ and homogeneous $x\in \mathbf{f}$. We denote by $v_{\zeta}\in M_{\zeta}$ the highest weight vector $1\in \mathbf{f}$. 

Let $X^+=\{\l \in X; \langle i_Y,\lambda\rangle \in \mathbb{N}\ \textrm{for any}\ i\in I\}$ be the set of dominant weights. For any $\lambda\in X^+$, the simple integrable $\mathbf{U}$-module with the highest weight $\lambda$ is  
$$\Lambda_\lambda=M_{\lambda}/\sum_{i\in I}\mathbf{U}^-F_i^{\langle i_Y,\lambda\rangle+1}v_{\lambda}.$$
We denote by $\eta_{\lambda}\in \Lambda_{\lambda}$ the highest weight vector which is the image of $v_{\lambda}\in M_{\lambda}$ under the natural projection $M_{\lambda}\rightarrow \Lambda_{\lambda}$.

The simple integrable $\mathbf{U}$-module with the lowest weight $-\lambda$ is ${^{\omega}\Lambda_{\lambda}}$. We denote by $\xi_{-\lambda}\in {^{\omega}\Lambda_{\lambda}}$ the lowest weight vector ${^{\omega}\eta_{\lambda}}$.

\subsection{Modified quantum group}
We refer to \cite[Chapter 23]{Lusztig-1993} for details.

For the root datum $(Y,X,\langle\,,\,\rangle,\ldots)$ of type $(I,\cdot)$, the modified quantum group is $\dot{\mathbf{U}}=\bigoplus_{\gamma,\zeta\in X}{_{\gamma}\mathbf{U}_{\zeta}}$, where
$${_{\gamma}\mathbf{U}_{\zeta}}=\mathbf{U}/(\sum_{\mu\in Y}(K_{\mu}-v^{\langle \mu,\gamma\rangle})\mathbf{U}+\sum_{\mu\in Y}\mathbf{U}(K_{\mu}-v^{\langle \mu,\zeta\rangle})).$$
Let ${_{\gamma}\pi_{\zeta}}:\mathbf{U}\rightarrow {_{\gamma}\mathbf{U}_{\zeta}}$ be the natural projection for any $\gamma,\zeta\in X$, and let $1_{\zeta}={_{\zeta}\pi_{\zeta}}(1)$ for any $\zeta\in X$. Then $\dot{\mathbf{U}}$ is a free $(\mathbf{U}^+\otimes (\mathbf{U}^-)^{\mathrm{opp}})$-module with a basis $\{1_{\zeta};\zeta\in X\}$. Comparing with the triangular decomposition $\mathbf{U}=\mathbf{U}^+\otimes \mathbf{U}^0\otimes \mathbf{U}^-$ of the quantum group, $\dot{\mathbf{U}}$ is a modified form of $\mathbf{U}$ in which $\mathbf{U}^0$ is replaced by $\bigoplus_{\zeta\in X}\mathbb{Q}(v)1_{\zeta}$.

The multiplication of $\mathbf{U}$ induces a well-defined multiplication on $\dot{\mathbf{U}}$ such that 
$${_{\gamma}\pi_{\zeta}}(u)\cdot {_{\gamma'}\pi_{\zeta'}}(u')=\delta_{\zeta,\gamma'}\cdot {_{\gamma}\pi_{\zeta'}}(uu')$$ 
for any $\gamma,\gamma',\zeta,\zeta'\in X,\nu,\nu'\in \mathbb{Z}[I]$ and $u,u'\in \mathbf{U}$ such that $\gamma-\zeta=\nu_X,\gamma'-\zeta'=\nu'_X$ and $u\in \mathbf{U}(\nu),u'\in \mathbf{U}(\nu')$.

The comultiplication $\Delta:\mathbf{U}\rightarrow \mathbf{U}\otimes \mathbf{U}$ of $\mathbf{U}$ induces a well-defined linear map ${_{\gamma',\gamma''}\dot{\Delta}_{\zeta',\zeta''}}:{_{\gamma'+\gamma''}\mathbf{U}_{\zeta'+\zeta''}}\rightarrow {_{\gamma'}\mathbf{U}_{\zeta'}}\otimes {_{\gamma''}\mathbf{U}_{\zeta''}}$ for any $\gamma',\gamma'',\zeta',\zeta''\in X$ such that 
$${_{\gamma',\gamma''}\dot{\Delta}_{\zeta',\zeta''}}({_{\gamma'+\gamma''}\pi_{\zeta'+\zeta''}}(x))=({_{\gamma'}\pi_{\zeta'}}\otimes {_{\gamma''}\pi_{\zeta''}})(\Delta(x))\ \textrm{for any}\ x\in \mathbf{U}.$$
The comultiplication of $\dot{\mathbf{U}}$ is the collection $\{{_{\gamma',\gamma''}\dot{\Delta}_{\zeta',\zeta''}};\gamma',\gamma'',\zeta',\zeta''\in X\}$.

A {\it unital} $\dot{\mathbf{U}}$-module $M$ is a left $\dot{\mathbf{U}}$-module over $\mathbb{Q}(v)$ such that for any $m\in M$, we have $1_{\zeta}m=0$ for all but finitely many $\zeta\in X$, and $\sum_{\zeta\in X}1_{\zeta}m=m$. Then the category of unital $\dot{\mathbf{U}}$-modules is equivalent to the category of weight $\mathbf{U}$-modules (see \cite[\S 23.1.4]{Lusztig-1993}). For this reason, $\dot{\mathbf{U}}$ is an algebra more appropriate than $\mathbf{U}$ for the study of weight modules.

\subsection{Canonical basis}\label{Canonical basis}

We refer to \cite[Chapter 14]{Lusztig-1993} for details.

Let $\mathcal{A}=\mathbb{Z}[v,v^{-1}]\subset \mathbb{Q}(v)$ be the ring of Laurent polynomials with coefficients in $\mathbb{Z}$. The integral form ${_{\mathcal{A}}\mathbf{f}}$ of $\mathbf{f}$ is the $\mathcal{A}$-subalgebra generated by $\theta_i^{(n)}$ for any $i\in I$ and $n\in \mathbb{N}$.

The bar-involution on $\mathbf{f}$ is the $\mathbb{Q}$-algebra involution $\mathbf{f}\rightarrow \mathbf{f},x\mapsto \overline{x}$ defined by $\overline{\theta_i}=\theta_i, \overline{v^n}=v^{-n}$ for any $i\in I$ and $n\in \mathbb{Z}$. 

By \cite[Proposition 1.2.3]{Lusztig-1993}, there is a unique non-degenerate symmetric bilinear form $(\,,\,)_{\mathbf{f}}:\mathbf{f}\times \mathbf{f}\rightarrow \mathbb{Q}(v)$ such that $(1,1)_{\mathbf{f}}=1,(\theta_i,\theta_j)_{\mathbf{f}}=\delta_{i,j}(1-v^{-2})^{-1}$ for any $i,j\in I$, and $(x_1x_2,x)_{\mathbf{f}}=(x_1\otimes x_2,r(x))_{\mathbf{f}\otimes \mathbf{f}}$ for any $x_1,x_2,x\in \mathbf{f}$, where the bilinear form on $\mathbf{f}\otimes \mathbf{f}$ is given by $(x_1\otimes x_2,x'_1\otimes x'_2)_{\mathbf{f}\otimes \mathbf{f}}=(x_1,x'_1)_{\mathbf{f}}(x_2,x'_2)_{\mathbf{f}}$ for any $x_1,x_2,x'_1,x'_2\in \mathbf{f}$.

The canonical basis $\mathbf{B}$ of $\mathbf{f}$ is defined in \cite[Definition 14.4.6]{Lusztig-1993}. It is a $\mathcal{A}$-basis of ${_{\mathcal{A}}\mathbf{f}}$ such that  $\overline{b}=b$ for any $b\in \mathbf{B}$, and $(b,b')_{\mathbf{f}}\in \delta_{b,b'}+v^{-1}\mathbb{Z}[[v^{-1}]]$ for any $b,b'\in \mathbf{B}$. The canonical basis $\mathbf{B}$ has the following positivity properties with respect to the multiplication, comultiplication and bilinear form.

\begin{theorem}[{\cite[Theorem 14.4.13]{Lusztig-1993}}]\label{14.4.13}
Let $b,b'\in \mathbf{B}$. Then we have
$$bb'\in \mathbb{N}[v,v^{-1}][\mathbf{B}],\ r(b)\in \mathbb{N}[v,v^{-1}][\mathbf{B}\otimes \mathbf{B}],\ (b,b')_{\mathbf{f}}\in \delta_{b,b'}+v^{-1}\mathbb{N}[[v^{-1}]].$$
\end{theorem}

Under the isomorphisms $\mathbf{f}\cong \mathbf{U}^-$ and $\mathbf{f}\cong \mathbf{U}^+$, the canonical basis $\mathbf{B}$ of $\mathbf{f}$ gives the canonical bases $\{b^-;b\in \mathbf{B}\}$ of $\mathbf{U}^-$ and $\{b^+;b\in \mathbf{B}\}$ of $\mathbf{U}^+$.

Let $\lambda\in X^+$. The integral form of $\Lambda_{\lambda}$ is ${_{\mathcal{A}}\Lambda_{\lambda}}=\{x^-\eta_{\lambda};x\in {_{\mathcal{A}}\mathbf{f}}\}$. 

The bar-involution on $\mathbf{U}$ is the $\mathbb{Q}$-algebra involution $\mathbf{U}\rightarrow \mathbf{U},u\mapsto \overline{u}$ defined by $\overline{E_i}=E_i, \overline{F_i}=F_i, \overline{K_{\mu}}=K_{-\mu}$ and $\overline{v^n}=v^{-n}$ for any $i\in I, \mu\in Y$ and $n\in \mathbb{Z}$. The bar-involution on $\Lambda_{\lambda}$ is the $\mathbb{Q}$-linear involution $\Lambda_{\lambda}\rightarrow \Lambda_{\lambda},m\mapsto \overline{m}$ defined by $\overline{u\eta_{\lambda}}=\overline{u}\eta_{\lambda}$ for any $u\in \mathbf{U}$. 

By \cite[Proposition 19.1.2]{Lusztig-1993}, there is a
unique symmetric bilinear form $(\,,\,)_{\Lambda_{\lambda}}:\Lambda_{\lambda}\times \Lambda_{\lambda}\rightarrow \mathbb{Q}(v)$ such that $(\eta_{\lambda},\eta_{\lambda})_{\Lambda_{\lambda}}=1$, and $(um,m')_{\Lambda_{\lambda}}\!\!=\!(m,\rho(u)m')_{\Lambda_{\lambda}}\!$ for any $u\in \mathbf{U}$ and $m,m'\in\Lambda_{\lambda}$, where $\rho:\mathbf{U}\rightarrow \mathbf{U}^{\mathrm{opp}}$ is the algebra isomorphism defined by $\rho(E_i)=vK_{i_Y}F_i,\rho(F_i)=vK_{-i_Y}E_i$ and $\rho(K_{\mu})=K_{\mu}$ for any $i\in I$ and $\mu\in Y$.

The canonical basis $\mathbf{B}(\Lambda_{\lambda})$ of $\Lambda_{\lambda}$ is defined in \cite[Definition 14.4.12]{Lusztig-1993}. By \cite[Theorem 14.4.11]{Lusztig-1993}, $\mathbf{B}(\Lambda_{\lambda})=\{b^-\eta_{\lambda}; b\in \mathbf{B}\}\setminus \{0\}$. It is a $\mathcal{A}$-basis of ${_{\mathcal{A}}\Lambda_{\lambda}}$ such that $\overline{b}=b$ for any $b\in \mathbf{B}(\Lambda_{\lambda})$, and $(b,b')_{\Lambda_{\lambda}}\in \delta_{b,b'}+v^{-1}\mathbb{Z}[v^{-1}]$ for any $b,b'\in \mathbf{B}(\Lambda_{\lambda})$. The canonical basis $\mathbf{B}(\Lambda_{\lambda})$ has the following positivity properties with respect to the actions by Chevalley generators.

\begin{theorem}[{\cite[Theorem 22.1.7]{Lusztig-1993}}]\label{22.1.7}
Let $\lambda\in X^+$ and $i\in I$. Then we have
$$F_i(\mathbf{B}(\Lambda_{\lambda})), E_i(\mathbf{B}(\Lambda_{\lambda}))\subset \mathbb{N}[v,v^{-1}][\mathbf{B}(\Lambda_{\lambda})].$$
\end{theorem}

Similarly, the simple lowest weight module ${^{\omega}\Lambda_{\lambda}}$ has the integral form, bar-involution, symmetric bilinear form and canonical basis 
$\mathbf{B}({^{\omega}\Lambda_{\lambda}})=\{b^+\xi_{-\lambda}; b\in \mathbf{B}\}\setminus \{0\}$ which satisfy similar properties. 

\subsection{Based module}\label{Based module}

A based $\mathbf{U}$-module $(M,B)$ is an integrable $\mathbf{U}$-module $M$ with a basis $B$ satisfying the conditions (a)-(d) in \cite[\S 27.1.2]{Lusztig-1993} (also see \cite[\S 2.1]{Bao-Wang-2016}). One condition is that the $\mathbb{Q}$-linear involution $M\rightarrow M,m\mapsto \overline{m}$ defined by $\overline{b}=b,\overline{v^n}=v^{-n}$ for any $b\in B$ and $n\in \mathbb{Z}$, is compatible with the bar-involution and the action of $\mathbf{U}$, that is, 
\begin{equation}\label{27.1.2c}
\overline{um}=\bar{u}\bar{m}\ \textrm{for any}\ u\in \mathbf{U}, m\in M.   
\end{equation}

A homomorphism of based $\mathbf{U}$-modules from $(M,B)$ to $(M',B')$ is a $\mathbf{U}$-module homomorphism $f:M\rightarrow M'$ such that $f(B)\subset B'\sqcup\{0\}$ and $B\cap \mathrm{ker}\,f$ is a basis of $\mathrm{ker}\,f$.

For any $\lambda\in X^+$, $(\Lambda_{\lambda},\mathbf{B}(\Lambda_{\lambda}))$ and $({^{\omega}\Lambda_{\lambda}},\mathbf{B}({^{\omega}\Lambda_{\lambda}}))$ are based $\mathbf{U}$-modules (see \cite[\S  27.1.4]{Lusztig-1993}).

\subsection{Canonical basis of tensor product}\label{Canonical basis of tensor product}

We refer to \cite[Chapter 24]{Lusztig-1993} for details.

For any $\mathbf{U}$-modules $M_1$ and $M_2$, the tensor product $M_1\otimes M_2$ over $\mathbb{Q}(v)$ is naturally a $(\mathbf{U}\otimes \mathbf{U})$-module. We regard it as a $\mathbf{U}$-module via the comultiplication $\Delta:\mathbf{U}\rightarrow \mathbf{U}\otimes \mathbf{U}$.

For subsets $B_1\subset M_1,B_2\subset M_2$, we denote $B_1 \otimes B_2:=\{b_1 \otimes b_2; b_1 \in B_1, b_2 \in B_2\}$. If the elements $b_1\diamondsuit b_2\in M_1\otimes M_2$ have been defined for any $b_1\in B_1$ and $b_2\in B_2$, then we simply denote $B_1\diamondsuit B_2:=\{b_1\diamondsuit b_2;b_1\in B_1,b_2\in B_2\}$.
 
Let $\lambda_1,\lambda_2\in X^+$. The tensor product ${^{\omega}\Lambda_{\lambda_1}}\otimes \Lambda_{\lambda_2}$ has a basis $\mathbf{B}({^{\omega}\Lambda_{\lambda_1}})\otimes \mathbf{B}(\Lambda_{\lambda_2})$. The bar-involution on ${^{\omega}\Lambda_{\lambda_1}}\otimes \Lambda_{\lambda_2}$ is defined by
$\overline{m_1\otimes m_2}=\overline{m_1}\otimes \overline{m_2}$ for any $m_1\in {^{\omega}\Lambda_{\lambda_1}}$ and $m_2\in \Lambda_{\lambda_2}$. In general $\Delta(\overline{u})\not=\overline{\Delta(u)}$ for $u\in \mathbf{U}$, and so this bar-involution do not satisfy \eqref{27.1.2c}. Hence $({^{\omega}\Lambda_{\lambda_1}}\otimes \Lambda_{\lambda_2},\mathbf{B}({^{\omega}\Lambda_{\lambda_1}})\otimes \mathbf{B}(\Lambda_{\lambda_2}))$ is not a based $\mathbf{U}$-module.

Lusztig defined a new $\mathbb{Q}$-linear involution on ${^{\omega}\Lambda_{\lambda_1}}\otimes \Lambda_{\lambda_2}$ satisfying \eqref{27.1.2c}. The quasi-$\mathcal{R}$-matrix $\Theta$ defined in \cite[\S 4.1]{Lusztig-1993} induces a linear transformation on the tensor product ${^{\omega}\Lambda_{\lambda_1}}\otimes \Lambda_{\lambda_2}$ (see \cite[\S 24.1.1]{Lusztig-1993}). More precisely,
$$\Theta(m_1\otimes m_2)=\sum_{\nu\in \mathbb{N}[I]}(-v)^{\mathrm{tr}\,\nu}\sum_{b\in B_{\nu}}b^-m_1\otimes b^{*+}m_2\ \textrm{for any}\ m_1\in {^{\omega}\Lambda_{\lambda_1}}, m_2\in \Lambda_{\lambda_2}.$$
where $\mathrm{tr}\,\nu=\sum_{i\in I}\nu_i$, $B_{\nu}$ is a basis of $\mathbf{f}_{\nu}$, and $\{b^*; b\in B_{\nu}\}$ is the dual basis to $B_{\nu}$ under $(\,,\,)_{\mathbf{f}}$ for any $\nu=\sum_{i\in I}\nu_ii\in \mathbb{N}[I]$. Note that there are always only finitely many non-zero terms in the summation. By \cite[Lemma 24.1.2]{Lusztig-1993}, we have $\Delta(\overline{u})\Theta(m_1\otimes m_2)=\Theta\overline{\Delta(u)}(m_1\otimes m_2)$ for any $u\in \mathbf{U}$, $m_1\in {^{\omega}\Lambda_{\lambda_1}}$ and $m_2\in \Lambda_{\lambda_2}$. The quasi-$\mathcal{R}$-matrix is an intertwiner between the compositions of the comultiplication and bar-involution taken in different orders, which is similar to the universal $\mathcal{R}$-matrix as an intertwiner between the comultiplication and its transpose. The $\Psi$-involution on the tensor product ${^{\omega}\Lambda_{\lambda_1}}\otimes \Lambda_{\lambda_2}$ is the $\mathbb{Q}$-linear transformation on it defined by
\begin{equation}\label{Phi-involution}
\Psi(m_1\otimes m_2)=\Theta(\overline{m_1}\otimes \overline{m_2})\ \textrm{for any}\ m_1\in {^{\omega}\Lambda_{\lambda_1}}, m_2\in \Lambda_{\lambda_2}.
\end{equation}
Then $\Psi(u(m_1\otimes m_2))=\overline{u}\Psi(m_1\otimes m_2)$ for any $u\in \mathbf{U}$ and $m_1\in {^{\omega}\Lambda_{\lambda_1}},m_2\in \Lambda_{\lambda_2}$, that is, the $\Psi$-involution satisfies \eqref{27.1.2c}.

The canonical basis of the tensor product ${^{\omega}\Lambda_{\lambda_1}}\otimes \Lambda_{\lambda_2}$ was established by Lusztig in \cite[Theorem 24.3.3]{Lusztig-1993}.

\begin{theorem}\label{24.3.3} 
Let $\l_1, \l_2 \in X^+$. Then for any $b_1\in \mathbf{B}({^{\omega}\Lambda_{\lambda_1}})$ and $b_2\in \mathbf{B}(\Lambda_{\lambda_2})$, there exists a unique $b_1 \diamondsuit b_2 \in \mathbb{Z}[v^{-1}][\mathbf{B}({^{\omega}\Lambda_{\lambda_1}})\otimes \mathbf{B}(\Lambda_{\lambda_2})]$ such that
\begin{align*}
\Psi(b_1 \diamondsuit b_2)=b_1 \diamondsuit b_2 \textrm{ and } b_1 \diamondsuit b_2-b_1 \otimes b_2 \in v^{-1}\mathbb{Z}[v^{-1}][\mathbf{B}({^{\omega}\Lambda_{\lambda_1}})\otimes \mathbf{B}(\Lambda_{\lambda_2})].
\end{align*}
Moreover, the set $\mathbf{B}({^{\omega}\Lambda_{\lambda_1}}\otimes \Lambda_{\lambda_2}):=\mathbf{B}({^{\omega}\Lambda_{\lambda_1}}) \diamondsuit \mathbf{B}(\Lambda_{\lambda_2})$ is a basis of ${^{\omega}\Lambda_{\lambda_1}}\otimes \Lambda_{\lambda_2}$. The transition matrix from $\mathbf{B}(^{\omega}\Lambda_{\lambda_1})\otimes  \mathbf{B}(\Lambda_{\lambda_2})$ to $\mathbf{B}({^{\omega}\Lambda_{\lambda_1}}\otimes \Lambda_{\lambda_2})$ is unitriangular with off-diagonal entries in $v^{-1}\mathbb{Z}[v^{-1}]$.
\end{theorem}

The canonical basis of the tensor product ${^{\omega}\Lambda_{\lambda_1}}\otimes \Lambda_{\lambda_2}$ is $\mathbf{B}({^{\omega}\Lambda_{\lambda_1}}\otimes \Lambda_{\lambda_2})$. Moreover, $({^{\omega}\Lambda_{\lambda_1}}\otimes \Lambda_{\lambda_2},\mathbf{B}({^{\omega}\Lambda_{\lambda_1}}\otimes \Lambda_{\lambda_2}))$ is a based $\mathbf{U}$-module.

\subsection{Canonical basis of modified quantum group}

We refer to \cite[Chapter 25]{Lusztig-1993} for details.

The integral form ${_{\mathcal{A}}\dot{\mathbf{U}}}$ of $\dot{\mathbf{U}}$ is the $\mathcal{A}$-subalgebra generated by $E_i^{(n)}1_{\zeta},F_i^{(n)}1_{\zeta}$ for any $i\in I,n\in \mathbb{N}$ and $\zeta\in X$.

The canonical bases $\mathbf{B}({^{\omega}\Lambda_{\lambda_1+\lambda}}\otimes \Lambda_{\lambda_2+\lambda})$ of ${^{\omega}\Lambda_{\lambda_1+\lambda}}\otimes \Lambda_{\lambda_2+\lambda}$ for any $\lambda_1,\lambda_2,\lambda\in X^+$ have very nice stability property when $\langle i_Y,\lambda\rangle\rightarrow \infty$ for any $i\in I$ in the sense of \cite[\S  25.1]{Lusztig-1993}. The canonical basis of the modified quantum group $\dot{\mathbf{U}}$ was established by Lusztig in \cite[Theorem 25.2.1]{Lusztig-1993} by gluing these canonical bases of tensor products.

\begin{theorem}\label{25.2.1}
Let $\zeta\in X$ and $b_1,b_2\in \mathbf{B}$. Then there exists a unique $b_1\diamondsuit_{\zeta}b_2\in {_{\mathcal{A}}\dot{\mathbf{U}}1_{\zeta}}$ such that
$$(b_1\diamondsuit_{\zeta}b_2)(\xi_{-\lambda_1}\otimes \eta_{\lambda_2})=b_1^+\xi_{-\lambda_1}\diamondsuit b_2^-\eta_{\lambda_2}\in \mathbf{B}({^{\omega}\Lambda_{\lambda_1}}\otimes \Lambda_{\lambda_2})$$
for any $\lambda_1,\lambda_2\in X^+$ such that $\zeta=\lambda_2-\lambda_1$ and $b_1^+\xi_{-\lambda_1}\neq0, b_2^-\eta_{\lambda_2}\neq0$.
Moreover, the set $\dot{\mathbf{B}}:=\{b_1\diamondsuit_{\zeta}b_2;\zeta\in X,b_1,b_2\in \mathbf{B}\}$ is a basis of $\dot{\mathbf{U}}$, and we have $(b_1\diamondsuit_{\zeta}b_2)(\xi_{-\lambda_1}\otimes \eta_{\lambda_2})=0$
for any $\lambda_1,\lambda_2\in X^+$ such that $\zeta\neq \lambda_2-\lambda_1$, $b_1^+\xi_{-\lambda_1}=0$, or $b_2^-\eta_{\lambda_2}=0$.
\end{theorem}

The canonical basis of the modified quantum group $\dot{\mathbf{U}}$ is $\dot{\mathbf{B}}$. 

By \cite[Proposition 25.2.6]{Lusztig-1993}, we have $b^-1_{\zeta}=1\diamondsuit_{\zeta}b,b^+1_{\zeta}=b\diamondsuit_{\zeta}1\in \dot{\mathbf{B}}$ for any $\zeta\in X$ and $b\in \mathbf{B}$. Hence $\dot{\mathbf{B}}$ is a generalization of $\mathbf{B}$.

\subsection{Canonical basis of Demazure module}

Let $\mathbf{U}(\mathfrak{b}^-)$ be the subalgebra of $\mathbf{U}$ generated by $F_i$ and $K_{\mu}$ for any $i\in I$ and $\mu\in Y$. Similarly, let $\mathbf{U}(\mathfrak{b}^+)$ be the subalgebra of $\mathbf{U}$ generated by $E_i$ and $K_{\mu}$ for any $i\in I$ and $\mu\in Y$.

For the Cartan datum $(I,\cdot)$, the Weyl group $W$ is generated by $s_i$ for any $i\in I$, subject to the relations in \cite[\S  2.1.1]{Lusztig-1993}. It acts on $Y$ and $X$ (see \cite[\S  2.2.6]{Lusztig-1993}).

Let $\lambda\in X^+$ and $w\in W$. By \cite[Proposition 5.2.7]{Lusztig-1993}, the $w\lambda$-weight space of $\Lambda_{\lambda}$ is one-dimensional. Let $\eta_{w\lambda}\in \mathbf{B}(\Lambda_{\lambda})$ be the unique canonical basis element of $\Lambda_{\lambda}$ of weight $w\lambda$. The {\it Demazure module} $V_{w}(\lambda)$ is the $\mathbf{U}(\mathfrak{b}^+)$-submodule of $\Lambda_{\lambda}$ generated by $\eta_{w\lambda}$. 

By \cite[Proposition 3.2.3]{Kashiwara-1993}, the set $\mathbf{B}(V_{w}(\lambda)):=\mathbf{B}(\L_\l)\cap V_w(\l)$ is a basis of $V_{w}(\lambda)$. We call it the {\it canonical basis} of $V_w(\l)$. 

By \cite[Corollary 3.2.2]{Kashiwara-1993}, we have $V_{w}(\lambda)\subset V_{w'}(\lambda)$ for any $w,w'\in W$ with $w\leqslant w'$ under the Bruhat order. Let $\ast$ be the Demazure product on $W$ (see \cite[\S 1.2]{He-Nie-2024}). Then $w, w' \le w \ast w'$ for any $w, w'\in W$. Thus $(W, \le)$ is a directed set, and $\{V_w(\lambda);w\in W\}$ forms a direct system. We have $\varinjlim_{w\in W} V_{w}(\lambda)=\Lambda_{\lambda}$, and so $\mathbf{B}(\Lambda_{\lambda})=\bigcup_{w\in W}\mathbf{B}(V_{w}(\lambda))$. This enables us to approximate the canonical basis of $\Lambda_{\lambda}$ by the canonical bases of its Demazure submodules.

Similarly, let $\xi_{-w\lambda}={^{\omega}\eta_{w\lambda}}\in \mathbf{B}({^{\omega}\Lambda_{\lambda}})$ be the unique canonical basis element of ${^{\omega}\Lambda_{\lambda}}$ of weight $-w\lambda$. The {\it Demazure module} ${^{\omega}V_{w}(\lambda)}$ is the $\mathbf{U}(\mathfrak{b}^-)$-submodule of ${^{\omega}\Lambda_{\lambda}}$ generated by $\xi_{-w\lambda}$. It has the {\it canonical basis} $\mathbf{B}({^{\omega}V_{w}(\lambda)}):=\mathbf{B}({}^\o \L_\l)\cap {}^\o V_w(\l)$. We have $\varinjlim_{w\in W}{^{\omega}V_{w}(\lambda)}={^{\omega}\Lambda_{\lambda}}$ and $\mathbf{B}({^{\omega}\Lambda_{\lambda}})=\bigcup_{w\in W}\mathbf{B}({^{\omega}V_{w}(\lambda)})$.

Recall that $M_{-w\lambda}$ is the Verma module of $\mathbf{U}$ with the highest weight $-w\lambda$. As a $\mathbf{U}^-$-module, $M_{\zeta}$ is free and is generated by $v_{\zeta}$. There is a surjective $\mathbf{U}(\mathfrak{b}^-)$-module homomorphism $$a_{-w\lambda}: M_{-w \l} \to {^{\omega}V_{w}(\lambda)},x^-v_{-w\lambda}\mapsto x^-\xi_{-w\lambda}\ \textrm{for any}\ x\in \mathbf{f}.$$ 
We call it the {\it projection map}.

The underlying vector space of $M_{-w\lambda}$ is $\mathbf{f}$, thus the canonical basis $\mathbf{B}$ of $\mathbf{f}$ gives a canonical basis $\mathbf{B}(M_{-w\lambda}):=\{b^-v_{-w\lambda_1};b\in \mathbf{B}\}$ of $M_{-w\lambda}$. The following proposition follows from \cite[Lemma 8.2.1]{Kashiwara-1994}. It enables us to realize the canonical basis of ${^{\omega}V_{w}(\lambda)}$ by the canonical basis of $\mathbf{U}^-$.

\begin{proposition}\label{action map is a based map}
Let $\l \in X^+$ and $w \in W$. Then we have 
$$a_{-w\lambda}(\mathbf{B}(M_{-w\lambda}))\subset \mathbf{B}({^{\omega}V_{w}(\lambda)})\sqcup \{0\},$$ 
and the set $\mathbf{B}(M_{-w\lambda})\cap \mathrm{ker}\,a_{-w\lambda}$ is a basis of $\mathrm{ker}\,a_{-w\lambda}$.
\end{proposition}

\begin{proposition}\label{canonical basis of demazure otimes highest weight}
Let $\lambda_1,\lambda_2\in X^+,w\in W$. Then the subset $\mathbf{B}({^{\omega}V_{w}(\lambda_1)}\otimes \Lambda_{\lambda_2}):=\mathbf{B}({^{\omega}V_w(\lambda_1)}) \diamondsuit \mathbf{B}(\Lambda_{\lambda_2})$ of $\mathbf{B}({^{\omega}\Lambda_{\lambda_1}}\otimes \Lambda_{\lambda_2})$ is a basis of ${^{\omega}V_{w}(\lambda_1)}\otimes \Lambda_{\lambda_2}$. Moreover, we have $\mathbf{B}({^{\omega}V_{w}(\lambda_1)}\otimes \Lambda_{\lambda_2})=\mathbf{B}({^{\omega}\Lambda_{\lambda_1}}\otimes \Lambda_{\lambda_2})\cap ({^{\omega}V_{w}(\lambda_1)}\otimes \Lambda_{\lambda_2})$.
\end{proposition}
\begin{proof}
Note that the quasi-$\mathcal{R}$-matrix $\Theta$ and the $\Psi$-involution on the tensor product ${^{\omega}\Lambda_{\lambda_1}}\otimes \Lambda_{\lambda_2}$ preserve the subspace ${^{\omega}V_{w}(\lambda_1)}\otimes \Lambda_{\lambda_2}$. Indeed, for any $b_1\in \mathbf{B}({^{\omega}V_{w}(\lambda_1)})$ and $b_2\in \mathbf{B}(\Lambda_{\lambda_2})$, by \eqref{Phi-involution}, we have 
\begin{align*}
\Psi(b_1\otimes b_2)=\Theta(b_1\otimes b_2)=\sum_{\nu\in \mathbb{N}[I]}(-v)^{\mathrm{tr}\,\nu}\sum_{b\in B_{\nu}}b^-b_1\otimes {b^*}^+ b_2\in {^{\omega}V_{w}(\lambda_1)}\otimes \Lambda_{\lambda_2}.
\end{align*}
By the proof of Theorem \ref{24.3.3}, $b_1\diamondsuit b_2\in {^{\omega}V_{w}(\lambda_1)}\otimes \Lambda_{\lambda_2}$, and the transition matrix from the basis $\mathbf{B}({^{\omega}V_{w}(\lambda_1)})\otimes \mathbf{B}(\Lambda_{\lambda_2})$ to the set $\mathbf{B}({^{\omega}V_{w}(\lambda_1)}\otimes \Lambda_{\lambda_2})$ is unitriangular. So $\mathbf{B}({^{\omega}V_{w}(\lambda_1)}\otimes \Lambda_{\lambda_2})$ is also a basis of ${^{\omega}V_w(\lambda_1)}\otimes \Lambda_{\lambda_2}$, and the remaining statement follows from definition directly.
\end{proof}

We call $\mathbf{B}({^{\omega}V_{w}(\lambda_1)}\otimes \Lambda_{\lambda_2})$ the {\it canonical basis} of the tensor product ${^{\omega}V_{w}(\lambda_1)}\otimes \Lambda_{\lambda_2}$. 

\section{Thickening realization of tensor product}\label{sec:thick-module}

\subsection{Thickening product}\label{The thickening product}

For the symmetric Cartan datum $(I,\cdot)$, we define its {\it thickening Cartan datum} to be the symmetric Cartan datum $(\tilde{I},\cdot)$, where $\tilde{I}=I\sqcup I'$ is obtained from $I$ by adding a copy $I'=\{i';i\in I\}$, and the symmetric bilinear form $\mathbb{Z}[\tilde{I}]\times \mathbb{Z}[\tilde{I}]\rightarrow \mathbb{Z}$ is extended from that in $(I,\cdot)$ such that 
$$i\cdot j'=-\delta_{i,j},\ i'\cdot j'=2\delta_{i,j}\ \textrm{for any}\ i,j\in I.$$

For the simply connected root datum $(Y,X,\langle\,,\,\rangle,\ldots)$ of type $(I,\cdot)$, we define its {\it thickening root datum} to be the simply connected root datum $(\tilde{Y},\tilde{X},\langle\,,\,\rangle,\ldots)$ of type $(\tilde{I},\cdot)$. We have $\tilde{Y}=Y\oplus \mathbb{Z}[I'],\tilde{X}=X\oplus \mathrm{Hom}_{\mathbb{Z}}(\mathbb{Z}[I'],\mathbb{Z})$, the perfect bilinear pairing $\langle\,,\,\rangle:\tilde{Y}\times \tilde{X}\rightarrow \mathbb{Z}$ and the imbeddings $\tilde{I}\hookrightarrow \tilde{Y},\tilde{i}\mapsto \tilde{i}_{\tilde{Y}};\, \tilde{I}\hookrightarrow \tilde{X},\tilde{i}\mapsto \tilde{i}_{\tilde{X}}$ are extended from those in $(Y,X,\langle\,,\,\rangle,\ldots)$. It is $\tilde{Y}$-regular.

For any $\zeta\in X$ and $\lambda\in X^+$, we define $\zeta\odot\lambda\in \tilde{X}$ by
\begin{equation}\label{definition of odot}
\langle i_{\tilde{Y}},\zeta\odot\lambda\rangle=\langle i_Y,\zeta\rangle,\ \langle i'_{\tilde{Y}},\zeta\odot\lambda\rangle=\langle i_Y,\lambda\rangle\ \textrm{for any}\ i\in I.
\end{equation}
We call $\zeta\odot\lambda$ the {\it thickening product} of $\zeta$ and $\lambda$. This product ``thickens'' the weight $\zeta$ by incorporating the data of $\lambda$ into the new $I'$-components.

\subsection{Realization of tensor product}\label{ftheta_lambdaf}

In this subsection, we establish the thickening realization for the tensor product. In particular, we have the following maps 
$$\xymatrix@C=1.8cm{(\mathbf{f}\theta_{\lambda_2}\mathbf{f})_{-w\lambda_1\odot\lambda_2} \ar@<.5ex>[r]^{\varphi_{-w\lambda_1,\lambda_2}} &M_{-w\lambda_1}\otimes \Lambda_{\lambda_2} \ar@<.5ex>[l]^{\psi_{-w\lambda_1,\lambda_2}} \ar[r]^-{a_{-w\lambda_1}\otimes \mathrm{Id}} &{^{\omega}V_w(\lambda_1)}\otimes \Lambda_{\lambda_2}}$$
in the diagram \eqref{*}.

Let $\tilde{\mathbf{f}}$ and $\tilde{\mathbf{U}}$ be the algebra and the quantum group defined in \S \ref{The algebra f} and \S \ref{Quantized enveloping algebra} associated with the thickening data $(\tilde{I},\cdot)$ and $(\tilde{Y},\tilde{X},\langle\,,\,\rangle,\ldots)$ respectively. Then the natural imbedding $I\hookrightarrow \tilde{I}$ induces subalgebra imbeddings $\mathbf{f}\hookrightarrow\tilde{\mathbf{f}}$ and $\mathbf{U}\hookrightarrow\tilde{\mathbf{U}}$.

Let $\lambda\in X^+$. Since $i'\cdot j'=0$ and $\theta_{i'}\theta_{j'}=\theta_{j'}\theta_{i'}$ for any $i\not=j$ in $I$, the element 
$$\theta_{\lambda}=\prod_{i\in I}\theta_{i'}^{(\langle i_Y,\lambda\rangle)}\in \tilde{\mathbf{f}}$$
is independent of the order in the product. We define $\mathbf{f}\theta_{\lambda}\mathbf{f}$ to be the subspace of $\tilde{\mathbf{f}}$ spanned by $x\theta_{\lambda}y$ for any $x,y\in \mathbf{f}$. 

Let $\zeta\in X$ and $\lambda\in X^+$. Then we have the thickening product $\zeta\odot\lambda\in \tilde{X}$. Let $M_{\zeta\odot\lambda}$ be the Verma module of $\tilde{\mathbf{U}}$ with the highest weight $\zeta\odot\lambda$. As a $\tilde{\mathbf{U}}^-$-module, $M_{\zeta\odot\lambda}$ is free and is generated by the highest weight vector $v_{\zeta\odot\lambda}$. We define $$(\mathbf{f}\theta_{\lambda}\mathbf{f})_{\zeta\odot\lambda}:=\{z^-v_{\zeta\odot\lambda};z\in \mathbf{f}\theta_{\lambda}\mathbf{f}\}.$$
Then there is a natural identification $\mathbf{f}\theta_{\lambda}\mathbf{f}\cong (\mathbf{f}\theta_{\lambda}\mathbf{f})_{\zeta\odot\lambda},z\mapsto z^-v_{\zeta\odot\lambda}$. We regard $M_{\zeta\odot \lambda}$ as a $\mathbf{U}$-module via the imbedding $\mathbf{U}\hookrightarrow\tilde{\mathbf{U}}$. Similar to \cite[Lemma 3.1]{Fang-Lan-2025}, $(\mathbf{f}\theta_{\lambda}\mathbf{f})_{\zeta\odot\lambda}$ is a $\mathbf{U}$-submodule of $M_{\zeta\odot\lambda}$. 

Analogously to \cite[\S  3.2, (18)]{Li-2014}, we first construct a $\mathbf{U}$-module homomorphism $\psi_{\zeta,\lambda}:M_{\zeta}\otimes \Lambda_{\lambda}\rightarrow (\mathbf{f}\theta_{\lambda}\mathbf{f})_{\zeta\odot\lambda}$. Let $\mathbb{Q}(v)_{\lambda}$ be the one-dimensional $\mathbf{U}(\mathfrak{b^+})$-module with the basis $\{\rho_{\lambda}\}$ and the module structure defined by $E_i\rho_{\lambda}=0$ and $K_{\mu}\rho_{\lambda}=v^{\langle \mu,\lambda\rangle}\rho_{\lambda}$ for any $i\in I$ and $\mu\in Y$. Then $M_{\lambda}=\mathbf{U}\otimes_{\mathbf{U}(\mathfrak{b^+})}\mathbb{Q}(v)_{\lambda}$. It is easy to check that $M_{\zeta}\otimes \mathbb{Q}(v)_{\lambda}\rightarrow (\mathbf{f}\theta_{\lambda}\mathbf{f})_{\zeta\odot\lambda}$ defined by $x^-v_{\zeta}\otimes \rho_{\lambda}\mapsto \rho_{\lambda}\theta_{\lambda}^-x^-v_{\zeta\odot\lambda}$ for any $x\in \mathbf{f}$, is a $\mathbf{U}(\mathfrak{b^+})$-module homomorphism. By the tensor identity and Frobenius reciprocity 
\begin{align*}
M_{\zeta}\otimes M_{\lambda}&\cong \mathbf{U}\otimes_{\mathbf{U}(\mathfrak{b}^+)}(M_{\zeta}\otimes \mathbb{Q}(v)_{\lambda}),\\
\mathrm{Hom}_{\mathbf{U}}(M_{\zeta}\otimes M_{\lambda},(\mathbf{f}\theta_{\lambda}\mathbf{f})_{\zeta\odot\lambda})&\cong \mathrm{Hom}_{\mathbf{U}(\mathfrak{b^+})}(M_{\zeta}\otimes \mathbb{Q}(v)_{\lambda},(\mathbf{f}\theta_{\lambda}\mathbf{f})_{\zeta\odot\lambda}),
\end{align*}
we obtain a $\mathbf{U}$-module homomorphism $\psi_0:M_{\zeta}\otimes M_{\lambda}\rightarrow (\mathbf{f}\theta_{\lambda}\mathbf{f})_{\zeta\odot\lambda}$ such that 
$$\psi_0(x^-v_{\zeta}\otimes v_{\lambda})=\theta_{\lambda}^-x^-v_{\zeta\odot\lambda}\ \textrm{for any}\ x\in \mathbf{f}.$$
For any $i\in I$, by induction on $n\in \mathbb{N}$, we have 
$$\psi_0(x^-v_{\zeta}\otimes F_i^{(n)}v_{\lambda})=\sum_{k=0}^{n} (-1)^k v^{-k(\langle i_Y,\lambda\rangle+1-n)}F_i^{(n-k)}(\prod_{i\in I}F_{i'}^{(\langle i_Y,\lambda\rangle)})F_i^{(k)}x^-v_{\zeta\odot\lambda}$$
for any $x\in \mathbf{f}$. Since $j'\cdot i=0$ and $F_{j'}F_i=F_iF_{j'}$ for any $j\neq i$ in $I$, taking $n=\langle i_Y,\lambda\rangle+1$, the right-hand side becomes 
\begin{align*}
\sum_{k=0}^{\langle i_Y,\lambda\rangle+1} (-1)^k F_i^{(\langle i_Y,\lambda\rangle+1-k)}F_{i'}^{(\langle i_Y,\lambda\rangle)}F_i^{(k)}(\prod_{j\in I;j\not=i}F_{j'}^{(\langle j_Y
,\lambda\rangle)})x^-v_{\zeta\odot\lambda}=0 
\end{align*}
by the higher order quantum Serre relations in \cite[Proposition 7.1.5]{Lusztig-1993}. Moreover, by induction on $\mathrm{tr}\,|y|\in \mathbb{N}$, we have 
$$\psi_0(x^-v_{\zeta}\otimes y^-F_i^{(\langle i_Y,\lambda\rangle+1)}v_{\lambda})=0\ \textrm{for any homogeneous}\ x,y\in \mathbf{f}.$$
So $\psi_0$ induces a $\mathbf{U}$-module homomorphism $\psi_{\zeta,\lambda}:M_{\zeta}\otimes \Lambda_{\lambda}\rightarrow (\mathbf{f}\theta_{\lambda}\mathbf{f})_{\zeta\odot\lambda}$ such that 
\begin{equation}\label{im psi}
\psi_{\zeta,\lambda}(x^-v_{\zeta}\otimes \eta_{\lambda})=\theta_{\lambda}^-x^-v_{\zeta\odot\lambda}\ \textrm{for any}\ x\in \mathbf{f}. 
\end{equation}

Next we introduce the map $\varphi_{\zeta,\lambda}:(\mathbf{f}\theta_{\lambda}\mathbf{f})_{\zeta\odot\lambda}\rightarrow M_{\zeta}\otimes \Lambda_{\lambda}$.

\begin{lemma}\label{linear map varphi}
Let $\zeta\in X$ and $\lambda\in X^+$. Then there exists a $\mathbf{U}^-$-module homomorphism $\varphi_{\zeta,\lambda}:(\mathbf{f}\theta_{\lambda}\mathbf{f})_{\zeta\odot\lambda}\rightarrow M_{\zeta}\otimes \Lambda_{\lambda}$ such that
\begin{equation}\label{im varphi}
\varphi_{\zeta,\lambda}(x^-\theta_{\lambda}^-y^-v_{\zeta\odot\lambda})=\sum v^{|x_2|\cdot |\theta_{\lambda}|}x_2^-y^-v_{\zeta}\otimes x_1^-\eta_{\lambda}\ \textrm{for any}\ x,y\in \mathbf{f},
\end{equation}
where we write $r(x)=\sum x_1\otimes x_2$ and $x_1,x_2\in \mathbf{f}$ are homogeneous.
\end{lemma}
\begin{proof}
We identify $(\mathbf{f}\theta_{\lambda}\mathbf{f})_{\zeta\odot\lambda}\cong \mathbf{f}\theta_{\lambda}\mathbf{f}$ via $z^-v_{\zeta\odot\lambda}\mapsto z$ for any $z\in \mathbf{f}\theta_{\lambda}\mathbf{f}$, and define 
\begin{align*}
\varphi_0:(\mathbf{f}\theta_{\lambda}\mathbf{f})_{\zeta\odot\lambda}\cong \mathbf{f}\theta_{\lambda}\mathbf{f}\hookrightarrow \tilde{\mathbf{f}}\cong \tilde{\mathbf{U}}^-\xrightarrow{\Delta} \tilde{\mathbf{U}}\otimes \tilde{\mathbf{U}}\xrightarrow{\mathrm{Id}\otimes p} \tilde{\mathbf{U}}\otimes \bigoplus_{\nu\in \mathbb{N}[I]}\tilde{\mathbf{U}}(-\nu-|\theta_{\lambda}|),
\end{align*}
where $\Delta$ is the comultiplication of $\tilde{\mathbf{U}}$, $p$ is natural projection from $\tilde{\mathbf{U}}$ to the direct sum of its certain homogeneous components.

The vector space $(\mathbf{f}\theta_{\lambda}\mathbf{f})_{\zeta\odot\lambda}$ is spanned by $x^-\theta_{\lambda}^-y^-v_{\zeta\odot\lambda}$ for any $x,y\in \mathbf{f}$. We write $r(x)=\sum x_1\otimes x_2$ and $r(y)=\sum y_1\otimes y_2$ with homogeneous $x_1,x_2,y_1,y_2\in \mathbf{f}$. By \cite[\S 3.1.5 \& Lemma 1.2.11]{Lusztig-1993}, we have 
$$\varphi_0(x^-\theta_{\lambda}^-y^-v_{\zeta\odot\lambda})=\sum v^{-|x_1|\cdot |x_2|-|y_1|\cdot |y_2|}x_2^-y_2^-\otimes K_{-|x_2|_Y}x_1^-\theta_{\lambda}^-K_{-|y_2|_Y}y_1^-\in \mathbf{U}\otimes \tilde{\mathbf{U}},$$
and so $\mathrm{im}\,\varphi_0\subset \mathbf{U}\otimes \tilde{\mathbf{U}}$. By \eqref{definition of odot}, $0\odot\lambda\in \tilde{X}^+$ is dominant. Let $M_{0\odot\lambda}$ and $\Lambda_{0\odot\lambda}$ be the Verma module and the simple highest weight module of $\tilde{\mathbf{U}}$ respectively. Then we define $\kappa:\mathbf{U}\otimes \tilde{\mathbf{U}}\rightarrow M_{\zeta}\otimes \Lambda_{0\odot\lambda},u\otimes \tilde{u}\mapsto uv_{\zeta}\otimes \tilde{u}\eta_{0\odot\lambda}$. By definition, $\kappa\varphi_0:(\mathbf{f}\theta_{\lambda}\mathbf{f})_{\zeta\odot\lambda}\rightarrow M_{\zeta}\otimes \Lambda_{0\odot\lambda}$ is a $\mathbf{U}^-$-module homomorphism. Since $\langle i,0\odot\lambda\rangle=0$ for any $i\in I$ and $y_1^-\eta_{0\odot\lambda}=0$ whenever $|y_1|\not=0$, we have
\begin{align*}
\kappa\varphi_0(x^-\theta_{\lambda}^-y^-v_{\zeta\odot\lambda})=&\sum v^{-|x_1|\cdot |x_2|}x_2^-y^-v_{\zeta}\otimes K_{-|x_2|_Y}x_1^-\theta_{\lambda}^-K_{-|y|_Y}\eta_{0\odot\lambda}\\
=&\sum v^{-|x_1|\cdot |x_2|}x_2^-y^-v_{\zeta}\otimes v^{|x_2|\cdot (|x_1|+|\theta_{\lambda}|)}x_1^-\theta_{\lambda}^-K_{-|x_2|_Y}K_{-|y|_Y}\eta_{0\odot\lambda}\\
=&\sum v^{|x_2|\cdot |\theta_{\lambda}|}x_2^-y^-v_{\zeta}\otimes x_1^-\theta_{\lambda}^-\eta_{0\odot\lambda}\in M_{\zeta}\otimes \pi_{0\odot\lambda}((\mathbf{f}\theta_{\lambda}\mathbf{f})_{0\odot\lambda}),
\end{align*}
and so $\mathrm{im}\,(\kappa\varphi_0)\subset M_{\zeta}\otimes \pi_{0\odot\lambda}((\mathbf{f}\theta_{\lambda}\mathbf{f})_{0\odot\lambda})$, where $\pi_{0\odot\lambda}:M_{0\odot\lambda}\rightarrow \Lambda_{0\odot\lambda}$ is the natural projection. By \cite[Proposition 3.4]{Fang-Lan-2025}, there is a $\mathbf{U}$-module isomorphism $$\phi_{\lambda}:\Lambda_{\lambda}\rightarrow\pi_{0\odot\lambda}((\mathbf{f}\theta_{\lambda}\mathbf{f})_{0\odot\lambda}),x^-\eta_{\lambda}\mapsto x^-\theta_{\lambda}^-\eta_{0\odot\lambda}\ \textrm{for any}\ x\in \mathbf{f}.$$ 
Then $\varphi_{\zeta,\lambda}=(\mathrm{Id}\otimes \phi_{\lambda}^{-1})\kappa\varphi_0:(\mathbf{f}\theta_{\lambda}\mathbf{f})_{\zeta\odot\lambda}\rightarrow M_{\zeta}\otimes \Lambda_{\lambda}$ is the desired $\mathbf{U}^-$-module homomorphism.
\end{proof}

The following lemma establishes the thickening realization for the tensor product of a Verma module and a simple highest weight module.

\begin{lemma}\label{isomorphism varphi}
Let $\zeta\in X$ and $\lambda\in X^+$. Then $\varphi_{\zeta, \l}$ and $\psi_{\zeta,\l}$ are inverse to each other, and both of them are $\mathbf{U}$-module isomorphisms.
\end{lemma}
\begin{proof}
We simply denote $\varphi=\varphi_{\zeta,\lambda}$ and $\psi=\psi_{\zeta,\lambda}$. 

The vector space $(\mathbf{f}\theta_{\lambda}\mathbf{f})_{\zeta\odot\lambda}$ is spanned by $x^-\theta_{\lambda}^-y^-v_{\zeta\odot\lambda}$ for any homogeneous $x,y\in \mathbf{f}$. We prove  $\psi\varphi(x^-\theta_{\lambda}^-y^-v_{\zeta\odot\lambda})=x^-\theta_{\lambda}^-y^-v_{\zeta\odot\lambda}$ by induction on $\mathrm{tr}\,|x|\in \mathbb{N}$. If $\mathrm{tr}\,|x|=0$, then $x\in \mathbb{Q}(v)$. By \eqref{im varphi} and \eqref{im psi}, we have
$$\psi\varphi(\theta_{\lambda}^-y^-v_{\zeta\odot\lambda})=\psi(y^-v_{\zeta}\otimes \eta_{\lambda})=\theta_{\lambda}^-y^-v_{\zeta\odot\lambda}.$$ 
If $\mathrm{tr}\,|x|>0$, then it is enough to prove the statement for any $x=\theta_ix'$ with $i\in I$ and homogeneous $x'\in \mathbf{f}$. Since both $\varphi$ and $\psi$ commute with the actions of $F_i$, and $\mathrm{tr}\,|x'|<\mathrm{tr}\,|x|$, by the inductive hypothesis, we have 
\begin{align*}
\psi\varphi(x^-\theta_{\lambda}^-y^-v_{\zeta\odot\lambda})=&\psi\varphi(F_i{x'}^-\theta_{\lambda}^-y^-v_{\zeta\odot\lambda})=F_i\psi\varphi({x'}^-\theta_{\lambda}^-y^-v_{\zeta\odot\lambda})\\
=&F_i{x'}^-\theta_{\lambda}^-y^-v_{\zeta\odot\lambda}=x^-\theta_{\lambda}^-y^-v_{\zeta\odot\lambda}.
\end{align*}
Hence $\psi\varphi=\mathrm{Id}:(\mathbf{f}\theta_{\lambda}\mathbf{f})_{\zeta\odot\lambda}\rightarrow (\mathbf{f}\theta_{\lambda}\mathbf{f})_{\zeta\odot\lambda}$.

The vector space $M_{\zeta}\otimes \Lambda_{\lambda}$ is spanned by $x^-v_{\zeta}\otimes y^-\eta_{\lambda}$ for any homogeneous $x,y\in \mathbf{f}$. We prove  $\varphi\psi(x^-v_{\zeta}\otimes y^-\eta_{\lambda})=x^-v_{\zeta}\otimes y^-\eta_{\lambda}$ by induction on $\mathrm{tr}\,|y|\in \mathbb{N}$. If $\mathrm{tr}\,|y|=0$, then $y\in \mathbb{Q}(v)$. By \eqref{im psi} and \eqref{im varphi}, we have 
$$\varphi\psi(x^-v_{\zeta}\otimes \eta_{\lambda})=\varphi(\theta_{\lambda}^-x^-v_{\zeta\odot\lambda})=x^-v_{\zeta}\otimes \eta_{\lambda}.$$ 
If $\mathrm{tr}\,|y|>0$, then it is enough to prove the statement for any $y=\theta_iy'$ with $i\in I$ and homogeneous $y'\in \mathbf{f}$. Notice that 
\begin{align*}
F_i(x^-v_{\zeta}\otimes y'^-\eta_{\lambda})=&x^-v_{\zeta}\otimes F_iy'^-\eta_{\lambda}+F_ix^-v_{\zeta}\otimes K_{-i_Y}y'^-\eta_{\lambda}\\
=&x^-v_{\zeta}\otimes y^-\eta_{\lambda}+(\theta_ix)^-v_{\zeta}\otimes v^{\langle -i_Y,\lambda-|y'|_X\rangle}y'^-\eta_{\lambda}
\end{align*}
Since both $\varphi$ and $\psi$ commute with the actions of $F_i$, and $\mathrm{tr}\,|y'|<\mathrm{tr}\,|y|$, by the inductive hypothesis, we have 
\begin{align*}
&\varphi\psi(x^-v_{\zeta}\otimes y^-\eta_{\lambda})=\varphi\psi(F_i(x^-v_{\zeta}\otimes y'^-\eta_{\lambda}))-\varphi\psi((\theta_ix)^-v_{\zeta}\otimes v^{\langle -i_Y,\lambda-|y'|_X\rangle}y'^-\eta_{\lambda})\\
=&F_i(x^-v_{\zeta}\otimes y'^-\eta_{\lambda})-(\theta_ix)^-v_{\zeta}\otimes v^{\langle -i_Y,\lambda-|y'|_X\rangle}y'^-\eta_{\lambda}=x^-v_{\zeta}\otimes y^-\eta_{\lambda}.
\end{align*}
Hence $\varphi\psi=\mathrm{Id}:M_{\zeta}\otimes \Lambda_{\lambda}\rightarrow M_{\zeta}\otimes \Lambda_{\lambda}$.

Thus, $\varphi$ and $\psi$ are inverse to each other. Since $\psi$ is a $\mathbf{U}$-module homomorphism, both $\varphi$ and $\psi$ are $\mathbf{U}$-module isomorphisms.
\end{proof}

Let $\lambda_1,\lambda_2\in X^+$ and $w\in W$. The projection map $a_{-w\lambda_1}:M_{-w\lambda_1}\rightarrow {^{\omega}V_w(\lambda_1)}$ defined by $x^-v_{-w\lambda_1}\mapsto x^-\xi_{-w\lambda_1}$ for any $x\in \mathbf{f}$, is a surjective $\mathbf{U}(\mathfrak{b}^-)$-module homomorphism. We denote the surjective $\mathbf{U}(\mathfrak{b}^-)$-module homomorphism
$$\underline{\varphi_{-w\lambda_1,\lambda_2}}=(a_{-w\lambda_1}\otimes \mathrm{Id})\varphi_{-w\lambda_1,\lambda_2}:(\mathbf{f}\theta_{\lambda_2}\mathbf{f})_{-w\lambda_1\odot\lambda_2}\rightarrow {^{\omega}V_w(\lambda_1)}\otimes \Lambda_{\lambda_2}.$$ 
This is the thickening realization for the tensor product of a Demazure module and a simple highest weight module.

\section{Thickening realization of canonical basis of tensor product}\label{Thickening realization of canonical basis}

\subsection{Canonical basis of tensor product of Verma module and simple module}\label{Canonical basis of Verma otimes highest weight}

Let $\zeta\in X$ and $\lambda\in X^+$. The underlying vector space of the Verma module $M_{\zeta}$ is $\mathbf{f}$, thus the canonical basis $\mathbf{B}$ of $\mathbf{f}$ gives a canonical basis $\mathbf{B}(M_\zeta):=\{b^-v_{\zeta};b\in \mathbf{B}\}$ of $M_{\zeta}$. Note that $M_{\zeta}$ is not integrable, and $(M_{\zeta},\mathbf{B}(M_\zeta))$ does not satisfy the definition of the based module in \cite[\S  27.1.2]{Lusztig-1993} or \cite[\S  2.1]{Bao-Wang-2016}. Nevertheless, a canonical basis of the tensor product $M_{\zeta}\otimes \Lambda_{\lambda}$ still exists.

We use Lusztig's quasi-$\mathcal{R}$-matrix $\Theta$ and the same formula \eqref{Phi-involution} to define a $\Psi$-involution on $M_{\zeta}\otimes \Lambda_{\lambda}$, that is,
\begin{equation}\label{Theta and Psi}
\Psi(m_1\otimes m_2)=\sum_{\nu\in \mathbb{N}[I]}(-v)^{\mathrm{tr}\,\nu}\sum_{b\in B_{\nu}}b^-\overline{m_1}\otimes b^{*+}\overline{m_2}\ \textrm{for any} \ m_1\in M_{\zeta},m_2\in \Lambda_{\lambda}.
\end{equation}
Note that there are still only finitely many non-zero terms in the summation. 

The following lemma is essentially the same as \cite[Theorem 27.3.2]{Lusztig-1993} and \cite[Theorem 2.7]{Bao-Wang-2016}. Indeed, by the same proof as \cite[Proposition 2.4 \& Remark 2.5]{Bao-Wang-2016}, both $\Theta$ and $\Psi$ preserve the $\mathcal{A}$-submodule of $M_{\zeta}\otimes \Lambda_{\lambda}$ generated by $\mathbf{B}(M_\zeta)\otimes \mathbf{B}(\Lambda_{\lambda})$. Then the statement follows from the same proof as \cite[Theorem 27.3.2]{Lusztig-1993}.

\begin{lemma}\label{canonical basis of Verma otime highest weight}
Let $\zeta\in X$ and $\lambda\in X^+$. Then for any $b_1\in \mathbf{B}(M_\zeta)$ and $b_2\in \mathbf{B}(\Lambda_{\lambda})$, there exists a unique $b_1\diamondsuit b_2\in \mathbb{Z}[v^{-1}][\mathbf{B}(M_\zeta)\otimes \mathbf{B}(\Lambda_{\lambda})]$ such that
$$\Psi(b_1\diamondsuit b_2)=b_1\diamondsuit b_2\ \text{and}\ b_1\diamondsuit b_2-b_1\otimes b_2 \in v^{-1}\mathbb{Z}[v^{-1}][\mathbf{B}(M_\zeta)\otimes \mathbf{B}(\Lambda_{\lambda})].$$
Moreover, the set $\mathbf{B}(M_\zeta\otimes \Lambda_{\lambda}):=\mathbf{B}(M_\zeta)\diamondsuit \mathbf{B}(\Lambda_{\lambda})$ is a basis of $M_{\zeta}\otimes \Lambda_{\lambda}$. The transition matrix from $\mathbf{B}(M_\zeta)\otimes \mathbf{B}(\Lambda_{\lambda})$ to $\mathbf{B}(M_\zeta\otimes \Lambda_{\lambda})$ is unitriangular with off-diagonal entries in $v^{-1}\mathbb{Z}[v^{-1}]$.
\end{lemma}

We call $\mathbf{B}(M_\zeta \otimes \Lambda_{\lambda})$ the {\it canonical basis} of the tensor product $M_{\zeta}\otimes \Lambda_{\lambda}$.

\begin{lemma}\label{a otimes Id is a based map}
Let $\lambda_1,\lambda_2\in X^+,w\in W$, and let $a_{-w\lambda_1}: M_{-w\l_1} \to {^{\omega}V_w(\lambda_1)}$ be the projection map. Then we have  
$$(a_{-w\lambda_1}\otimes \mathrm{Id})(\mathbf{B}(M_{-w\lambda_1}\otimes \Lambda_{\lambda_2})) \subset \mathbf{B}({^{\omega}V_{w}(\lambda_1)}\otimes \Lambda_{\lambda_2})\sqcup \{0\},$$ 
and the set $\mathbf{B}(M_{-w\lambda_1}\otimes \Lambda_{\lambda_2})\cap \mathrm{ker}\,(a_{-w\lambda_1}\otimes \mathrm{Id})$ is a basis of $\mathrm{ker}\,(a_{-w\lambda_1}\otimes \mathrm{Id})$.
\end{lemma}
\begin{proof}
By Proposition \ref{action map is a based map}, we have $a_{-w\lambda_1}(\mathbf{B}(M_{-w\lambda_1}))\subset \mathbf{B}({^{\omega}V_w(\lambda_1)})\sqcup\{0\}$. Let $b_1 \in \mathbf{B}(M_{-w\lambda_1})$ and $b_2 \in \mathbf{B}(\L_{\l_2})$. We prove that
\begin{equation}\label{image a otimes Id}
(a_{-w\lambda_1}\otimes \mathrm{Id})(b_1\diamondsuit b_2)=\delta_{a_{-w\lambda_1}(b_1),0}\cdot (a_{-w\lambda_1}(b_1) \diamondsuit b_2).    
\end{equation}
Let $\Psi$ and $\Psi'$ be the $\Psi$-involutions on ${^{\omega}\Lambda_{\lambda_1}}\otimes \Lambda_{\lambda_2}$ and $M_{-w\lambda_1}\otimes \Lambda_{\lambda_2}$ respectively. By definition, we have $\Psi(a_{-w\lambda_1}\otimes \mathrm{Id})=(a_{-w\lambda_1}\otimes m\mathrm{Id})\Psi':M_{-w \l_1} \otimes \L_{\l_2}\rightarrow {^{\omega}\Lambda_{\lambda_1}}\otimes \Lambda_{\lambda_2}$. By Lemma \ref{canonical basis of Verma otime highest weight}, we have $\Psi'(b_1\diamondsuit b_2)=b_1\diamondsuit b_2\in \mathbb{Z}[v^{-1}][\mathbf{B}(M_{-w\lambda_1})\otimes \mathbf{B}(\Lambda_{\lambda_2})]$ and $b_1\diamondsuit b_2-b_1\otimes b_2\in v^{-1}\mathbb{Z}[v^{-1}][\mathbf{B}(M_{-w\lambda_1})\otimes \mathbf{B}(\Lambda_{\lambda_2})]$. Applying $a_{-w\lambda_1}\otimes \mathrm{Id}$, by Proposition \ref{action map is a based map}, we obtain
\begin{align*}
&\Psi(a_{-w\lambda_1}\otimes \mathrm{Id})(b_1\diamondsuit b_2)
=(a_{-w\lambda_1}\otimes \mathrm{Id})(b_1\diamondsuit b_2)\in \mathbb{Z}[v^{-1}][\mathbf{B}({^{\omega}\Lambda_{\lambda_1}})\otimes \mathbf{B}(\Lambda_{\lambda_2})],\\
&(a_{-w\lambda_1}\otimes \mathrm{Id})(b_1\diamondsuit b_2)-a_{-w\lambda_1}(b_1)\otimes b_2\in v^{-1}\mathbb{Z}[v^{-1}][\mathbf{B}({^{\omega}\Lambda_{\lambda_1}})\otimes \mathbf{B}(\Lambda_{\lambda_2})].
\end{align*}
If $a_{-w\lambda_1}(b_1) \neq 0$, then $a_{-w\lambda_1}(b_1) \in \mathbf{B}({^{\omega}\Lambda_{\lambda_1}})$, and $(a_{-w\lambda_1}\otimes \mathrm{Id})(b_1\diamondsuit b_2)$ satisfies the characterizing conditions in Theorem \ref{24.3.3}. Hence $$(a_{-w\lambda_1}\otimes \mathrm{Id})(b_1\diamondsuit b_2)=a_{-w\lambda_1}(b_1)\diamondsuit b_2.$$ 
Moreover, by Proposition \ref{canonical basis of demazure otimes highest weight}, $a_{-w\lambda_1}(b_1)\diamondsuit b_2\in \mathbf{B}({^{\omega}V_{w}(\lambda_1)}\otimes \Lambda_{\lambda_2})$. If $a_{-w\lambda_1}(b_1)=0$, then $(a_{-w\lambda_1}\otimes \mathrm{Id})(b_1\diamondsuit b_2)\in v^{-1}\mathbb{Z}[v^{-1}][\mathbf{B}({^{\omega}\Lambda_{\lambda_1}})\otimes \mathbf{B}(\Lambda_{\lambda_2})]$ is $\Psi$-invariant. By Theorem \ref{24.3.3}, $\mathbf{B}({^{\omega}\Lambda_{\lambda_1}})\otimes \mathbf{B}(\Lambda_{\lambda_2})\subset \mathbb{Z}[v^{-1}][\mathbf{B}({^{\omega}\Lambda_{\lambda_1}}\otimes\Lambda_{\lambda_2})]$, and $0$ is the unique $\Psi$-invariant element in $v^{-1}\mathbb{Z}[v^{-1}][\mathbf{B}({^{\omega}\Lambda_{\lambda_1}}\otimes\Lambda_{\lambda_2})]$. Hence $$(a_{-w\lambda_1}\otimes \mathrm{Id})(b_1\diamondsuit b_2)=0.$$
The remaining statement about $\mathrm{ker}\,(a_{-w\lambda_1}\otimes \mathrm{Id})$ follows from \eqref{image a otimes Id} and the fact that $\mathbf{B}(M_{-w\lambda_1})\cap \mathrm{ker}\,a_{-w\lambda_1}$ is a basis of $\mathrm{ker}\,a_{-w\lambda_1}$ given by Proposition \ref{action map is a based map}. 
\end{proof}

\subsection{Canonical basis of $(\mathbf{f}\theta_{\lambda}\mathbf{f})_{\zeta\odot\lambda}$}

Let $\tilde{\mathbf{B}}$ be the canonical basis of $\tilde{\mathbf{f}}$.

Let $\lambda\in X^+$, and let $\mathbf{f}\theta_{\lambda}\mathbf{f}$ be the subspace of $\tilde{\mathbf{f}}$ defined in \S \ref{ftheta_lambdaf}. By \cite[Proposition 4.3]{Fang-Lan-2025}, the set $\mathbf{B}(\mathbf{f}\theta_{\lambda}\mathbf{f}):=\tilde{\mathbf{B}}\cap (\mathbf{f}\theta_{\lambda}\mathbf{f})$ is a basis of $\mathbf{f}\theta_{\lambda}\mathbf{f}$. We call it the {\it canonical basis} of $\mathbf{f}\theta_{\lambda}\mathbf{f}$. 

Let $\mathbf{B}(\lambda)=\bigcap_{i\in I}(\mathbf{B}\setminus \mathbf{f}\theta_i^{\langle i_Y,\lambda\rangle+1})$. By \cite[Theorem 14.4.11]{Lusztig-1993}, we have
\begin{equation}\label{14.4.11}
\mathbf{B}(\lambda)=\{b\in \mathbf{B};b^-\eta_{\lambda}\neq0\}. 
\end{equation}
By \cite[\S 5.1, (14)]{Fang-Lan-2025}, the right multiplication by $\theta_{\lambda}$ gives a bijection
\begin{equation}\label{Fang-Lan-2025-5.3*}
\mathbf{B}(\lambda)\rightarrow \bigcap_{i\in I}(\tilde{\mathbf{B}}\setminus \tilde{\mathbf{f}}\theta_i)\cap \mathbf{B}(\mathbf{f}\theta_{\lambda}\mathbf{f}), b\mapsto b\theta_{\lambda}.
\end{equation}

Let $\zeta\in X$ and $\lambda\in X^+$. Then 
$\mathbf{B}((\mathbf{f}\theta_{\lambda}\mathbf{f})_{\zeta\odot\lambda}):=\{b^-v_{\zeta\odot\lambda};b\in \mathbf{B}(\mathbf{f}\theta_{\lambda}\mathbf{f})\}$
is a basis of $(\mathbf{f}\theta_{\lambda}\mathbf{f})_{\zeta\odot\lambda}$. We call it the {\it canonical basis} of $(\mathbf{f}\theta_{\lambda}\mathbf{f})_{\zeta\odot\lambda}$. 
\subsection{Involution and bilinear form}\label{sec:technical}

Let $\zeta\in X$ and $\lambda\in X^+$. In this subsection, we study the compatibility of the thickening realization $\varphi_{\zeta,\lambda}:(\mathbf{f}\theta_{\lambda}\mathbf{f})_{\zeta\odot\lambda}\rightarrow M_{\zeta}\otimes \Lambda_{\lambda}$ established in Lemma \ref{isomorphism varphi} with the involutions and symmetric bilinear forms on $(\mathbf{f}\theta_{\lambda}\mathbf{f})_{\zeta\odot\lambda}$ and $M_{\zeta}\otimes \Lambda_{\lambda}$. 

We first study the involutions.

On the one hand, the underlying vector space of the Verma module $M_{\zeta}$ is $\mathbf{f}$, thus the bar-involution on $\mathbf{f}$ gives a bar-involution on $M_{\zeta}$, that is, $\overline{x^-v_{\zeta}}=(\overline{x})^-v_{\zeta}$ for any $x\in \mathbf{f}$. We identify $M_{\zeta}$ with $\mathbf{U}/(\sum_{i\in I}\mathbf{U}E_i+\sum_{\mu\in Y}\mathbf{U}(K_{\mu}-v^{\langle \mu,\zeta\rangle}))$. Under this identification, the bar-involution on $M_{\zeta}$ is induced by the bar-involution on $\mathbf{U}$, and so $\overline{um}=\bar{u}\bar{m}$ for any $u\in \mathbf{U}$ and $m\in M_{\zeta}$. Similarly, the bar-involution on $M_{\zeta\odot\lambda}$ satisfies the same property. It restricts to a bar-involution on $(\mathbf{f}\theta_{\lambda}\mathbf{f})_{\zeta\odot\lambda}$. On the other hand, $M_{\zeta}\otimes \Lambda_{\lambda}$ has a $\Psi$-involution defined by \eqref{Theta and Psi}.

\begin{lemma}\label{varphi and involution}
Let $\zeta\in X$ and $\lambda\in X^+$. Then we have 
$$\Psi\varphi_{\zeta,\lambda}(m)=\varphi_{\zeta,\lambda}(\overline{m})\ \textrm{for any}\ m\in (\mathbf{f}\theta_{\lambda}\mathbf{f})_{\zeta\odot\lambda}.$$
\end{lemma}
\begin{proof}
We simply denote $\varphi_{\zeta,\lambda}=\varphi$. By definition, both $\Psi\varphi$ and $\varphi(\bar{\cdot})$ are $\mathbb{Q}$-linear exchanging $v$ and $v^{-1}$. The vector space $(\mathbf{f}\theta_{\lambda}\mathbf{f})_{\zeta\odot\lambda}$ is spanned by $x^-\theta_{\lambda}^-y^-v_{\zeta\odot\lambda}$ for any homogeneous $x,y\in \mathbf{f}$. We prove $\Psi\varphi(x^-\theta_{\lambda}^-y^-v_{\zeta\odot\lambda})=\varphi(\overline{x^-\theta_{\lambda}^-y^-v_{\zeta\odot\lambda}})$ by induction on $\mathrm{tr}\,|x|\in \mathbb{N}$. If $\mathrm{tr}\,|x|=0$, then $x\in \mathbb{Q}(v)$. By \eqref{im varphi} and \eqref{Theta and Psi}, we have 
\begin{align*}
\Psi\varphi(\theta_{\lambda}^-y^-v_{\zeta\odot\lambda})=&\Psi(y^-v_{\zeta}\otimes \eta_{\lambda})=\Theta(\overline{y^-v_{\zeta}}\otimes \overline{\eta_{\lambda}})\\
=&(\overline{y})^-v_{\zeta}\otimes \eta_{\lambda}=\varphi(\theta_{\lambda}^-(\overline{y})^-v_{\zeta\odot\lambda})=\varphi(\overline{\theta_{\lambda}^-y^-v_{\zeta\odot\lambda}})   . 
\end{align*}
If $\mathrm{tr}\,|x|>0$, then it suffices to prove the statement for $x=\theta_ix'$ with $i\in I$ and homogeneous $x'\in \mathbf{f}$. By \cite[Lemma 24.1.2]{Lusztig-1993}, $\Psi(um)=\overline{u}\Psi(m)$ for any $u\in \mathbf{U}$ and $m\in M_{\zeta}\otimes \Lambda_{\lambda}$. Since $\mathrm{tr}\,|x'|<\mathrm{tr}\,|x|$, by Lemma \ref{isomorphism varphi} and the inductive hypothesis,
\begin{align*}
\Psi\varphi(x^-\theta_{\lambda}^-y^-v_{\zeta\odot\lambda})=&\Psi\varphi(F_i{x'}^-\theta_{\lambda}^-y^-v_{\zeta\odot\lambda})=\Psi(F_i\varphi({x'}^-\theta_{\lambda}^-y^-v_{\zeta\odot\lambda}))\\
=&\overline{F_i}\Psi\varphi({x'}^-\theta_{\lambda}^-y^-v_{\zeta\odot\lambda})=F_i\varphi(\overline{{x'}^-\theta_{\lambda}^-y^-v_{\zeta\odot\lambda}})\\
=&\varphi(F_i\overline{{x'}^-\theta_{\lambda}^-y^-v_{\zeta\odot\lambda}})=\varphi(\overline{F_i{x'}^-\theta_{\lambda}^-y^-v_{\zeta\odot\lambda}})=\varphi(\overline{x^-\theta_{\lambda}^-y^-v_{\zeta\odot\lambda}}),
\end{align*}
as desired.
\end{proof}

Next we study the bilinear forms.

On the one hand, by \cite[Proposition 1.2.3]{Lusztig-1993}, there is a non-degenerate symmetric bilinear form $(\,,\,)_{\mathbf{f}}:\mathbf{f}\times \mathbf{f}\rightarrow \mathbb{Q}(v)$. Similarly, there is a bilinear form $(\,,\,)_{\tilde{\mathbf{f}}}:\tilde{\mathbf{f}}\times \tilde{\mathbf{f}}\rightarrow \mathbb{Q}(v)$. On the other hand, by \cite[Proposition 19.1.2]{Lusztig-1993}, there is a symmetric bilinear form $(\,,\,)_{\Lambda_{\lambda}}:\Lambda_{\lambda}\times \Lambda_{\lambda}\rightarrow \mathbb{Q}(v)$. We define a symmetric bilinear form 
$(\,,\,)_{\otimes}:(M_{\zeta}\otimes \Lambda_{\lambda})\times (M_{\zeta}\otimes \Lambda_{\lambda})\rightarrow \mathbb{Q}(v)$
by
\begin{equation}\label{definition the bilinear form on tensor product}
(x_1^-v_{\zeta}\otimes m_2,{x'_1}^-v_{\zeta}\otimes m'_2)_{\otimes}=(x_1,x'_1)_{\mathbf{f}}(m_2,m'_2)_{\Lambda_{\lambda}}\ \textrm{for}\ x_1,x'_1\in \mathbf{f},m_2,m'_2\in \Lambda_{\lambda}.
\end{equation}

For any $i\in I$, let ${_{i}r}:\mathbf{f}\rightarrow \mathbf{f}$ and  ${_{i}r}:\tilde{\mathbf{f}}\rightarrow \tilde{\mathbf{f}}$ be the linear maps defined in \cite[\S 1.2.13]{Lusztig-1993}. For any homogeneous $x,y\in \mathbf{f}$, by definition, we have 
\begin{equation}\label{irxthetalambday}
{_{i}r}(x\theta_{\lambda}y)={_{i}r}(x)\theta_{\lambda}y+v^{i\cdot(|x|+|\theta_{\lambda}|)}x\theta_{\lambda}\,{_{i}r}(y)\in \mathbf{f}\theta_{\lambda}\mathbf{f}.
\end{equation}
Hence ${_{i}r}:\tilde{\mathbf{f}}\rightarrow \tilde{\mathbf{f}}$ can restrict to ${_{i}r}:\mathbf{f}\theta_{\lambda}\mathbf{f}\rightarrow \mathbf{f}\theta_{\lambda}\mathbf{f}$. We identify $ (\mathbf{f}\theta_{\lambda}\mathbf{f})_{\zeta\odot\lambda}\cong \mathbf{f}\theta_{\lambda}\mathbf{f}$ via $z^-v_{\zeta\odot\lambda}\mapsto z$ for any $z\in \mathbf{f}\theta_{\lambda}\mathbf{f}$, then we obtain a linear map 
$${_{i}r}:(\mathbf{f}\theta_{\lambda}\mathbf{f})_{\zeta\odot\lambda}\rightarrow (\mathbf{f}\theta_{\lambda}\mathbf{f})_{\zeta\odot\lambda},z^-v_{\zeta\odot\lambda}\mapsto ({_ir(z)})^-v_{\zeta\odot\lambda}\ \textrm{for any}\ z\in \mathbf{f}\theta_{\lambda}\mathbf{f}.$$ 
Let $\epsilon_i:M_{\zeta}\otimes \Lambda_{\lambda}\rightarrow M_{\zeta}\otimes \Lambda_{\lambda}$ be the linear map defined in \cite[\S 18.1.3]{Lusztig-1993}, that is,
\begin{equation}\label{definition of epsiloni}
\epsilon_i(x_1^-v_{\zeta}\otimes m_2)=({_{i}r}(x_1))^-v_{\zeta}\otimes K_{-i_Y}m_2+(v-v^{-1})x_1^-v_{\zeta}\otimes K_{-i_Y}E_i m_2
\end{equation}
for any $x_1\in \mathbf{f}$ and $m_2\in \Lambda_{\lambda}$.

\begin{lemma}\label{varphi and inner product}
Let $\zeta\in X,\lambda\in X^+,i\in I$ and $x_1,x'_1\in \mathbf{f},m_2,m'_2\in \Lambda_{\lambda},z,z'\in \mathbf{f}\theta_{\lambda}\mathbf{f}$. Then we have
\begin{enumerate}
\item $\epsilon_i\varphi_{\zeta,\lambda}=\varphi_{\zeta,\lambda}\,{_{i}r}:(\mathbf{f}\theta_{\lambda}\mathbf{f})_{\zeta\odot\lambda}\rightarrow M_{\zeta}\otimes \Lambda_{\lambda}$;
\item $(F_i(x_1^-v_{\zeta}\otimes m_2),{x'_1}^-v_{\zeta}\otimes m'_2)_{\otimes}=(1-v^{-2})^{-1}(x_1^-v_{\zeta}\otimes m_2,\epsilon_i({x'_1}^-v_{\zeta}\otimes m'_2))_{\otimes}$;
\item $(\varphi_{\zeta,\lambda}(z^-v_{\zeta\odot\lambda}),\varphi_{\zeta,\lambda}({z'}^-v_{\zeta\odot\lambda}))_{\otimes}=(z,z')_{\tilde{\mathbf{f}}}\cdot \prod_{i\in I}\prod_{k=1}^{\langle i_Y,\lambda\rangle}(1-v^{-2k})$.
\end{enumerate}
\end{lemma}
\begin{proof}
We simply denote $\varphi_{\zeta,\lambda}=\varphi$.

(1) The vector space $(\mathbf{f}\theta_{\lambda}\mathbf{f})_{\zeta\odot\lambda}$ is spanned by $x^-\theta_{\lambda}^-y^-v_{\zeta\odot\lambda}$ for any homogeneous $x,y\in \mathbf{f}$. We prove $\epsilon_i\varphi(x^-\theta_{\lambda}^-y^-v_{\zeta\odot\lambda})=\varphi\,{_{i}r}(x^-\theta_{\lambda}^-y^-v_{\zeta\odot\lambda})$ by induction on $\mathrm{tr}\,|x|\in \mathbb{N}$. If $\mathrm{tr}\,|x|=0$, then $x\in \mathbb{Q}(v)$. By \eqref{im varphi}, \eqref{definition of epsiloni},\eqref{irxthetalambday}, we have 
\begin{align*}
&\epsilon_i\varphi(\theta_{\lambda}^-y^-v_{\zeta\odot\lambda})=\epsilon_i(y^-v_{\zeta}\otimes \eta_{\lambda})=({_{i}r}(y))^-v_{\zeta}\otimes K_{-i_Y}\eta_{\lambda}\\
=&v^{-\langle i_Y,\lambda\rangle}({_{i}r}(y))^-v_{\zeta}\otimes \eta_{\lambda}=\varphi(v^{-\langle i_Y,\lambda\rangle}\theta_{\lambda}^-({_{i}r}(y))^-v_{\zeta\odot\lambda})=\varphi\,{_{i}r}(\theta_{\lambda}^-y^-v_{\zeta\odot\lambda}).
\end{align*}
If $\mathrm{tr}\,|x|>0$, then it suffices to prove the statement for $x=\theta_jx'$ with $j\in I$ and homogeneous $x'\in \mathbf{f}$. Since $\mathrm{tr}\,|x'|<\mathrm{tr}\,|x|$, by Lemma \ref{isomorphism varphi}, \cite[\S 18.1.3 \& \S 15.1.2(b) \& Lemma 15.1.4]{Lusztig-1993} and the inductive hypothesis, we have 
\begin{align*}
&\epsilon_i\varphi(x^-\theta_{\lambda}^-y^-v_{\zeta\odot\lambda})=\epsilon_i\varphi(F_j{x'}^-\theta_{\lambda}^-y^-v_{\zeta\odot\lambda})=\epsilon_i(F_j\varphi({x'}^-\theta_{\lambda}^-y^-v_{\zeta\odot\lambda}))\\
=&v^{i\cdot j}F_j\epsilon_i\varphi({x'}^-\theta_{\lambda}^-y^-v_{\zeta\odot\lambda})+\delta_{i,j}\varphi({x'}^-\theta_{\lambda}^-y^-v_{\zeta\odot\lambda})\\
=&v^{i\cdot j}F_j\varphi\,{_{i}r}({x'}^-\theta_{\lambda}^-y^-v_{\zeta\odot\lambda})+\delta_{i,j}\varphi({x'}^-\theta_{\lambda}^-y^-v_{\zeta\odot\lambda})\\
=&v^{i\cdot j}\varphi(F_j\,{_{i}r}({x'}^-\theta_{\lambda}^-y^-v_{\zeta\odot\lambda}))+\delta_{i,j}\varphi({x'}^-\theta_{\lambda}^-y^-v_{\zeta\odot\lambda})\\
=&\varphi(v^{i\cdot j}\theta_j\,{_{i}r}({x'}^-\theta_{\lambda}^-y^-v_{\zeta\odot\lambda})+\delta_{i,j}{x'}^-\theta_{\lambda}^-y^-v_{\zeta\odot\lambda})=\varphi\,{_{i}r}({x'}^-\theta_{\lambda}^-y^-v_{\zeta\odot\lambda}),
\end{align*}
as desired.

(2) By definition and \cite[Proposition 19.1.2 \& \S 1.2.13]{Lusztig-1993}, we have
\begin{align*}
&(F_i(x_1^-v_{\zeta}\otimes m_2),{x'_1}^-v_{\zeta}\otimes m'_2)_{\otimes}\\
=&(x_1^-v_{\zeta}\otimes F_im_2+F_ix_1^-v_{\zeta}\otimes K_{-i_Y}m_2,{x'_1}^-v_{\zeta}\otimes m'_2)_{\otimes}\\
=&(x_1,x'_1)_{\mathbf{f}}(F_im_2,m'_2)_{\Lambda_{\lambda}}+(\theta_ix_1,x'_1)_{\mathbf{f}}(K_{-i_Y}m_2,m'_2)_{\Lambda_{\lambda}}\\
=&(x_1,x'_1)_{\mathbf{f}}(m_2,vK_{-i_Y}E_im'_2)_{\Lambda_{\lambda}}+(1-v^{-2})^{-1}(x_1,{_{i}}r(x'_1))_{\mathbf{f}}(m_2,K_{-i_Y}m'_2)_{\Lambda_{\lambda}}\\
=&(1-v^{-2})^{-1}(x_1^-v_{\zeta}\otimes m_2,(v-v^{-1}){x'_1}^-v_{\zeta}\otimes K_{-i_Y}E_im'_2+({_{i}}r(x'_1))^-v_{\zeta}\otimes K_{-i_Y}m'_2)_{\otimes}\\
=&(1-v^{-2})^{-1}(x_1^-v_{\zeta}\otimes m_2,\epsilon_i({x'_1}^-v_{\zeta}\otimes m'_2))_{\otimes}.
\end{align*}

(3) The vector space $\mathbf{f}\theta_{\lambda}\mathbf{f}$ is spanned by $x\theta_{\lambda}y$ for any homogeneous $x,y\in \mathbf{f}$. We prove the statement for any $z=x\theta_{\lambda}y$ and $z'=x'\theta_{\lambda}y'$ by induction on $\mathrm{tr}\,|x|\in \mathbb{N}$. If $\mathrm{tr}\,|x|=0$, then $x\in \mathbb{Q}(v)$. On the one hand, we write $r(x')=\sum x'_1\otimes x'_2$ with homogeneous $x'_1,x'_2\in \mathbf{f}$, then by \eqref{im varphi} and \cite[Proposition 19.1.2]{Lusztig-1993}, we have
\begin{align*}
&(\varphi(\theta_{\lambda}^-y^-v_{\zeta\odot\lambda}),\varphi({x'}^-\theta_{\lambda}^-{y'}^-v_{\zeta\odot\lambda}))_{\otimes}=(y^-v_{\zeta}\otimes \eta_{\lambda},\sum v^{|x'_2|\cdot |\theta_{\lambda}|}{x'_2}^-{y'}^-v_{\zeta}\otimes {x'_1}^-\eta_{\lambda})_{\otimes}\\
=&\sum v^{|x'_2|\cdot |\theta_{\lambda}|}(y,x'_2y')_{\mathbf{f}}(\eta_{\lambda},{x'_1}^-\eta_{\lambda})_{\Lambda_{\lambda}}=v^{|x'|\cdot |\theta_{\lambda}|}(y,x'y')_{\mathbf{f}}.
\end{align*}
On the other hand, by \cite[Proposition 1.2.3]{Lusztig-1993} and the fact that $(\tilde{\mathbf{f}}_{\nu},\tilde{\mathbf{f}}_{\nu'})_{\tilde{\mathbf{f}}}=0$ for any $\nu\not=\nu'$ in $\mathbb{N}[\tilde{I}]$ (see the proof of \cite[Proposition 1.2.3]{Lusztig-1993}), we have 
\begin{align*}
(\theta_{\lambda}y,x'\theta_{\lambda}y')_{\tilde{\mathbf{f}}}=&(\theta_{\lambda}\otimes y,r(x'\theta_{\lambda}y'))_{\tilde{\mathbf{f}}\otimes\tilde{\mathbf{f}}}=(\theta_{\lambda}\otimes y,v^{|x'|\cdot|\theta_{\lambda}|}\theta_{\lambda}\otimes x'y')_{\tilde{\mathbf{f}}\otimes\tilde{\mathbf{f}}}\\
=&v^{|x'|\cdot|\theta_{\lambda}|}(\theta_{\lambda},\theta_{\lambda})_{\tilde{\mathbf{f}}}(y,x'y')_{\tilde{\mathbf{f}}}.
\end{align*}
It is clear that $(y,x'y')_{\mathbf{f}}=(y,x'y')_{\tilde{\mathbf{f}}}$. By \cite[Proposition 1.2.3 \& Lemma 1.4.4]{Lusztig-1993}, we have 
$(\theta_{\lambda},\theta_{\lambda})_{\tilde{\mathbf{f}}}=\prod_{i\in I}\prod_{k=1}^{\langle i_Y,\lambda\rangle}(1-v^{-2k})^{-1}$
and so $$(\varphi(\theta_{\lambda}^-y^-v_{\zeta\odot\lambda}),\varphi({x'}^-\theta_{\lambda}^-{y'}^-v_{\zeta\odot\lambda}))_{\otimes}=(\theta_{\lambda}y,x'\theta_{\lambda}y')_{\tilde{\mathbf{f}}}\cdot \prod_{i\in I}\prod_{k=1}^{\langle i_Y,\lambda\rangle}(1-v^{-2k}).$$
If $\mathrm{tr}\,|x|>0$, then it suffices to prove the statement for any $x=\theta_ix''$ with $i\in I$ homogeneous $x''\in \mathbf{f}$. Since $\mathrm{tr}\,|x''|<\mathrm{tr}\,|x|$, by Lemma \ref{isomorphism varphi}, (2), (1), the inductive hypothesis and \cite[\S  1.2.13]{Lusztig-1993}, we have 
\begin{align*}
&(\varphi(x^-\theta_{\lambda}^-y^-v_{\zeta\odot\lambda}),\varphi({x'}^-\theta_{\lambda}^-{y'}^-v_{\zeta\odot\lambda}))_{\otimes}=(\varphi(F_i{x''}^-\theta_{\lambda}^-y^-v_{\zeta\odot\lambda}),\varphi({x'}^-\theta_{\lambda}^-{y'}^-v_{\zeta\odot\lambda}))_{\otimes}\\
=&(F_i\varphi({x''}^-\theta_{\lambda}^-y^-v_{\zeta\odot\lambda}),\varphi({x'}^-\theta_{\lambda}^-{y'}^-v_{\zeta\odot\lambda}))_{\otimes}\\
=&(1-v^{-2})^{-1}(\varphi({x''}^-\theta_{\lambda}^-y^-v_{\zeta\odot\lambda}),\epsilon_i\varphi({x'}^-\theta_{\lambda}^-{y'}^-v_{\zeta\odot\lambda}))_{\otimes}\\
=&(1-v^{-2})^{-1}(\varphi({x''}^-\theta_{\lambda}^-y^-v_{\zeta\odot\lambda}),\varphi\,{_{i}r}({x'}^-\theta_{\lambda}^-{y'}^-v_{\zeta\odot\lambda}))_{\otimes}\\
=&(1-v^{-2})^{-1}(x''\theta_{\lambda}y,\,{_{i}r}(x'\theta_{\lambda}y'))_{\tilde{\mathbf{f}}}\cdot \prod_{i\in I}\prod_{k=1}^{\langle i_Y,\lambda\rangle}(1-v^{-2k})\\
=&(\theta_ix''\theta_{\lambda}y,x'\theta_{\lambda}y')_{\tilde{\mathbf{f}}}\cdot \prod_{i\in I}\prod_{k=1}^{\langle i_Y,\lambda\rangle}(1-v^{-2k})=(x\theta_{\lambda}y,x'\theta_{\lambda}y')_{\tilde{\mathbf{f}}}\cdot \prod_{i\in I}\prod_{k=1}^{\langle i_Y,\lambda\rangle}(1-v^{-2k}),
\end{align*}
as desired.
\end{proof}

\subsection{Realization of canonical basis of tensor product}

In this subsection, we study the compatibility of the thickening realizations of tensor products established in \S \ref{ftheta_lambdaf} with the canonical bases.

\begin{theorem}\label{positivity of transition matrix}
Let $\zeta\in X$ and $\lambda\in X^+$. Then 
\begin{enumerate}
\item there are bijections
$$\xymatrix@C=1.5cm{\mathbf{B}((\mathbf{f}\theta_{\lambda}\mathbf{f})_{\zeta \odot\lambda}) \ar@<.5ex>[r]^{\varphi_{\zeta,\lambda}} &\mathbf{B}(M_{\zeta}\otimes \Lambda_{\lambda}); \ar@<.5ex>[l]^{\psi_{\zeta,\lambda}}}$$
\item let $\tilde{b}\in \mathbf{B}(\mathbf{f}\theta_{\lambda}\mathbf{f})$, write $r(\tilde{b})=\sum_{\tilde{b}_1,\tilde{b}_2\in \tilde{\mathbf{B}}}c_{\tilde{b}_1,\tilde{b}_2}^{\tilde{b}}\tilde{b}_1\otimes\tilde{b}_2$, then
$$\varphi_{\zeta,\lambda}(\tilde{b}^-v_{\zeta\odot\lambda})=\sum_{b_1\in \mathbf{B},b_2^-\eta_{\lambda}\in \mathbf{B}(\Lambda_{\lambda})}c_{b_2\theta_{\lambda_2},b_1}^{\tilde{b}}b_1^-v_{\zeta}\otimes b_2^-\eta_{\lambda}.$$
\end{enumerate}
\end{theorem}

\begin{remark}\label{structure constant remark}
A similar result to part (1) that $\psi_{\zeta,\lambda}$ is a bijection was established by Li in \cite[Theorem 7.17]{Li-2014} using perverse sheaves and crystal theory. Our approach is different and leads to the comparison of the entries of the transition matrix with the structure constants of the comultiplication in part (2). 
\end{remark}

\begin{proof}
We simply denote $\varphi_{\zeta,\lambda}=\varphi$. For any $\tilde{b}\in \mathbf{B}(\mathbf{f}\theta_{\lambda}\mathbf{f})\subset \tilde{\mathbf{B}}$, we write $r(\tilde{b})=\sum_{\tilde{b}_1,\tilde{b}_2\in \tilde{\mathbf{B}}}c_{\tilde{b}_1,\tilde{b}_2}^{\tilde{b}} \tilde{b}_1\otimes \tilde{b}_2$. Then $c_{\tilde{b}_1,\tilde{b}_2}^{\tilde{b}}\in \mathbb{N}[v,v^{-1}]$ by the positivity property of $\tilde{\mathbf{B}}$ with respect to the comultiplication in Theorem \ref{14.4.13}. By the definition of $\varphi$ and \cite[\S 3.1.5 \& Lemma 1.2.11]{Lusztig-1993}, we have 
\begin{align*}
\varphi(\tilde{b}^-v_{\zeta\odot\lambda})=&\sum v^{-|\tilde{b}_1|\cdot|b_2|}c_{\tilde{b}_1,b_2}^{\tilde{b}}b_2^-v_{\zeta}\otimes \phi_{\lambda}^{-1}(K_{-|b_2|_Y}\tilde{b}_1^-\eta_{0\odot\lambda})\\
=&\sum c_{\tilde{b}_1,b_2}^{\tilde{b}}b_2^-v_{\zeta}\otimes \phi_{\lambda}^{-1}(\tilde{b}_1^-\eta_{0\odot\lambda}),    
\end{align*}
where the summation is taken over $\tilde{b}_1\in \mathbf{B}(\mathbf{f}\theta_{\lambda}\mathbf{f})$ and $b_2\in \mathbf{B}$ such that $\tilde{b}_1^-\eta_{0\odot\lambda}\not=0$. Since $\tilde{b}_1^-\eta_{0\odot\lambda}\not=0$, by \eqref{14.4.11}, we have $\tilde{b}_1\in \tilde{\mathbf{B}}(0\odot\lambda)\subset \bigcap_{i\in I}(\tilde{\mathbf{B}}\setminus \tilde{\mathbf{f}}\theta_i)$. Then by \eqref{Fang-Lan-2025-5.3*}, there exists a unique $b_1\in\mathbf{B}(\lambda)$ such that $\tilde{b}_1=b_1\theta_{\lambda}$. Hence
\begin{align}\label{varphib integral}
\varphi(\tilde{b}^-v_{\zeta\odot\lambda})=\sum c_{b_1\theta_{\lambda},b_2}^{\tilde{b}}b_2^-v_{\zeta}\otimes b_1^-\eta_{\lambda}\in \mathbb{N}[v,v^{-1}][\mathbf{B}(M_{\zeta})\otimes \mathbf{B}(\Lambda_{\lambda})].   
\end{align}

Let $\mathbf{A}=\mathbb{Q}[[v^{-1}]]\cap \mathbb{Q}(v)$ and $L(M_{\zeta}\otimes \Lambda_{\lambda})$ be the $\mathbf{A}$-submodule of $M_{\zeta}\otimes \Lambda_{\lambda}$ generated by $\mathbf{B}(M_{\zeta})\otimes \mathbf{B}(\Lambda_{\lambda})$. By \cite[Theorem 14.2.3 \& Proposition 19.3.3]{Lusztig-1993}, the basis $\mathbf{B}(M_{\zeta})\otimes \mathbf{B}(\Lambda_{\lambda})$ of $M_{\zeta}\otimes \Lambda_{\lambda}$ is almost orthonormal (in the sense of \cite[\S  14.2.1]{Lusztig-1993}) with respect to the bilinear form $(\,,\,)_{\otimes}$ defined in \eqref{definition the bilinear form on tensor product}. By \cite[Theorem 14.2.3]{Lusztig-1993} and Lemma \ref{varphi and inner product}, we have $(\varphi(\tilde{b}^-v_{\zeta\odot\lambda}),\varphi(\tilde{b}^-v_{\zeta\odot\lambda}))_{\otimes}=(\tilde{b},\tilde{b})_{\tilde{\mathbf{f}}}\cdot \prod_{i\in I}\prod_{k=1}^{\langle i_Y,\lambda\rangle}(1-v^{-2k})\in 1+v^{-1}\mathbf{A}$. By \cite[Lemma 14.2.2]{Lusztig-1993}, there exists ${b'_1}^-v_{\zeta}\otimes {b'_2}^-\eta_{\lambda}\in \mathbf{B}(M_{\zeta})\otimes \mathbf{B}(\Lambda_{\lambda})$ such that 
$$\varphi(\tilde{b}^-v_{\zeta\odot\lambda})=\pm {b'_1}^-v_{\zeta}\otimes {b'_2}^-\eta_{\lambda}\ \mathrm{mod}\,v^{-1}L(M_{\zeta}\otimes \Lambda_{\lambda}).$$
By \eqref{varphib integral} and $\mathbb{N}[v,v^{-1}]\cap\mathbf{A}=\mathbb{N}[v^{-1}]$, we have 
\begin{align*}
\varphi(\tilde{b}^-v_{\zeta\odot\lambda})={b'_1}^-v_{\zeta}\otimes {b'_2}^-\eta_{\lambda}\ \mathrm{mod}\,v^{-1}\mathbb{N}[v^{-1}][\mathbf{B}(M_{\zeta})\otimes \mathbf{B}(\Lambda_{\lambda})].
\end{align*}
By \cite[Theorem 14.2.3]{Lusztig-1993} and Lemma \ref{varphi and involution}, we have $$\overline{\tilde{b}^-v_{\zeta\odot\lambda}}=(\overline{\tilde{b}})^-v_{\zeta\odot\lambda}=\tilde{b}^-v_{\zeta\odot\lambda},\ \Psi\varphi(\tilde{b}^-v_{\zeta\odot\lambda})=\varphi(\overline{\tilde{b}^-v_{\zeta\odot\lambda}})=\varphi(\tilde{b}^-v_{\zeta\odot\lambda}).$$ Hence $\varphi(\tilde{b}^-v_{\zeta\odot\lambda})$ satisfies the characterizing conditions in Lemma \ref{canonical basis of Verma otime highest weight}, and so $\varphi(\tilde{b}^-v_{\zeta\odot\lambda})={b'_1}^-v_{\zeta}\diamondsuit {b'_2}^-\eta_{\lambda}\in \mathbf{B}(M_{\zeta}\otimes \Lambda_{\lambda})$. Hence $\varphi(\mathbf{B}((\mathbf{f}\theta_{\lambda}\mathbf{f})_{\zeta\odot\lambda}))\subset \mathbf{B}(M_{\zeta}\otimes \Lambda_{\lambda})$. Then part (1) follows from Lemma \ref{isomorphism varphi}, and part (2) follows from \eqref{varphib integral}.
\end{proof}

Combining with Lemma \ref{a otimes Id is a based map}, we obtain

\begin{theorem}\label{thm:structure-trans}
Let $\lambda_1,\lambda_2\in X^+$ and $w\in W$. Then
\begin{enumerate}
\item there is a bijection
$$\mathbf{B}((\mathbf{f}\theta_{\lambda_2}\mathbf{f})_{-w\lambda_1\odot\lambda_2})\setminus\mathrm{ker}\,\underline{\varphi_{-w\lambda_1,\lambda_2}}\xrightarrow{\underline{\varphi_{-w\lambda_1,\lambda_2}}}\mathbf{B}({^{\omega}V_w(\lambda_1)}\otimes \Lambda_{\lambda_2});$$
\item let $\tilde{b}\in \mathbf{B}(\mathbf{f}\theta_{\lambda_2}\mathbf{f})$, write $r(\tilde{b})=\sum_{\tilde{b}_1,\tilde{b}_2\in \tilde{\mathbf{B}}}c_{\tilde{b}_1,\tilde{b}_2}^{\tilde{b}}\tilde{b}_1\otimes\tilde{b}_2$, then 
$$\underline{\varphi_{-w\lambda_1,\lambda_2}}(\tilde{b}^-v_{-w\lambda_1\odot\lambda_2})=\sum c_{b_2\theta_{\lambda_2},b_1}^{\tilde{b}}b_1^-\xi_{-w\lambda_1}\otimes b_2^-\eta_{\lambda_2},$$
where the summation is taken over $b_1^-\xi_{-w\lambda_1}\in \mathbf{B}({^{\omega}V_w(\lambda_1)})$ and $b_2^-\eta_{\lambda_2}\in \mathbf{B}(\Lambda_{\lambda_2})$.
\end{enumerate}
\end{theorem}

\subsection{Thickening map}\label{Thickening map}

Let $\lambda_1,\lambda_2\in X^+$ and $w\in W$. We define the {\it thickening map}
$$\mathrm{th}_{\lambda_1,\lambda_2,w}:\dot{\mathbf{B}}\rightarrow \tilde{\mathbf{B}}\sqcup\{0\},\dot{b}\mapsto \begin{cases}\tilde{b}\ &\textrm{if}\ \dot{b}(\xi_{-\lambda_1}\otimes \eta_{\lambda_2})\in \mathbf{B}({^{\omega}V_w(\lambda_1)\otimes \Lambda_{\lambda_2}});\\0 &\textrm{otherwise},\end{cases}$$
where in the first case, $\tilde{b}$ is the unique element in $\mathbf{B}(\mathbf{f}\theta_{\lambda_2}\mathbf{f})$ such that 
\begin{equation}\label{definition of thickening map}
\dot{b}(\xi_{-\lambda_1}\otimes \eta_{\lambda_2})=\underline{\varphi_{-w\lambda_1,\lambda_2}}(\tilde{b}^-v_{-w\lambda_1\odot\lambda_2}).   
\end{equation}
The existence and uniqueness of such $\tilde{b}$ follow from Theorem \ref{thm:structure-trans}.

For any fixed $\dot{b}\in\dot{\mathbf{B}}$, by Theorem \ref{25.2.1}, we have $\dot{b}(\xi_{-\lambda_1}\otimes \eta_{\lambda_2})\in \mathbf{B}({^{\omega}\Lambda_{\lambda_1}}\otimes \Lambda_{\lambda_2})$ for sufficiently regular $\lambda_1,\lambda_2\in X^+$, that is, $\langle i_Y,\lambda_1\rangle,\langle i_Y,\lambda_2\rangle\in \mathbb{N}$ are sufficiently large for any $i\in I$. Then by $\mathbf{B}({^{\omega}\Lambda_{\lambda}})=\bigcup_{w\in W}\mathbf{B}({^{\omega}V_{w}(\lambda)})$ and Proposition \ref{canonical basis of demazure otimes highest weight}, $\dot{b}(\xi_{-\lambda_1}\otimes \eta_{\lambda_2})\in \mathbf{B}({}^{\omega}V_w(\lambda_1)\otimes\Lambda_{\lambda_2})$ for sufficiently large $w\in W$. In this way, every canonical basis element in $\dot{\mathbf{B}}$ can be realized as a canonical basis element in $\tilde{\mathbf{B}}$ via the thickening map after choosing appropriate parameters $\lambda_1,\lambda_2,w$.

\subsection{Example}\label{Example of thickening map}
We present the example for $\mathfrak{sl}_2$.

Let $(I,\cdot)$ be the Cartan datum of type $A_1$ and $(Y,X,\langle,\rangle,...)$ be the simply connected root datum of type $(I,\cdot)$, that is, $I=\{i\}, i\cdot i=2$ and $Y=X=\mathbb{Z},i_Y=1,i_X=2$. Then the canonical basis of $\mathbf{f}$ is $\mathbf{B}=\{\theta_i^{(k)}; k\in \mathbb{N}\}$.

For $m,n\in \mathbb{N}$, the canonical bases of ${^{\omega}\Lambda_m}$ and $\Lambda_{n}$ are 
$$\mathbf{B}({^{\omega}\Lambda_m})=\{E_i^{(k)}\xi_{-m}; 0\leqslant k\leqslant m\}\ \textrm{and}\ \mathbf{B}(\Lambda_n)=\{F_i^{(l)}\eta_n; 0\leqslant l\leqslant n\}$$
respectively. The canonical basis of ${^{\omega}\Lambda_m}\otimes \Lambda_n$ is $$\mathbf{B}({^{\omega}\Lambda_m}\otimes \Lambda_n)=\{E_i^{(k)}\xi_{-m}\diamondsuit F_i^{(l)}\eta_n; 0\leqslant k\leqslant m,0\leqslant l\leqslant n\},$$ where $E_i^{(k)}\xi_{-m}\diamondsuit F_i^{(l)}\eta_n$ is given by
\begin{align*}
\begin{cases}
\displaystyle\sum_{s=0}^{\mathrm{min}(k,l)}v^{s(k-m-s)}\begin{bmatrix}
n-l+s\\s
\end{bmatrix} E_i^{(k-s)}\xi_{-m}\otimes F_i^{(l-s)}\eta_n &\textrm{if}\  n-m\leqslant l-k;\\
\displaystyle \sum_{s=0}^{\mathrm{min}(k,l)}v^{s(l-n-s)}\begin{bmatrix}
m-k+s\\s
\end{bmatrix} E_i^{(k-s)}\xi_{-m}\otimes F_i^{(l-s)}\eta_n &\textrm{if}\ n-m\geqslant l-k
\end{cases}
\end{align*}
with the identification of two expressions for $n-m=l-k$ (see \cite[\S 6]{Lusztig-1992}), where $\begin{bmatrix}
q\\p
\end{bmatrix}=[q]!([p]!\,[q-p]!)^{-1}$ for any $p\leqslant q$ in $\mathbb{N}$. 

The canonical basis of $\dot{\mathbf{U}}$ is $\dot{\mathbf{B}}=\{\theta_i^{(k)}\diamondsuit_t\,\theta_i^{(l)};t\in \mathbb{Z},k,l\in \mathbb{N}\}$, where 
$$\theta_i^{(k)}\diamondsuit_t\,\theta_i^{(l)}=\begin{cases}
E_i^{(k)}F_i^{(l)}1_t\ &\textrm{if}\ t\leqslant l-k;\\
F_i^{(l)}E_i^{(k)}1_t\ &\textrm{if}\ t\geqslant l-k
\end{cases}$$
with the identification $E_i^{(k)}F_i^{(l)}1_t=F_i^{(l)}E_i^{(k)}1_t$ for $t=l-k$ (see \cite[\S 25.3.1]{Lusztig-1993}).

The thickening Cartan datum $(\tilde{I},\cdot)$ is of type $A_2$. The canonical basis of $\tilde{\mathbf{f}}$ is
$$\tilde{\mathbf{B}}=\{\theta_i^{(p)}\theta_{i'}^{(q)}\theta_i^{(r)}; p,q,r\in \mathbb{N},q\geqslant p+r\}\cup\{\theta_{i'}^{(r)}\theta_i^{(q)}\theta_{i'}^{(p)}; p,q,r\in \mathbb{N},q\geqslant p+r\}$$
with the identification $\theta_i^{(p)}\theta_{i'}^{(q)}\theta_i^{(r)}=\theta_{i'}^{(r)}\theta_i^{(q)}\theta_{i'}^{(p)}$ for $q=p+r$ (see \cite[\S 14.5.4]{Lusztig-1993}). Let $w=-1$. Then we have $^{\omega}V_{-1}(m)={^{\omega}\Lambda_{m}}$ and the thickening realization $\underline{\varphi_{m,n}}:(\mathbf{f}\theta_n\mathbf{f})_{m\odot n}\rightarrow {^{\omega}\Lambda_m}\otimes \Lambda_n$ such that 
$$E_i^{(k)}\xi_{-m}\diamondsuit F_i^{(l)}\eta_n=\begin{cases}
\underline{\varphi_{m,n}}(F_{i'}^{(n-l)}F_i^{(m-k+l)}F_{i'}^{(l)}v_{m\odot n})\ &\textrm{if}\ n-m\leqslant l-k;\\
\underline{\varphi_{m,n}}(F_i^{(l)}F_{i'}^{(n)}F_i^{(m-k)}v_{m\odot n})\ &\textrm{if}\ n-m\geqslant l-k
\end{cases}$$
for any $0\leqslant k\leqslant m$ and $0\leqslant l\leqslant n$ (see \cite[Proposition 4.9]{Fang-Lan-2025}).

The thickening map $\mathrm{th}_{m,n,-1}:\dot{\mathbf{B}}\rightarrow \tilde{\mathbf{B}}\sqcup\{0\}$ is given by $\mathrm{th}_{m,n,-1}(\theta_i^{(k)}\diamondsuit_t\,\theta_i^{(l)})=0$ unless $t=n-m,0\leqslant k\leqslant m,0\leqslant l\leqslant n$, and 
$$\mathrm{th}_{m,n,-1}(\theta_i^{(k)}\diamondsuit_{n-m}\,\theta_i^{(l)})=\begin{cases}
\theta_{i'}^{(n-l)}\theta_i^{(m-k+l)}\theta_{i'}^{(l)}\ &\textrm{if}\ n-m\leqslant l-k;\\
\theta_i^{(l)}\theta_{i'}^{(n)}\theta_i^{(m-k)}\ &\textrm{if}\ n-m\geqslant l-k.
\end{cases}$$

\section{Structure constant of multiplication}\label{Structure constant of multiplication}

By using the thickening realizations, we compare the structure constants of the multiplication in the modified quantum group and the multiplication in the negative part of a larger quantum group with respect to their canonical bases.

\subsection{Spherical parabolic subalgebra}
Let $J \subset I$ be a subset, and let $\mathbf{U}_J^+,\mathbf{U}_J^-$ be the subalgebras of $\mathbf{U}$ generated by $E_j,F_j$ for any $j\in J$ respectively. 

By \cite[\S 23.2.1]{Lusztig-1993}, $\dot{\mathbf{U}}$ is a free $(\mathbf{U}^+\otimes (\mathbf{U}^-)^{\mathrm{opp}})$-module with a basis $\{1_{\zeta};\zeta\in X\}$. Let $\dot{\mathbf{U}}_J$ be the $(\mathbf{U}_J^+\otimes (\mathbf{U}^-)^{\mathrm{opp}})$-submodule generated by $\{1_{\zeta};\zeta\in X\}$, and ${^{\omega}\dot{\mathbf{U}}_J}$ be the $(\mathbf{U}^+\otimes (\mathbf{U}_J^-)^{\mathrm{opp}})$-submodule generated by $\{1_{\zeta};\zeta\in X\}$. By \cite[\S 23.1.3]{Lusztig-1993}, both $\dot{\mathbf{U}}_J$ and ${^{\omega}\dot{\mathbf{U}}_J}$ are subalgebras of $\dot{\mathbf{U}}$. We call them the {\it parabolic subalgebras} of $\dot{\mathbf{U}}$ associated with $J$.  

A subset $J\subset I$ is called spherical if the subgroup $W_J$ of $W$ generated by $s_j$ for any $j\in J$ is finite. For any $w\in W$, the set $J(w):=\{j\in I; s_j w<w\}$ is spherical. 

A {\it spherical parabolic subalgebra} of $\dot{\mathbf{U}}$ is either $\dot{\mathbf{U}}_J$ or ${}^{\omega}\dot{\mathbf{U}}_J$ associated with a spherical subset $J \subset I$.

\begin{lemma}\label{lem:Demazure-UJ}
Let $\lambda \in X^+,w \in W$ and $J\subset J(w)$. Then the action of $\dot{\mathbf{U}}_J$ on $^{\omega}\Lambda_{\lambda}$ preserves ${^{\omega}V_{w}(\lambda)}$, and the projection map $a_{-w\lambda}: M_{-w\lambda} \to {^{\omega}V_{w}(\lambda)}$ commutes with the action of $\dot{\mathbf{U}}_J$.
\end{lemma}
\begin{proof}
For any $j\in J$, we have $s_jw<w$, and so $E_j({^{\omega}V_{w}(\lambda)})\subset {^{\omega}V_{w}(\lambda)}$. The projection map is already a $\mathbf{U}(\mathfrak{b}^-)$-module homomorphism. It remains to prove that it commutes with the action of $E_j$ for any $j\in J$. Let $r_j:\mathbf{f}\rightarrow \mathbf{f}$ and ${_jr}:\mathbf{f}\rightarrow \mathbf{f}$ be the linear maps defined in \cite[\S 1.2.13]{Lusztig-1993}. For any homogeneous $x\in \mathbf{f}$, by \cite[Proposition 3.1.6]{Lusztig-1993}, we have 
$$E_jx^--x^-E_j=\frac{K_j({_jr(x)})^--(r_j(x))^-K_{-j}}{v-v^{-1}}\in \mathbf{U}.$$
Since $E_jv_{-w\lambda}=0$ in $M_{-w\lambda}$ and $E_j\xi_{-w\lambda}=0$ in ${^{\omega}V_w(\lambda)}$, both $a_{-w\lambda}(E_jx^-v_{-w\lambda})$ and $E_ja_{-w\lambda}(x^-v_{-w\lambda})$ are equal to
$(v-v^{-1})^{-1}(K_j({_jr(x)})^-\xi_{-w\lambda}-(r_j(x))^-K_{-j}\xi_{-w\lambda})$.
\end{proof}

\subsection{A vanishing lemma}
Let $J\subset I$, and let $\mathbf{f}_J$ be the subalgebra of $\mathbf{f}$ generated by $\theta_j$ for any $j\in J$. We regard $\mathbf{f}_J$ as the algebra defined in \S \ref{The algebra f} associated with the sub-Cartan datum $(J,\cdot)$ of $(I,\cdot)$, and let $\mathbf{B}_J$ be the canonical basis of $\mathbf{f}_J$. Then $\mathbf{B}_J\subset \mathbf{B}$. Similar to \cite[\S 23.2.1]{Lusztig-1993}, $\dot{\mathbf{U}}_J$ has a basis $\{b_1^+1_{\zeta}b_2^-;\zeta\in X,b_1\in \mathbf{B}_J,b_2\in \mathbf{B}\}$. Then by the proof of Theorem \ref{25.2.1}, $\{b_1\diamondsuit_{\zeta}b_2\in \dot{\mathbf{B}};\zeta\in X,b_1\in \mathbf{B}_J,b_2\in \mathbf{B}\}$ is a basis of $\dot{\mathbf{U}}_J$. Moreover, it is equal to $\dot{\mathbf{B}}\cap\dot{\mathbf{U}}_J$. 

Let $\dot{\tilde{\mathbf{U}}}$ be the modified quantum group associated with the thickening root datum $(\tilde{Y},\tilde{X},\langle\,,\,\rangle,\ldots)$, and $\dot{\tilde{\mathbf{B}}}$ be its canonical basis. Then there are natural imbeddings $\dot{\mathbf{U}}\hookrightarrow \dot{\tilde{\mathbf{U}}}$ and $\dot{\mathbf{B}}\subset \dot{\tilde{\mathbf{B}}}$.

\begin{lemma}\label{canonical basis of U dot acts on Verma}
Let $\lambda_1,\lambda_2\in X^+,w\in W$ and $J\subset J(w)$. Then for any $b_1\in \mathbf{B}_J$ such that $b_1 \neq 1$ and $b_2\in \mathbf{B}(\mathbf{f}\theta_{\lambda_2}\mathbf{f})$, we have 
$$\underline{\varphi_{-w\lambda_1,\lambda_2}}\bigl((b_1\diamondsuit_{-w\lambda_1\odot\lambda_2}b_2)v_{-w\lambda_1\odot\lambda_2}\bigr)=0.$$
\end{lemma}
\begin{proof}
We define two auxiliary dominant weights $\tilde{\lambda}_1,\tilde{\lambda}_2\in \tilde{X}^+$ by
\begin{align*}
\langle i_{\tilde{Y}},\tilde{\lambda}_1\rangle=\begin{cases}
0 &\textrm{if}\ s_iw<w;\\
\langle i_Y,w\lambda_1\rangle &\textrm{if}\ s_iw>w,
\end{cases}\ \ &\langle i_{\tilde{Y}},\tilde{\lambda}_2\rangle=\begin{cases}
\langle i_Y,-w\lambda_1\rangle &\textrm{if}\ s_iw<w;\\
0 &\textrm{if}\ s_iw>w,\end{cases}\\
\langle i'_{\tilde{Y}},\tilde{\lambda}_1\rangle=0,\ &\langle i'_{\tilde{Y}},\tilde{\lambda}_2\rangle=\langle i_Y,\lambda_2\rangle.
\end{align*}
for any $i\in I$.
A direct verification using \eqref{definition of odot} shows that $-w\lambda_1\odot\lambda_2=\tilde{\lambda}_2-\tilde{\lambda}_1$. 

Consider the following $\tilde{\mathbf{U}}$-module homomorphisms
$$\xymatrix@C=2.5cm{\dot{\tilde{\mathbf{U}}}1_{-w\lambda_1\odot\lambda_2} \ar[d]_-{\dot{u}\mapsto \dot{u}v_{-w\lambda_1\odot\lambda_2}} \ar@{->>}[r]^-{\dot{u}\mapsto \dot{u}(\xi_{-\tilde{\lambda}_1}\otimes v_{\tilde{\lambda}_2})} &{^{\omega}\Lambda_{\tilde{\lambda}_1}}\otimes M_{\tilde{\lambda}_2} \ar@{->>}[r]^-{\mathrm{Id}\otimes \pi_{\tilde{\lambda}_2}} \ar@{-->}[ld]^f &{^{\omega}\Lambda_{\tilde{\lambda}_1}}\otimes \Lambda_{\tilde{\lambda}_2}\\
M_{-w\lambda_1\odot\lambda_2}.}$$
By \cite[\S 23.3.1 \& \S 23.3.5]{Lusztig-1993}, $\dot{\tilde{\mathbf{U}}}1_{-w\lambda_1\odot\lambda_2}\rightarrow {^{\omega}\Lambda_{\tilde{\lambda}_1}}\otimes M_{\tilde{\lambda}_2}, \dot{u}\mapsto \dot{u}(\xi_{-\tilde{\lambda}_1}\otimes v_{\tilde \l_2})$ is surjective with kernel $\sum_{\tilde{i}\in \tilde{I},n>\langle \tilde{i}_{\tilde{Y}}, \tilde{\lambda}_1\rangle}\dot{\tilde{\mathbf{U}}}E_{\tilde{i}}^{(n)}1_{-w\lambda_1\odot\lambda_2}$. It is clear that this kernel acts on $v_{-w\lambda_1\odot\lambda_2}$ by $0$, and so $\dot{\tilde{\mathbf{U}}}1_{-w\lambda_1\odot\lambda_2}\rightarrow M_{-w\lambda_1\odot\lambda_2},\dot{u}\mapsto \dot{u} v_{-w\lambda_1\odot\lambda_2}$ factor though a $\tilde{\mathbf{U}}$-module homomorphism $f:{^{\omega}\Lambda_{\tilde{\lambda}_1}}\otimes M_{\tilde{\lambda}_2}\rightarrow M_{-w\lambda_1\odot\lambda_2}$ such that $f(\xi_{-\tilde{\lambda}_1}\otimes v_{\tilde{\lambda}_2})=v_{-w\lambda_1\odot\lambda_2}$. In particular, we have 
\begin{equation}\label{factor through f}
(b_1\diamondsuit_{-w\lambda_1\odot\lambda_2}b_2)v_{-w\lambda_1\odot\lambda_2}=f((b_1\diamondsuit_{-w\lambda_1\odot\lambda_2}b_2)(\xi_{-\tilde{\lambda}_1}\otimes v_{\tilde{\lambda}_2})).    
\end{equation}

Since $b_1\in \mathbf{B}_J$ and $b_2\in \mathbf{B}(\mathbf{f}\theta_{\lambda_2}\mathbf{f})$, by the construction of $b_1\diamondsuit_{-w\lambda_1\odot\lambda_2}b_2\in \dot{\tilde{\mathbf{B}}}$ in Theorem \ref{25.2.1} and the fact that $(\mathbf{f}\theta_{\lambda_2}\mathbf{f})_{\tilde{\lambda}}$ is a $\mathbf{U}$-submodule of $M_{\tilde{\lambda}}$ for any $\tilde{\lambda}\in \tilde{X}^+$, $b_1\diamondsuit_{-w\lambda_1\odot\lambda_2}b_2$ is contained in the $\mathcal{A}$-submodule of $\dot{\tilde{\mathbf{U}}}$ generated by ${b'_1}^+{b'_2}^-1_{-w\lambda_1\odot\lambda_2}$ for any $b'_1\in \mathbf{B}_J$ and $b'_2\in \mathbf{B}(\mathbf{f}\theta_{\lambda_2}\mathbf{f})$. Consequently, we can write
$$(b_1\diamondsuit_{-w\lambda_1\odot\lambda_2}b_2)(\xi_{-\tilde{\lambda}_1}\otimes v_{\tilde \l_2})=\sum_{b'_1 \in \mathbf{B}_J, b'_2 \in \mathbf{B}(\mathbf{f}\theta_{\lambda_2}\mathbf{f})} c_{b'_1,b'_2} {b'_1}^+\xi_{-\tilde{\lambda}_1}\otimes {b'_2}^- v_{\tilde \l_2}.$$ 
Since $\langle j_{\tilde{Y}},\tilde{\lambda}_1\rangle=0$ for any $j\in J$, we have ${b'_1}^+\xi_{-\tilde{\lambda}_1}=0$ for any $b'_1 \in \mathbf{B}_J$ with $b'_1 \neq 1$. Hence the linear combination can be rewritten as 
\begin{equation}\label{linear combination}
(b_1\diamondsuit_{-w\lambda_1\odot\lambda_2}b_2)(\xi_{-\tilde{\lambda}_1}\otimes v_{\tilde \l_2})=\sum_{b \in \mathbf{B}(\mathbf{f}\theta_{\lambda_2}\mathbf{f})}c_{b}\,\xi_{-\tilde{\lambda}_1}\otimes b^- v_{\tilde \l_2}.
\end{equation}
Similarly, since $b_1\in \mathbf{B}_J$ and $b_1\neq 1$, we have $b_1^+\xi_{-\tilde{\lambda}_1}=0$. By Theorem \ref{25.2.1}, applying $\mathrm{Id}\otimes \pi_{\tilde{\lambda}_2}:{^{\omega}\Lambda_{\tilde{\lambda}_1}}\otimes M_{\tilde{\lambda}_2}\rightarrow {^{\omega}\Lambda_{\tilde{\lambda}_1}}\otimes \Lambda_{\tilde{\lambda}_2}$ to \eqref{linear combination}, we obtain
$$0=(b_1\diamondsuit_{-w\lambda_1\odot\lambda_2}b_2)(\xi_{-\tilde{\lambda}_1}\otimes \eta_{\tilde{\lambda}_2})=\sum_{b\in \mathbf{B}(\mathbf{f}\theta_{\lambda_2}\mathbf{f})}c_b\,\xi_{-\tilde{\lambda}_1}\otimes b^-\eta_{\tilde{\lambda}_2}.$$
By \cite[Theorem 14.4.11]{Lusztig-1993}, \{$b^-\eta_{\tilde{\lambda}_2};b\in \mathbf{B}(\mathbf{f}\theta_{\lambda_2}\mathbf{f}),b^-\eta_{\tilde{\lambda}_2}\neq 0\}\subset \mathbf{B}(\Lambda_{\tilde{\lambda}_2})$ is linearly independent, and so $c_b=0$ for any $b\in \mathbf{B}(\mathbf{f}\theta_{\lambda_2}\mathbf{f})$ with $b^- \eta_{\tilde \l_2} \neq 0$. Then by \eqref{factor through f}, 
\begin{align*}
&(b_1\diamondsuit_{-w\lambda_1\odot\lambda_2}b_2) v_{-w\lambda_1\odot\lambda_2} 
=\sum_{b\in \mathbf{B}(\mathbf{f}\theta_{\lambda_2}\mathbf{f});b^-\eta_{\tilde{\lambda}_2}=0} f(c_b\,\xi_{-\tilde{\lambda}_1}\otimes b^- v_{\tilde \l_2})\\
=&\sum_{b\in \mathbf{B}(\mathbf{f}\theta_{\lambda_2}\mathbf{f});b^-\eta_{\tilde{\lambda}_2}=0} c_b\,b^- f(\xi_{-\tilde{\lambda}_1}\otimes v_{\tilde \l_2})
=\sum_{b\in \mathbf{B}(\mathbf{f}\theta_{\lambda_2}\mathbf{f});b^-\eta_{\tilde{\lambda}_2}=0} c_b\,b^- v_{-w\lambda_1\odot\lambda_2}
\end{align*}
It remains to prove that it belongs to $\mathrm{ker}\,\underline{\varphi_{-w\lambda_1,\lambda_2}}$.

Let $b \in \mathbf{B}(\mathbf{f}\theta_{\lambda_2}\mathbf{f})$ with $b^-\eta_{\tilde{\lambda}_2}=0$. By \eqref{14.4.11}, $b\in \bigcup_{\tilde{i}\in \tilde{I}}(\tilde{\mathbf{B}}\cap \tilde{\mathbf{f}}\theta_{\tilde{i}}^{\langle \tilde{i}_{\tilde{Y}},\tilde{\lambda}_2\rangle+1})$. Note that $|b|\in \sum_{i\in I}\langle i_Y,\lambda_2\rangle i'+\mathbb{N}[I]$. Hence $b\not\in \bigcup_{i'\in I'}\tilde{\mathbf{f}}\theta_{i'}^{\langle i'_{\tilde{Y}},\tilde{\lambda}_2\rangle+1}=\bigcup_{i'\in I'}\tilde{\mathbf{f}}\theta_{i'}^{\langle i_Y,\lambda_2\rangle+1}$ and $b\in \tilde{\mathbf{f}}\theta_i^{\langle i_{\tilde{Y}},\tilde{\lambda}_2\rangle+1}$ for some $i\in I$. We write $b=x\theta_i^{\langle i_{\tilde{Y}},\tilde{\lambda}_2\rangle+1}$ with $x\in \tilde{\mathbf{f}}$, and write $r(x)=\sum x_1\otimes x_2,r(\theta_i^{\langle i_{\tilde{Y}},\tilde{\lambda}_2\rangle+1})=y_1\otimes y_2$ with homogeneous $x_1,x_2\in \tilde{\mathbf{f}},y_1,y_2\in \mathbf{f}$. By the definition of $\varphi_{-w\lambda_1,\lambda_2}$ and \cite[\S 3.1.5 \& Lemma 1.2.11]{Lusztig-1993}, we have 
$$\underline{\varphi_{-w\lambda_1,\lambda_2}}(b^-v_{-w\lambda_1\odot\lambda_2})
=\sum v^{|x_2|\cdot |y_1|}x_2^-y_2^-\xi_{-w\lambda_1}\otimes  \phi_{\lambda_2}^{-1}(x_1^-y_1^-\eta_{0\odot\lambda_2}).$$
Since $\langle k_{\tilde{Y}},0\odot\lambda_2\rangle=0$ for any $k\in I$, we have $y_1^-\eta_{0\odot\lambda_2}=0$ whenever $|y_1|\neq 0$. Notice that $x_2^-F_i^{\langle i_{\tilde{Y}},\tilde{\lambda}_2\rangle+1}\xi_{-w\lambda_1}=0$, and so $\underline{\varphi_{-w\lambda_1,\lambda_2}}(b^-v_{-w\lambda_1\odot\lambda_2})=0$. 
\end{proof}

\subsection{Structure constant}

Let $\zeta\in X$ and $b_1,b_2\in \mathbf{B}$. Recall that $b_1\diamondsuit_{\zeta}b_2\in 1_{\zeta+|b_1|-|b_2|}\dot{\mathbf{U}}1_{\zeta}$ 
(see \cite[Remark 25.2.4]{Lusztig-1993}). We denote $|b_1\diamondsuit_{\zeta} b_2|=|b_1|-|b_2|\in \mathbb{Z}[I]$.

Let $\sigma:\mathbf{f}\rightarrow \mathbf{f}^{\mathrm{opp}}$ be the algebra isomorphism defined by $\sigma(\theta_i)=\theta_i$ for any $i\in I$. By \cite[Theorem 14.4.3]{Lusztig-1993}, $\sigma(\mathbf{B})=\mathbf{B}$. Similarly, we have $\sigma:\tilde{\mathbf{f}}\rightarrow \tilde{\mathbf{f}}^{\mathrm{opp}}$ such that $\sigma(\tilde{\mathbf{B}})=\tilde{\mathbf{B}}$. Let $\sigma:\mathbf{U}\rightarrow \mathbf{U}^{\mathrm{opp}}$ be the algebra isomorphism defined by $\sigma(E_i)=E_i,\sigma(F_i)=F_i,\sigma(K_{\mu})=K_{-\mu}$ for any $i\in I$ and $\mu\in Y$. Then it induces a linear isomorphism ${_{\gamma}\mathbf{U}_{\zeta}}\rightarrow {_{-\zeta}\mathbf{U}_{-\gamma}}$
for any $\gamma,\zeta\in X$. Taking direct sums, we obtain an algebra automorphism $\sigma:\dot{\mathbf{U}}\rightarrow \dot{\mathbf{U}}^{\mathrm{opp}}$ (see \cite[\S 23.1.6]{Lusztig-1993}). By \cite[Theorem 4.3.2]{Kashiwara-1994}, we have $\dot{\sigma}(\dot{\mathbf{B}})=\dot{\mathbf{B}}$.

Let $(\tilde{\tilde{I}},\cdot)$ be the thickening Cartan datum of $(\tilde{I},\cdot)$, and let $\tilde{\tilde{\mathbf{f}}}$ be the associated algebra defined in \S \ref{The algebra f} with the canonical basis $\tilde{\tilde{\mathbf{B}}}$. Then there are natural imbeddings $\tilde{\mathbf{f}}\hookrightarrow \tilde{\tilde{\mathbf{f}}}$ and $\tilde{\mathbf{B}}\subset \tilde{\tilde{\mathbf{B}}}$. For any $b_1,b_2\in \tilde{\tilde{\mathbf{B}}}$, we write $b_1b_2=\sum_{b\in \tilde{\tilde{\mathbf{B}}}}m_{b_1,b_2}^b b$.

Let $\tilde{W}$ be the Weyl group associated with $(\tilde{I},\cdot)$. Then for any $\tilde{\lambda}_1,\tilde{\lambda}_2\in (\tilde{X})^+$ and $\tilde{w}\in \tilde{W}$, let $\tilde{\mathrm{th}}_{\tilde{\lambda}_1,\tilde{\lambda}_2,\tilde{w}}:\dot{\tilde{\mathbf{B}}}\rightarrow \tilde{\tilde{\mathbf{B}}}\sqcup\{0\}$ be the thickening map associated with the thickening data defined in \S \ref{Thickening map}.

\begin{theorem}\label{thm:structure-multiplication}
Let $J \subset I$ be a spherical subset. Then for any $\dot{b}\in \dot{\mathbf{B}}\cap \dot{\mathbf{U}}_J$ and $\dot{b}'\in \dot{\mathbf{B}}\cap \dot{\mathbf{U}}1_{\zeta}$, we have
$$\dot{b}\dot{b}'=\sum_{\dot{b}''\in \dot{\mathbf{B}}\cap \dot{\mathbf{U}}1_{\zeta}}m_{\sigma(\mathrm{th}_{\lambda_1,\lambda_2,w}(\dot{b}')),\widetilde{\mathrm{th}}_{\tilde{\lambda}_1,\tilde{\lambda}_2,\tilde{w}}(\sigma(\dot{\tilde{b}}))}^{\widetilde{\mathrm{th}}_{\tilde{\lambda}_1,\tilde{\lambda}_2,\tilde{w}}(\sigma(1\diamondsuit_{-w\lambda_1\odot\lambda_2}\,\mathrm{th}_{\lambda_1,\lambda_2,w}(\dot{b}'')))}\dot{b}'',$$
for sufficiently regular $\l_1, \l_2 \in X^+$, $\tilde{\lambda}_1,\tilde{\lambda}_2\in \tilde{X}^+$ and sufficiently large $w\in W,\tilde{w}\in \tilde{W}$ such that $\zeta=\lambda_2-\lambda_1,J\subset J(w)$ and $$\mathrm{th}_{\lambda_1,\lambda_2,w}(\dot{b}'),\widetilde{\mathrm{th}}_{\tilde{\lambda}_1,\tilde{\lambda}_2,\tilde{w}}(\sigma(\dot{\tilde{b}})),\widetilde{\mathrm{th}}_{\tilde{\lambda}_1,\tilde{\lambda}_2,\tilde{w}}(\sigma(1\diamondsuit_{-w\lambda_1\odot\lambda_2}\,\mathrm{th}_{\lambda_1,\lambda_2,w}(\dot{b}'')))\neq 0,$$
where $\dot{\tilde{b}}=b_1\diamondsuit_{-w\lambda_1\odot\lambda_2-|\mathrm{th}_{\lambda_1,\lambda_2,w}(\dot{b}')|_{\tilde{X}}}b_2\in \dot{\tilde{\mathbf{B}}}$ with $b_1\in \mathbf{B}_J$ and $b_2\in \mathbf{B}$ are determined by $\dot{b}=b_1\diamondsuit_{\zeta+|\dot{b}'|_X}b_2$.
\end{theorem} 
\begin{proof}
We simply denote $\mathrm{th}_{\lambda_1,\lambda_2,w}=\mathrm{th},\underline{\varphi_{-w\lambda_1,\lambda_2}}=\underline{\varphi}$ and $\widetilde{\mathrm{th}}_{\tilde{\lambda}_1,\tilde{\lambda}_2,\tilde{w}}=\widetilde{\mathrm{th}}$.

By Theorem \ref{25.2.1}, we can write $\dot{b}\dot{b}'=\sum \dot{m}_{\dot{b},\dot{b}'}^{\dot{b}''}\dot{b}''$, where the summation is taken over $\dot{b}''\in \dot{\mathbf{B}}\cap \dot{\mathbf{U}}1_{\zeta}$ and the structure constants $\dot{m}_{\dot{b},\dot{b}'}^{\dot{b}''}$ are zero for all but finitely many $\dot{b}''$. Choosing sufficiently regular $\lambda_1, \lambda_2 \in X^+$ such that $\zeta=\lambda_2-\lambda_1$ and $\dot{b}''(\xi_{-\lambda_1}\otimes \eta_{\lambda_2})\neq 0$ for any $\dot{b}''$ in the summation. Then 
\begin{equation}\label{structure constant m dot}
\dot{b}\dot{b}'(\xi_{-\lambda_1}\otimes \eta_{\lambda_2})=\sum\dot{m}_{\dot{b},\dot{b}'}^{\dot{b}''}\dot{b}''(\xi_{-\lambda_1}\otimes \eta_{\lambda_2}).    
\end{equation} 
By Theorem \ref{25.2.1}, $\{\dot{b}''(\xi_{-\lambda_1}\otimes \eta_{\lambda_2});\dot{b}'' \in \dot{\mathbf{B}},\dot{m}_{\dot{b},\dot{b}'}^{\dot{b}''}\neq 0\}$ is a linearly independent subset of $\mathbf{B}({^{\omega}\Lambda_{\lambda_1}}\otimes \Lambda_{\lambda_2})$. This uniquely isolates the structure constants $\dot{m}_{\dot{b},\dot{b}'}^{\dot{b}''}$ as the exact basis coefficients.

Since the right hand side of $\eqref{structure constant m dot}$ is non-zero, we have $\dot{b}'(\xi_{-\lambda_1}\otimes \eta_{\lambda_2})\not=0$. By Theorem \ref{25.2.1}, $\dot{b}'(\xi_{-\lambda_1}\otimes \eta_{\lambda_2})\in \mathbf{B}({^{\omega}\Lambda_{\lambda}}\otimes \Lambda_{\lambda_2})$. We fix a sufficiently large $w'\in W$ such that $\dot{b}'(\xi_{-\lambda_1}\otimes \eta_{\lambda_2})\in \mathbf{B}({^{\omega}V_{w'}(\lambda_1)}\otimes \Lambda_{\lambda_2})$, and let $w=w_J*w'\in W$ be the Demazure product of the longest element of $W_J$ and $w'$. Then we have $J\subset J(w)$ and $\dot{b}'(\xi_{-\lambda_1}\otimes \eta_{\lambda_2})\in \mathbf{B}({^{\omega}V_w(\lambda_1)}\otimes \Lambda_{\lambda_2})$. Since $\dot{b}\in \dot{\mathbf{U}}_J$, its action on ${^{\omega}\Lambda_{\lambda_1}}\otimes \Lambda_{\lambda_2}$ preserves the subspace ${^{\omega}V_w(\lambda_1)}\otimes \Lambda_{\lambda_2}$. By Proposition \ref{canonical basis of demazure otimes highest weight}, we have $\dot{b}''(\xi_{-\lambda_1}\otimes \eta_{\lambda_2})\in \mathbf{B}({^{\omega}V_w(\lambda_1)}\otimes \Lambda_{\lambda_2})$ for any $\dot{b}''$ in the summation. 

Consider the action of $\dot{b}\in \dot{\mathbf{U}}_J$ on $(\mathrm{th}(\dot{b}'))^-v_{-w\lambda_1\odot\lambda_2}\in (\mathbf{f}\theta_{\lambda_2}\mathbf{f})_{-w\lambda_1\odot\lambda_2}$. The vector space $(\mathbf{f}\theta_{\lambda_2}\mathbf{f})_{-w\lambda_1\odot\lambda_2}$ has a basis $\mathbf{B}((\mathbf{f}\theta_{\lambda_2}\mathbf{f})_{-w\lambda_1\odot\lambda_2})$, so we can write
\begin{equation}\label{structure constant a}
\dot{b}(\mathrm{th}(\dot{b}'))^-v_{-w\lambda_1\odot\lambda_2}=\sum a_{\dot{b},\mathrm{th}(\dot{b}')}^{\tilde{b}''}\tilde{b}^{''-}v_{-w\lambda_1\odot\lambda_2},  
\end{equation}
where the summation is taken over $\tilde{b}^{''-}v_{-w\lambda_1\odot\lambda_2}\in \mathbf{B}((\mathbf{f}\theta_{\lambda_2}\mathbf{f})_{-w\lambda_1\odot\lambda_2})$. By \eqref{definition of thickening map}, Lemma \ref{lem:Demazure-UJ} and Lemma \ref{isomorphism varphi}, we have 
\begin{align*}
\dot{b}\dot{b}'(\xi_{-\lambda_1}\otimes \eta_{\lambda_2})=&\dot{b}\underline{\varphi}((\mathrm{th}(\dot{b}'))^-v_{-w\lambda_1\odot\lambda_2})=\underline{\varphi}(\dot{b}(\mathrm{th}(\dot{b}'))^-v_{-w\lambda_1\odot\lambda_2})\\
=&\sum a_{\dot{b},\mathrm{th}(\dot{b}')}^{\tilde{b}''}\underline{\varphi}(\tilde{b}^{''-}v_{-w\lambda_1\odot\lambda_2}),   
\end{align*}
where the summation is taken over $\tilde{b}^{''-}v_{-w\lambda_1\odot\lambda_2}\in \mathbf{B}((\mathbf{f}\theta_{\lambda_2}\mathbf{f})_{-w\lambda_1\odot\lambda_2}) \setminus \mathrm{ker}\,\underline{\varphi}$. Comparing it with \eqref{structure constant m dot}, by Theorem \ref{thm:structure-trans} and \eqref{definition of thickening map}, we obtain the relation between the structure constants of the multiplication in $\dot{\mathbf{U}}$ and the action of $\dot{\mathbf{U}}$ on $(\mathbf{f}\theta_{\lambda_2}\mathbf{f})_{-w\lambda_1\odot\lambda_2}$. More precisely,

(a) $\dot{m}_{\dot{b},\dot{b}'}^{\dot{b}''}=a_{\dot{b},\mathrm{th}(\dot{b}')}^{\mathrm{th}(\dot{b}'')}$ for any $\dot{b}''$ in the summation $\dot{b}\dot{b}'=\sum \dot{m}_{\dot{b},\dot{b}'}^{\dot{b}''}\dot{b}''$.

In order to determine the structure constant $a_{\dot{b},\mathrm{th}(\dot{b}')}^{\mathrm{th}(\dot{b}'')}$, we first consider the special case that $\dot{b}=1\diamondsuit_{\zeta+|\dot{b}'|_X}b$ with $b\in \mathbf{B}$. The vector space $\mathbf{f}\theta_{\lambda_2}\mathbf{f}$ has a basis $\mathbf{B}(\mathbf{f}\theta_{\lambda_2}\mathbf{f})$. For the product $b\,\mathrm{th}(\dot{b}')\in \mathbf{f}\theta_{\lambda_2}\mathbf{f}$, we can write $b\,\mathrm{th}(\dot{b}')=\sum_{\tilde{b}''\in \mathbf{B}(\mathbf{f}\theta_{\lambda_2}\mathbf{f})}m_{b,\mathrm{th}(\dot{b}')}^{\tilde{b}''}\tilde{b}''$. Since
\begin{equation}\label{weight relation}
\langle i_Y,\zeta+|\dot{b}'|_X\rangle=\langle i_{\tilde{Y}},-w\lambda_1\odot\lambda_2-|\mathrm{th}(\dot{b}')|_{\tilde{X}}\rangle\ \textrm{for any}\ i\in I,   
\end{equation}
we have 
\begin{align*}
(1\diamondsuit_{\zeta+|\dot{b}'|_X}b)(\mathrm{th}(\dot{b}'))^-v_{-w\lambda_1\odot\lambda_2}=&b^-(\mathrm{th}(\dot{b}'))^-v_{-w\lambda_1\odot\lambda_2}
=(b\,\mathrm{th}(\dot{b}'))^-v_{-w\lambda_1\odot\lambda_2}\\
=&\sum_{\tilde{b}''\in \mathbf{B}(\mathbf{f}\theta_{\lambda_2}\mathbf{f})}m_{b,\mathrm{th}(\dot{b}')}^{\tilde{b''}}\tilde{b}^{''-}v_{-w\lambda_1\odot\lambda_2}.
\end{align*}
Comparing it with \eqref{structure constant a}, we obtain the relation between the structure constants of the action on $(\mathbf{f}\theta_{\lambda_2}\mathbf{f})_{-w\lambda_1\odot\lambda_2}$ by $1\diamondsuit_{\zeta+|\dot{b}'|_X}b$ and the multiplication in $\tilde{\mathbf{f}}$, that is, 
$$a_{1\diamondsuit_{\zeta+|\dot{b}'|_X}b,\mathrm{th}(\dot{b}')}^{\tilde{b}''}=m_{b,\mathrm{th}(\dot{b}')}^{\tilde{b}''}\ \textrm{for any}\ \tilde{b}''\in \mathbf{B}(\mathbf{f}\theta_{\lambda_2}\mathbf{f}).$$ 
Together with (a), we have $(1\diamondsuit_{\zeta+|\dot{b}'|_X}b)\dot{b}'=\sum m_{b,\mathrm{th}(\dot{b}')}^{\mathrm{th}(\dot{b}'')}\dot{b}''$. Applying $\sigma:\dot{\mathbf{U}}\rightarrow \dot{\mathbf{U}}^{\mathrm{opp}}$, we obtain 

(b) for any $\dot{b}\in \dot{\mathbf{B}}\cap 1_{-\zeta}\dot{\mathbf{U}},b\in \mathbf{B}$, sufficiently regular $\lambda_1,\lambda_2\in X^+$ and sufficiently large $w\in W$ such that $\zeta=\lambda_2-\lambda_1$ and $\sigma(\dot{b})(\xi_{-\lambda_1}\otimes \eta_{\lambda_2})\in \mathbf{B}({^{\omega}V_w(\lambda_1)\otimes \Lambda_{\lambda_2}})$, we have
$$\dot{b}(1\diamondsuit_{-\zeta-|\dot{b}|_X+|b|_X}b)=\sum m_{\sigma(b),\mathrm{th}(\sigma(\dot{b}))}^{\mathrm{th}(\sigma(\dot{b}''))}\dot{b}'',$$
where the summation is taken over $\dot{b}''\in \dot{\mathbf{B}}\cap 1_{-\zeta}\dot{\mathbf{U}}$.

Now we consider the general case. Applying (b) in the larger modified quantum group $\dot{\tilde{\mathbf{U}}}$ for $\dot{\tilde{b}}\in \dot{\tilde{\mathbf{B}}}\cap 1_{-w\lambda_1\odot\lambda_2-|\mathrm{th}(\dot{b}')|_{\tilde{X}}+|\dot{b}|_{\tilde{X}}}\dot{\tilde{\mathbf{U}}},\mathrm{th}(\dot{b}')\in \tilde{\mathbf{B}}$, sufficiently regular $\tilde{\lambda}_1,\tilde{\lambda}_2\in \tilde{X}^+$ and sufficiently large $\tilde{w}\in \tilde{W}$ such that $-w\lambda_1\odot\lambda_2-|\mathrm{th}(\dot{b}')|_{\tilde{X}}+|\dot{b}|_{\tilde{X}}=\tilde{\lambda}_1-\tilde{\lambda}_2$ and $\sigma(\dot{\tilde{b}})(\xi_{-\tilde{\lambda}_1}\otimes \eta_{\tilde{\lambda}_2})\in \mathbf{B}({^{\omega}V_{\tilde{w}}(\tilde{\lambda}_1)\otimes \Lambda_{\tilde{\lambda}_2}})$, we have  
$$\dot{\tilde{b}}(1\diamondsuit_{-w\lambda_1\odot\lambda_2}\,\mathrm{th}(\dot{b}'))=\sum m_{\sigma(\mathrm{th}(\dot{b}')),\widetilde{\mathrm{th}}(\sigma(\dot{\tilde{b}}))}^{\widetilde{\mathrm{th}}(\sigma(\dot{\tilde{b}}''))}\dot{\tilde{b}}'',$$
where the summation is taken over $\dot{\tilde{b}}''\in \dot{\tilde{\mathbf{B}}}\cap 1_{\tilde{\lambda}_1-\tilde{\lambda}_2}\dot{\tilde{\mathbf{U}}}$.
By the constructions of $\dot{b}=b_1\diamondsuit_{\zeta+|\dot{b}'|_X}b_2\in \dot{\mathbf{B}},\dot{\tilde{b}}=b_1\diamondsuit_{-w\lambda_1\odot\lambda_2-|\mathrm{th}(\dot{b}')|_{\tilde{X}}}b_2\in \dot{\tilde{\mathbf{B}}}$ in Theorem \ref{25.2.1} and \eqref{weight relation},   
\begin{equation}\label{the product acts on Verma}
\begin{aligned}
\dot{b}(\mathrm{th}(\dot{b}'))^-v_{-w\lambda_1\odot\lambda_2}=&\dot{\tilde{b}}(\mathrm{th}(\dot{b}'))^-v_{-w\lambda_1\odot\lambda_2}=\dot{\tilde{b}}(1\diamondsuit_{-w\lambda_1\odot\lambda_2}\,\mathrm{th}(\dot{b}'))v_{-w\lambda_1\odot\lambda_2}\\
=&\sum m_{\sigma(\mathrm{th}(\dot{b}')),\widetilde{\mathrm{th}}(\sigma(\dot{\tilde{b}}))}^{\widetilde{\mathrm{th}}(\sigma(\dot{\tilde{b}}''))}\dot{\tilde{b}}''v_{-w\lambda_1\odot\lambda_2},
\end{aligned}
\end{equation}
where the summation is taken over $\dot{\tilde{b}}''\in \dot{\tilde{\mathbf{B}}}\cap 1_{\tilde{\lambda}_1-\tilde{\lambda}_2}\dot{\tilde{\mathbf{U}}}$. Since $\dot{b}\in \dot{\mathbf{B}}\cap\dot{\mathbf{U}}_J$ and $\mathrm{th}(\dot{b}')\in \mathbf{B}(\mathbf{f}\theta_{\lambda_2}\mathbf{f})$, we have $b_1\in \mathbf{B}_J$, and the product
$$\dot{\tilde{b}}(1\diamondsuit_{-w\lambda_1\odot\lambda_2}\,\mathrm{th}(\dot{b}'))=\dot{\tilde{b}}(\mathrm{th}(\dot{b}'))^-1_{-w\lambda_1\odot\lambda_2}$$
belongs to the $\mathcal{A}$-submodule of $\dot{\tilde{\mathbf{U}}}$ generated by $\tilde{b}_1^+\tilde{b}_2^-1_{-w\lambda_1\odot\lambda_2}$ for any $\tilde{b}_1\in \mathbf{B}_J$ and $\tilde{b}_2\in \mathbf{B}(\mathbf{f}\theta_{\lambda_2}\mathbf{f})$. By the proof of Theorem \ref{25.2.1}, for any $\dot{\tilde{b}}''\in \dot{\tilde{\mathbf{B}}}$ is the summation \eqref{the product acts on Verma}, we have $\dot{\tilde{b}}''=\tilde{b}''_1\diamondsuit_{-w\lambda_1\odot\lambda_2}\tilde{b}''_2$ for some $\tilde{b}''_1\in \mathbf{B}_J,\tilde{b}''_2\in \mathbf{B}(\mathbf{f}\theta_{\lambda_2}\mathbf{f})$. The vector space $(\mathbf{f}\theta_{\lambda_2}\mathbf{f})_{-w\lambda_1\odot\lambda_2}$ has a basis $\mathbf{B}((\mathbf{f}\theta_{\lambda_2}\mathbf{f})_{-w\lambda_1\odot\lambda_2})$, so we can write \eqref{the product acts on Verma} into the linear combination of $\mathbf{B}((\mathbf{f}\theta_{\lambda_2}\mathbf{f})_{-w\lambda_1\odot\lambda_2})$. Then we compare the coefficients of $(\mathrm{th}(\dot{b}''))^-v_{-w\lambda_1\odot\lambda_2}$ in \eqref{structure constant a} and \eqref{the product acts on Verma}. Note that $(\mathrm{th}(\dot{b}''))^-v_{-w\lambda_1\odot\lambda_2}\notin \mathrm{ker}\,\underline{\varphi}$. 
By Lemma \ref{a otimes Id is a based map} and Lemma \ref{canonical basis of U dot acts on Verma}, the coefficient of $(\mathrm{th}(\dot{b}''))^-v_{-w\lambda_1\odot\lambda_2}$ in \eqref{the product acts on Verma} is $m_{\sigma(\mathrm{th}(\dot{b}')),\widetilde{\mathrm{th}}(\sigma(\dot{\tilde{b}}))}^{\widetilde{\mathrm{th}}(\sigma(1\diamondsuit_{-w\lambda_1\odot\lambda_2}\,\mathrm{th}(\dot{b}'')))}$, and so $\dot{m}_{\dot{b},\dot{b}'}^{\dot{b}''}=a_{\dot{b},\mathrm{th}(\dot{b}')}^{\mathrm{th}(\dot{b}'')}=m_{\sigma(\mathrm{th}(\dot{b}')),\widetilde{\mathrm{th}}(\sigma(\dot{\tilde{b}}))}^{\widetilde{\mathrm{th}}(\sigma(1\diamondsuit_{-w\lambda_1\odot\lambda_2}\,\mathrm{th}(\dot{b}'')))}$. 
\end{proof}

\section{Structure constant of comultiplication}\label{sec:comult}

By using the thickening realizations, we compare the structure constants of the comultiplication in the modified quantum group and the comultiplication in the negative part of a larger quantum group with respect to their canonical bases.

\subsection{Universal $\mathcal{R}$-matrix}

The universal $\mathcal{R}$-matrix is an intertwiner between the comultiplication of the quantum group and its transpose. Its construction depends on a function $f:X\times X\rightarrow \mathbb{Q}$ such that 
$$f(\zeta+\nu,\zeta'+\nu')=f(\zeta,\zeta')-\langle \nu_Y,\zeta'\rangle-\langle \nu'_Y,\zeta\rangle-\nu\cdot \nu'$$
for any $\zeta,\zeta'\in X$ and $\nu,\nu'\in \mathbb{Z}[I]$. The existence of such a function relies on the $X$-regularity of the root datum (see \cite[\S 32.1.3]{Lusztig-1993}). In this subsection, we first enlarge the fixed $Y$-regular (not necessarily $X$-regular) root datum to be a both $Y$-regular and $X$-regular root datum, and then construct the universal $\mathcal{R}$-matrix.

Let $(\mathrm{Hom}_{\mathbb{Z}}(\mathbb{Z}[I],\mathbb{Z}),\mathbb{Z}[I],\langle\,,\rangle_{\mathrm{ad}},\ldots)$ be the adjoint root datum of type $(I,\cdot)$, where $\langle\,,\rangle_{\mathrm{ad}}$ is the natural evaluation, the imbeddings $I\hookrightarrow \mathrm{Hom}_{\mathbb{Z}}(\mathbb{Z}[I],\mathbb{Z})$ and $I\hookrightarrow \mathbb{Z}[I]$ are defined by $i\mapsto (j\mapsto i\cdot j)$ and $i\mapsto i$ respectively (see \cite[\S 2.2.2]{Lusztig-1993}). It is $X$-regular. We simply denote $\mathbb{Z}[I]^{\vee}:=\mathrm{Hom}_{\mathbb{Z}}(\mathbb{Z}[I],\mathbb{Z})$. 

For the $Y$-regular root datum $(Y,X,\langle\,,\,\rangle,\ldots)$, we define its {\it $X$-‌regularization} to be the root datum $(Y^{\sharp},X^{\sharp},\langle\,,\,\rangle^{\sharp},\ldots)$, where 
\begin{align*}
&Y^{\sharp}=\mathbb{Z}[I]^{\vee}\oplus Y,\,X^{\sharp}=\mathbb{Z}[I]\oplus X,\\
&\langle (\rho,\mu),(\nu,\zeta)\rangle^{\sharp}=\langle \rho,\nu\rangle_{\mathrm{ad}}+\langle \mu,\zeta\rangle,\\
&I\hookrightarrow Y^{\sharp},\,i\mapsto (i_{\mathbb{Z}[I]^{\vee}},i_Y);\ I\hookrightarrow X^{\sharp},\,i\mapsto (i,0)
\end{align*}
for any $\rho\in \mathbb{Z}[I]^{\vee},\mu\in Y$ and $\nu\in \mathbb{Z}[I],\zeta\in X$. The root datum $(Y^{\sharp},X^{\sharp},\langle\,,\,\rangle^{\sharp},\ldots)$ is both $Y$-regular and $X$-regular (see \cite[\S 2.2.2]{Lusztig-1993}). 

\begin{lemma}\label{32.1.3}
There exists a function $f^{\sharp}:X^{\sharp}\times X^{\sharp}\rightarrow \mathbb{Z}$ such that 
\begin{equation}\label{function f}
\begin{aligned}
f^{\sharp}(\zeta_1^{\sharp}+\nu_{1X^{\sharp}},\zeta_2^{\sharp}+\nu_{2X^{\sharp}})=f^{\sharp}(\zeta_1^{\sharp},\zeta_2^{\sharp})-\langle \nu_{1Y^{\sharp}},\zeta_2^{\sharp}\rangle^{\sharp}-\langle \nu_{2Y^{\sharp}},\zeta_1^{\sharp}\rangle^{\sharp}-\nu_1\cdot \nu_2
\end{aligned}
\end{equation}
for any $\zeta_1^{\sharp},\zeta_2^{\sharp}\in X^{\sharp}$ and $\nu_1,\nu_2\in \mathbb{Z}[I]$. 
\end{lemma}
\begin{proof}
We define the function
$$f^{\sharp}:X^{\sharp}\times X^{\sharp}\rightarrow \mathbb{Z},
((\nu_1,\zeta_1),(\nu_2,\zeta_2))\mapsto -\langle \nu_{1Y},\zeta_2\rangle-\langle \nu_{2Y},\zeta_1\rangle-\nu_1\cdot \nu_2,$$
for any $\zeta_1,\zeta_2\in X$ and $\nu_1,\nu_2\in \mathbb{Z}[I]$. Then it is routine to verify \eqref{function f}. Indeed, for any $\nu'_1,\nu'_2\in \mathbb{Z}[I]$, we have $\nu'_{1X^{\sharp}}=(\nu'_1,0),\nu'_{2X^{\sharp}}=(\nu'_2,0)$, and 
\begin{align*}
&f^{\sharp}((\nu_1+\nu'_1,\zeta_1),(\nu_2+\nu'_2,\zeta_2))-f^{\sharp}((\nu_1,\zeta_1),(\nu_2,\zeta_2))\\
=&-\langle \nu'_{1Y},\zeta_2\rangle-\langle \nu'_{2Y},\zeta_1\rangle-\nu_1\cdot \nu'_2-\nu'_1\cdot \nu_2-\nu'_1\cdot \nu'_2\\
=&-\langle \nu'_{1Y^{\sharp}},(\nu_2,\zeta_2)\rangle^{\sharp}-\langle \nu'_{2Y^{\sharp}},(\nu_1,\zeta_1)\rangle^{\sharp}-\nu'_1\cdot \nu'_2,
\end{align*}
as desired.
\end{proof}

Let $\mathbf{U}^{\sharp}$ be the quantum group associated with the root datum $(Y^{\sharp},X^{\sharp},\langle\,,\,\rangle^{\sharp},\ldots)$. For any $\lambda\in X^+$, we have $\langle (i_{\mathrm{ad}},i_Y),(0,\lambda)\rangle^{\sharp}=\langle i_Y,\lambda\rangle\in \mathbb{N}$, and so $(0,\lambda)\in (X^{\sharp})^+$ is dominant. Let $\Lambda_{(0,\lambda)}$ and $^{\omega}\Lambda_{(0,\lambda)}$ be the simple highest weight module and the simple lowest weight module of $\mathbf{U}^{\sharp}$ respectively.

Let $\lambda_1,\lambda_2\in X^{+}$. We define two linear isomorphisms  
\begin{align*}
\Pi_{f^{\sharp}}:&\ {^{\omega}\Lambda_{(0,\lambda_1)}}\otimes\Lambda_{(0,\lambda_2)}\rightarrow{^{\omega}\Lambda_{(0,\lambda_1)}}\otimes\Lambda_{(0,\lambda_2)},\,m_1\otimes m_2\mapsto v^{f^{\sharp}(\zeta_1^{\sharp},\zeta_2^{\sharp})}m_1\otimes m_2,\\
\tau:&\ \Lambda_{(0,\lambda_2)}\otimes {^{\omega}\Lambda_{(0,\lambda_1)}}\rightarrow {^{\omega}\Lambda_{(0,\lambda_1)}}\otimes \Lambda_{(0,\lambda_2)},\,m_2\otimes m_1\mapsto m_1\otimes m_2
\end{align*}
for any $m_1\in {}^{\omega}\Lambda_{(0,\lambda_1)},m_2\in\Lambda_{(0,\lambda_2)}$ of weights $\zeta_1^{\sharp}, \zeta_2^{\sharp}\in X$ respectively. Let $\Theta^{\sharp}$ be the quasi-$\mathcal{R}$-matrix associated with $\mathbf{U}^{\sharp}$ defined in \cite[\S 4.1]{Lusztig-1993}. It induces a linear transformation on the tensor product ${^{\omega}\Lambda_{(0,\lambda_1)}}\otimes\Lambda_{(0,\lambda_2)}$ (see \cite[\S 24.1.1]{Lusztig-1993}). The following lemma is essentially the same as \cite[Theorem 32.1.5]{Lusztig-1993}. The statement of \cite[Theorem 32.1.5]{Lusztig-1993} has the assumption that the Cartan datum is of finite type. Nevertheless, the proof is still valid for our case because the quasi-$\mathcal{R}$-matrix $\Theta^{\sharp}$ still induces a well-defined linear transformation on the tensor product ${^{\omega}\Lambda_{(0,\lambda_1)}}\otimes\Lambda_{(0,\lambda_2)}$ and the function $f^{\sharp}$ satisfies the property \cite[\S 32.1.3(a)]{Lusztig-1993} associated with the root datum $(Y^{\sharp},X^{\sharp},\langle\,,\,\rangle^{\sharp},\ldots)$ by Lemma \ref{32.1.3}.

\begin{lemma}\label{32.1.5}
Let $\lambda_1,\lambda_2\in X^+$. Then the universal $\mathcal{R}$-matrix
\begin{align*}
{_{f^{\sharp}}\mathcal{R}_{\lambda_1,\lambda_2}^{\sharp}}:\Lambda_{(0,\lambda_2)}\otimes {^{\omega}\Lambda_{(0,\lambda_1)}}&\xrightarrow{\tau} {^{\omega}\Lambda_{(0,\lambda_1)}}\otimes \Lambda_{(0,\lambda_2)} \xrightarrow{\Pi_{f^{\sharp}}} {^{\omega}\Lambda_{(0,\lambda_1)}}\otimes \Lambda_{(0,\lambda_2)}\\
&\xrightarrow{\Theta^{\sharp}}{^{\omega}\Lambda_{(0,\lambda_1)}}\otimes \Lambda_{(0,\lambda_2)}
\end{align*}
is a $\mathbf{U}^{\sharp}$-module isomorphism with the inverse $({_{f^{\sharp}}\mathcal{R}_{\lambda_1,\lambda_2}^{\sharp}})^{-1}=\tau^{-1}\Pi_{f^{\sharp}}^{-1}\overline{\Theta^{\sharp}}$.   
\end{lemma}

For any $\mu\in Y$, we define $\rho_{\mu}\in \mathbb{Z}[I]^{\vee}$ by $\rho_{\mu}(\nu)=\langle \mu,\nu_X\rangle$ for any $\nu\in \mathbb{Z}[I]$. Consider the group homomorphisms $\alpha:Y\rightarrow Y^{\sharp},\mu\mapsto (\rho_{\mu},\mu) $ and $\beta:X^{\sharp}\rightarrow X,(\nu,\zeta)\mapsto \nu_X+\zeta$. By definition, we have $\langle \alpha(\mu),(\nu,\zeta)\rangle^{\sharp}=\langle \mu,\beta(\nu,\zeta)\rangle$ for any $\mu\in Y,\nu\in \mathbb{Z}[I],\zeta\in X$ and $\alpha(i_Y)=(i_{\mathbb{Z}[I]^{\vee}},i_Y),\beta(i,0)=i_X$ for any $i\in I$. Hence $\alpha$ and $\beta$ define a morphism of root data from $(Y,X,\langle\,,\,\rangle,\ldots)$ to $(Y^{\sharp},X^{\sharp},\langle\,,\,\rangle^{\sharp},\ldots)$. By \cite[\S 3.1.2]{Lusztig-1993}, there is a natural algebra homomorphism $\mathbf{U}\rightarrow \mathbf{U}^{\sharp}$ defined by $$E_i\mapsto E_i,F_i\mapsto F_i,K_{\mu}\mapsto K_{(\rho_{\mu},\mu)}\ \textrm{for any}\ i\in I,\mu\in Y.$$ 
By definition, this algebra homomorphism is also a bialgebra homomorphism.

Let $\lambda\in X^+$. We regard the simple highest module $\Lambda_{(0,\lambda)}$ of $\mathbf{U}^{\sharp}$ as a $\mathbf{U}$-module via the algebra homomorphism $\mathbf{U}\rightarrow \mathbf{U}^{\sharp}$. By definition, $\langle (\rho_{\mu},\mu),(0,\lambda)\rangle^{\sharp}=\langle \mu,\lambda\rangle$ for any $\mu\in Y$, the weight of $\eta_{(0,\lambda)}\in \Lambda_{(0,\lambda)}$ is $\lambda$. By \cite[Proposition 3.5.8]{Lusztig-1993}, there exists a unique $\mathbf{U}$-module homomorphism $p_{\lambda}:\Lambda_{\lambda}\rightarrow \Lambda_{(0,\lambda)}$ such that $p_{\lambda}(\eta_{\lambda})=\eta_{(0,\lambda)}$. It is an isomorphism since $\Lambda_{\lambda}$ is simple and $\mathbf{U}^-=(\mathbf{U}^{\sharp})^-$. Similarly, there is a unique $\mathbf{U}$-module isomorphism $^{\omega}p_{\lambda}:{^{\omega}\Lambda_{\lambda}}\rightarrow {^{\omega}\Lambda_{(0,\lambda)}}$ such that $({^{\omega}p_{\lambda}})(\xi_{-\lambda})=\xi_{-(0,\lambda)}$. Combining with Lemma \ref{32.1.5}, we obtain

\begin{proposition}\label{universal R matrix without X regular}
Let $\lambda_1,\lambda_2\in X^+$. Then there exits a $\mathbf{U}$-module isomorphism
\begin{align*}
{_{f^{\sharp}}\mathcal{R}_{\lambda_1,\lambda_2}}:\Lambda_{\lambda_2}\otimes {^{\omega}\Lambda_{\lambda_1}}&\xrightarrow{p_{\lambda_2}\otimes\, {^{\omega}p_{\lambda_1}}} \Lambda_{(0,\lambda_2)}\otimes {^{\omega}\Lambda_{(0,\lambda_1)}}\xrightarrow{{_{f^{\sharp}}\mathcal{R}_{\lambda_1,\lambda_2}^{\sharp}}} {^{\omega}\Lambda_{(0,\lambda_1)}}\otimes \Lambda_{(0,\lambda_2)}\\
&\xrightarrow{({^{\omega}p_{\lambda_1}})^{-1}\otimes p_{\lambda_2}^{-1}} {^{\omega}\Lambda_{\lambda_1}}\otimes \Lambda_{\lambda_2}
\end{align*}
with the inverse $({_{f^{\sharp}}\mathcal{R}_{\lambda_1,\lambda_2}})^{-1}=(p_{\lambda_2}^{-1}\otimes ({^{\omega}p_{\lambda_1}})^{-1})({_{f^{\sharp}}\mathcal{R}_{\lambda_1,\lambda_2}^{\sharp}})^{-1}({^{\omega}p_{\lambda_1}}\otimes p_{\lambda_2})$.  
\end{proposition}

\subsection{Module-theoretic incarnation of the comultiplication}\label{Tensor product approach to comultiplication}
In this subsection, we realize the comultiplication of $\dot{\mathbf{U}}$ on the level of the tensor products of simple lowest weight modules and simple highest weight modules. 

Let $\lambda',\lambda''\in X^+$ and $\lambda=\lambda'+\lambda''$. By \cite[Proposition 3.5.8]{Lusztig-1993}, there exists a $\mathbf{U}$-module homomorphism 
$\chi_{\lambda',\lambda''}:\Lambda_{\lambda}\rightarrow \Lambda_{\lambda'}\otimes \Lambda_{\lambda''}$ such that $\chi_{\lambda',\lambda''}(\eta_{\lambda})=\eta_{\lambda'}\otimes \eta_{\lambda''}$. Similarly, there exists a $\mathbf{U}$-module homomorphism $\chi_{-\lambda',-\lambda''}:{^{\omega}\Lambda_{\lambda}}\rightarrow {^{\omega}\Lambda_{\lambda'}}\otimes {^{\omega}\Lambda_{\lambda''}}$ such that 
$\chi_{-\lambda',-\lambda''}(\xi_{-\lambda})=\xi_{-\lambda'}\otimes \xi_{-\lambda''}$.

Let $\lambda'_1, \lambda''_1, \lambda'_2, \lambda''_2\in X^{+}$ and $\lambda_1=\lambda'_1+\lambda''_1$, $\lambda_2=\lambda'_2+\lambda''_2$. Inspired by \cite[\S 1.13]{Lusztig-2009}, we consider the $\mathbf{U}$-module homomorphism 
\begin{equation}\label{homomorphism g}
\begin{aligned}
g:{}^{\omega}\Lambda_{\lambda_1}\otimes\Lambda_{\lambda_2} &\xrightarrow{\chi_{-\lambda'_1,-\lambda''_1}\otimes\chi_{\lambda'_2,\lambda''_2}}{}^{\omega}\Lambda_{\lambda'_1}\otimes{}^{\omega}\Lambda_{\lambda''_1}\otimes\Lambda_{\lambda'_2}\otimes\Lambda_{\lambda''_2} \\
&\xrightarrow{\mathrm{Id}\otimes v^{f^{\sharp}(-(0,\lambda''_1),(0,\lambda'_2))}({}_{f^{\sharp}}\mathcal{R}_{\lambda''_1,\lambda'_2})^{-1}\otimes \mathrm{Id}}{}^{\omega}\Lambda_{\lambda'_1}\otimes\Lambda_{\lambda'_2}\otimes{}^{\omega}\Lambda_{\lambda''_1}\otimes\Lambda_{\lambda''_2}.
\end{aligned}
\end{equation}
By definition, we have 
\begin{align*}
({_{f^{\sharp}}\mathcal{R}_{\lambda''_1,\lambda'_2}})^{-1}(\xi_{-\lambda''_1}\otimes \eta_{\lambda'_2})=&(p_{\lambda'_2}^{-1}\otimes ({^{\omega}p_{\lambda''_1}})^{-1})\tau^{-1}\Pi_{f^{\sharp}}^{-1}\overline{\Theta^{\sharp}}({^{\omega}p_{\lambda''_1}}\otimes p_{\lambda'_2})(\xi_{-\lambda''_1}\otimes \eta_{\lambda'_2})\\
=&v^{-f^{\sharp}(-(0,\lambda''_1),(0,\lambda'_2))}(\eta_{\lambda'_2}\otimes \xi_{-\lambda''_1})
\end{align*}
and so $g(\xi_{-\lambda_1}\otimes \eta_{\lambda_2})=\xi_{-\lambda'_1}\otimes \eta_{\lambda'_2}\otimes \xi_{-\lambda''_1}\otimes \eta_{\lambda''_2}$.
This $\mathbf{U}$-module homomorphism enables us to formulate the relationship between the comultiplication of $\dot{\mathbf{U}}$ and its action on tensor products. More precisely, for any $\zeta', \zeta''\in X$, we denote 
$$\dot{\Delta}_{\zeta',\zeta''}=\bigoplus_{\gamma',\gamma''\in X}{}_{\gamma', \gamma''}\dot{\Delta}_{\zeta',\zeta''}:\bigoplus_{\gamma',\gamma''\in X}1_{\gamma'+\gamma''}\dot{\mathbf{U}}1_{\zeta'+\zeta''}\rightarrow \bigoplus_{\gamma',\gamma''\in X}(1_{\gamma'}\dot{\mathbf{U}}1_{\zeta'}\otimes 1_{\gamma''}\dot{\mathbf{U}}1_{\zeta''}).$$ 
Then for any $\lambda'_1,\lambda''_1,\lambda'_2,\lambda''_2\in X^+$ with $\zeta'=\lambda'_2-\lambda'_1$, $\zeta''=\lambda''_2-\lambda''_1$ and $\lambda_1=\lambda'_1+\lambda''_1,\lambda_2=\lambda'_2+\lambda''_2$, there exists a commutative diagram
$$\xymatrix@C=3cm{\bigoplus_{\gamma',\gamma''\in X}1_{\gamma'+\gamma''}\dot{\mathbf{U}}1_{\zeta'+\zeta''}\ar[r]^-{\dot{\Delta}_{\zeta',\zeta''}} \ar[d]_-{\pi_{\lambda_1,\lambda_2}} &\bigoplus_{\gamma',\gamma''\in X}(1_{\gamma'}\dot{\mathbf{U}}1_{\zeta'}\otimes 1_{\gamma''}\dot{\mathbf{U}}1_{\zeta''}) \ar[d]^-{\pi_{\lambda'_1,\lambda'_2}\otimes \pi_{\lambda''_1,\lambda''_2}}\\
{^{\omega}\Lambda_{\lambda_1}}\otimes \Lambda_{\lambda_2} \ar[r]^-g &{^{\omega}\Lambda_{\lambda'_1}}\otimes \Lambda_{\lambda'_2}\otimes {^{\omega}\Lambda_{\lambda''_1}}\otimes \Lambda_{\lambda''_2},}$$
where $\pi_{\lambda_1,\lambda_2}(\dot{u})=\dot{u}(\xi_{-\lambda_1}\otimes \eta_{\lambda_2})$ for any $\dot{u}\in \dot{\mathbf{U}}1_{\zeta'+\zeta''}$, and $\pi_{\lambda'_1,\lambda'_2},\pi_{\lambda''_1,\lambda''_2}$ are similar (see \cite[\S 1.13]{Lusztig-2009}). 

For any $\dot{b}\in\dot{\mathbf{B}}\cap \dot{\mathbf{U}}1_{\zeta'+\zeta''}$, we write
\begin{equation}\label{structure constant dotc}
\dot{\Delta}_{\zeta',\zeta''}(\dot{b})=\sum \dot{c}_{\dot{b}',\dot{b}''}^{\dot{b}} \dot{b}'\otimes\dot{b}'',    
\end{equation}
where the summation is taken over $\dot{b}'\in\dot{\mathbf{B}}\cap \dot{\mathbf{U}}1_{\zeta'},\dot{b}''\in\dot{\mathbf{B}}\cap \dot{\mathbf{U}}1_{\zeta''}$ and the structure constants $\dot{c}_{\dot{b}',\dot{b}''}^{\dot{b}}$ are zero for all but finitely many $\dot{b}',\dot{b}''$. By above commutative diagram, we have

\begin{lemma} \label{lem:g-comult-relation}
Let $\zeta', \zeta''\in X$ and $\dot{b}\in\dot{\mathbf{B}}\cap \dot{\mathbf{U}}1_{\zeta'+\zeta''}$. Then for any $\lambda'_1, \lambda''_1, \lambda'_2, \lambda''_2\in X^{+}$ with $\zeta'=\lambda'_2-\lambda'_1$, $\zeta''=\lambda''_2-\lambda''_1$ and  $\lambda_1=\lambda'_1+\lambda''_1,\lambda_2=\lambda'_2+\lambda''_2$, we have  
\begin{align*}
g(\dot{b}(\xi_{-\lambda_1}\otimes\eta_{\lambda_2})) = \sum \dot{c}_{\dot{b}',\dot{b}''}^{\dot{b}} (\dot{b}'(\xi_{-\lambda'_1}\otimes\eta_{\lambda'_2}))\otimes(\dot{b}''(\xi_{-\lambda''_1}\otimes\eta_{\lambda''_2})),
\end{align*}
where the summation is taken over $\dot{b}'\in\dot{\mathbf{B}}\cap \dot{\mathbf{U}}1_{\zeta'}$ and $\dot{b}''\in\dot{\mathbf{B}}\cap \dot{\mathbf{U}}1_{\zeta''}$.
\end{lemma}

\subsection{Compatibility with the thickening realization}
In this subsection, we study the compatibility of the module-theoretic incarnation of the comultiplication of $\dot{\mathbf{U}}$ established in \S \ref{Tensor product approach to comultiplication} with the thickening realization of the tensor product established in \S \ref{ftheta_lambdaf}.

\begin{lemma}\label{chi extreme weight vector}
Let $\lambda\in X^+$ and $w\in W$. Then we have 
\begin{enumerate}
\item $\langle (i_{\mathbb{Z}[I]^{\vee}},i_Y),w(0,\lambda)\rangle^{\sharp}=\langle i_Y,w\lambda\rangle$ for any $i\in I$; 
\item $p_{\lambda}(\eta_{w\lambda})=\eta_{w(0,\lambda)}$ and $({^{\omega}p_{\lambda}})(\xi_{-w\lambda})=\xi_{-w(0,\lambda)}$;
\item $\chi_{\lambda',\lambda''}(\eta_{w\lambda})=\eta_{w\lambda'}\otimes \eta_{w\lambda''}$ and $\chi_{-\lambda',-\lambda''}(\xi_{-w\lambda})=\xi_{-w\lambda'}\otimes \xi_{-w\lambda''}$.
\end{enumerate}
\end{lemma}
\begin{proof}
(1) We prove by induction on the length $l(w)$ of $w$. If $l(w)=0$, it is trivial. If $l(w)>0$, we take $j\in I$ such that $s_jw<w$. By the inductive hypothesis, $\langle (i_{\mathbb{Z}[I]^{\vee}},i_Y),s_jw(0,\lambda)\rangle^{\sharp}=\langle i_Y,s_jw\lambda\rangle$ for any $i\in I$. By definition, $
s_j(i_{\mathbb{Z}[I]^{\vee}},i_Y)=(i_{\mathbb{Z}[I]^{\vee}},i_Y)-\langle (i_{\mathbb{Z}[I]^{\vee}},i_Y),(j,0)\rangle^{\sharp}(j_{\mathbb{Z}[I]^{\vee}},j_Y\rangle
=(i_{\mathbb{Z}[I]^{\vee}},i_Y)-(i\cdot j)(j_{\mathbb{Z}[I]^{\vee}},j_Y\rangle$,
and so
\begin{align*}
&\langle (i_{\mathbb{Z}[I]^{\vee}},i_Y),w(0,\lambda)\rangle^{\sharp}=\langle s_j(i_{\mathbb{Z}[I]^{\vee}},i_Y),s_jw(0,\lambda)\rangle^{\sharp}\\
=&\langle (i_{\mathbb{Z}[I]^{\vee}},i_Y)-(i\cdot j)(j_{\mathbb{Z}[I]^{\vee}},j_Y\rangle,s_jw(0,\lambda)\rangle^{\sharp}\\
=&\langle i_Y-(i\cdot j)j_Y,s_jw\lambda\rangle=\langle s_ji_Y,s_jw\lambda\rangle=\langle i_Y,w\lambda\rangle,
\end{align*}
as desired.

(2) \& (3) We prove $p_{\lambda}(\eta_{w\lambda})=\eta_{w(0,\lambda)},\chi_{\lambda',\lambda''}(\eta_{w\lambda})=\eta_{w\lambda'}\otimes \eta_{w\lambda''}$ by induction on the length $l(w)$ of $w$, and the others can be proved similarly. If $l(w)=0$, it is trivial. If $l(w)>0$, we take $i\in I$ such that $s_iw<w$. By the inductive hypothesis, $p_{\lambda}(\eta_{s_iw\lambda})=\eta_{s_iw(0,\lambda)},\chi_{\lambda',\lambda''}(\eta_{s_iw\lambda})=\eta_{s_iw\lambda'}\otimes \eta_{s_iw\lambda''}$. By \cite[Proposition 28.1.4]{Lusztig-1993}, $\eta_{w\lambda}=F_i^{(\langle i_Y,s_iw\lambda\rangle)}\eta_{s_iw\lambda}$. Similarly, we have $\eta_{w(0,\lambda)}=F_i^{(\langle (i_{\mathbb{Z}[I]^{\vee}},i_Y),s_iw(0,\lambda)\rangle^{\sharp})}\eta_{s_iw(0,\lambda)}$ and $\eta_{w\lambda'}=F_i^{(\langle i_Y,s_iw\lambda'\rangle)}\eta_{s_iw\lambda'},\eta_{w\lambda''}=F_i^{(\langle i_Y,s_iw\lambda''\rangle)}\eta_{s_iw\lambda''}$. 

By (1), we have $\langle i_Y,s_iw\lambda\rangle=\langle (i_{\mathbb{Z}[I]^{\vee}},i_Y),s_iw(0,\lambda)\rangle^{\sharp}$ and so 
$$p_{\lambda}(\eta_{w\lambda})=F_i^{(\langle i_Y,s_iw\lambda\rangle)}p_{\lambda}(\eta_{s_iw\lambda})=F_i^{(\langle (i_{\mathbb{Z}[I]^{\vee}},i_Y),s_iw(0,\lambda)\rangle^{\sharp})}\eta_{s_iw(0,\lambda)}=\eta_{w(0,\lambda)},$$
as desired. By \cite[\S 3.1.5]{Lusztig-1993}, we have 
\begin{align*}
&\chi_{\lambda',\lambda''}(\eta_{w\lambda})=F_i^{(\langle i_Y,s_iw\lambda\rangle)}\chi_{\lambda',\lambda''}(\eta_{s_iw\lambda})=F_i^{(\langle i_Y,s_iw\lambda\rangle)}(\eta_{s_iw\lambda'}\otimes \eta_{s_iw\lambda''})\\
=&\sum_{p+q=\langle i_Y,s_iw\lambda\rangle}v^{-pq}F_i^{(p)}\eta_{s_iw\lambda'}\otimes K_{-i_Y}^pF_i^{(q)}\eta_{s_iw\lambda''}.
\end{align*}
By the proof of \cite[Lemma 39.1.2]{Lusztig-1993}, $E_i\eta_{s_iw\lambda'}=E_i\eta_{s_iw\lambda''}=0$. By \cite[Lemma 5.1.6]{Lusztig-1993}, we have $F_i^{(p)}\eta_{s_iw\lambda'}=F_{i}^{(q)}\eta_{s_iw\lambda''}=0$ for any $p>\langle i_Y,s_iw\lambda'\rangle,q>\langle i_Y,s_iw\lambda''\rangle$. Thus, the only non-zero term in the summation is given by $p=\langle i_Y,s_iw\lambda'\rangle,q=\langle i_Y,s_iw\lambda''\rangle$, and so
$$\chi_{\lambda',\lambda''}(\eta_{w\lambda})=v^{-\langle i_Y,s_iw\lambda'\rangle\langle i_Y,s_iw\lambda''\rangle}\eta_{w\lambda'}\otimes K_{-i_Y}^{\langle i_Y,s_iw\lambda'\rangle}\eta_{w\lambda''}=\eta_{w\lambda'}\otimes\eta_{w\lambda''},$$
as desired.
\end{proof}

Let $\lambda\in X^+$ and $w\in W$. By definition, we can uniquely write $$-w(0,\lambda)=-(0,\lambda)+(\nu_{w,\lambda},0)=-(0,\lambda)+(\nu_{w,\lambda})_{X^{\sharp}},\ \textrm{where}\ \nu_{w,\lambda}\in \mathbb{Z}[I].$$ 

\begin{lemma}\label{image of gvarphi}
Let $w\in W,\lambda'_1,\lambda''_1,\lambda'_2,\lambda''_2\in X^+$ and $\lambda_1=\lambda'_1+\lambda''_1,\lambda_2=\lambda'_2+\lambda''_2$. For any $y\in \mathbf{f}$, write $r(y)=\sum y_1\otimes y_2$ with homogeneous $y_1,y_2\in \mathbf{f}$, then we have 
\begin{align*}
g\underline{\varphi_{-w\lambda_1,\lambda_2}}(\theta_{\lambda_2}^-y^-v_{-w\lambda_1\odot\lambda_2})=v^{\langle (\nu_{w,\lambda''_1})_Y,\lambda'_2\rangle}
\sum &v^{\langle |y_2|_Y,w\lambda''_1\rangle-\langle |y_1|_Y,\lambda'_2\rangle}\\
&\times y_2^-\xi_{-w\lambda'_1}\otimes \eta_{\lambda'_2}\otimes y_1^-\xi_{-w\lambda''_1}\otimes \eta_{\lambda''_2}.
\end{align*}
\end{lemma}
\begin{proof}
By \eqref{im varphi}, $\underline{\varphi_{-w\lambda_1,\lambda_2}}(\theta_{\lambda_2}^-y^-v_{-w\lambda_1\odot\lambda_2})=y^-\xi_{-w\lambda_1}\otimes \eta_{\lambda_2}$. Recall that $g$ is defined by $g=(\mathrm{Id}\otimes v^{f^{\sharp}(-(0,\lambda''_1),(0,\lambda'_2))}({}_{f^{\sharp}}\mathcal{R}_{\lambda''_1,\lambda'_2})^{-1}\otimes \mathrm{Id})(\chi_{-\lambda'_1,-\lambda''_1}\otimes\chi_{\lambda'_2,\lambda''_2})$. By \cite[\S 3.1.5 \& Lemma 1.2.11]{Lusztig-1993} and Lemma \ref{chi extreme weight vector}, we have 
\begin{align*}
&(\chi_{-\lambda'_1,-\lambda''_1}\otimes\chi_{\lambda'_2,\lambda''_2})(y^-\xi_{-w\lambda_1}\otimes \eta_{\lambda_2})\\
=&\sum v^{-|y_2|\cdot |y_1|}y_2^-\xi_{-w\lambda'_1}\otimes K_{-|y_2|_Y}y_1^-\xi_{-w\lambda''_1}\otimes \eta_{\lambda'_2}\otimes \eta_{\lambda''_2}\\
=&\sum v^{\langle -|y_2|_Y,-w\lambda''_1\rangle}y_2^-\xi_{-w\lambda'_1}\otimes y_1^-\xi_{-w\lambda''_1}\otimes \eta_{\lambda'_2}\otimes \eta_{\lambda''_2}.
\end{align*}
Recall that $({}_{f^{\sharp}}\mathcal{R}_{\lambda''_1,\lambda'_2})^{-1}=(p_{\lambda'_2}^{-1}\otimes ({^{\omega}p_{\lambda''_1}})^{-1})\tau^{-1}\Pi_{f^{\sharp}}^{-1}\overline{\Theta^{\sharp}}({^{\omega}p_{\lambda''_1}}\otimes p_{\lambda'_2})$ (see Proposition \ref{universal R matrix without X regular} and Lemma \ref{32.1.5}). For any $y_1^-\xi_{-w\lambda''_1}\otimes \eta_{\lambda'_2}$ in the summation, by Lemma \ref{chi extreme weight vector},
\begin{align*}
&({^{\omega}p_{\lambda''_1}}\otimes p_{\lambda'_2})(y_1^-\xi_{-w\lambda''_1}\otimes \eta_{\lambda'_2})=y_1^-\xi_{-w(0,\lambda''_1)}\otimes \eta_{(0,\lambda'_2)},\\
&\tau^{-1}\Pi_{f^{\sharp}}^{-1}\overline{\Theta^{\sharp}}(y_1^-\xi_{-w(0,\lambda''_1)}\otimes \eta_{(0,\lambda'_2)})=v^{-f^{\sharp}(-w(0,\lambda''_1)-|y_1|_{X^\sharp},(0,\lambda'_2))}\eta_{(0,\lambda'_2)}\otimes y_1^-\xi_{-w(0,\lambda''_1)},\\
&(p_{\lambda'_2}^{-1}\otimes ({^{\omega}p_{\lambda''_1}})^{-1})(\eta_{(0,\lambda'_2)}\otimes y_1^-\xi_{-w(0,\lambda''_1)})=\eta_{\lambda'_2}\otimes y_1^-\xi_{-w\lambda''_1}.
\end{align*}
By Lemma \ref{32.1.3}, we have 
\begin{align*}
&f^{\sharp}(-w(0,\lambda''_1)-|y_1|_{X^\sharp},(0,\lambda'_2))=f^{\sharp}(-(0,\lambda''_1)+(\nu_{w,\lambda''_1})_{X^{\sharp}}-|y_1|_{X^\sharp},(0,\lambda'_2))\\
=&f^{\sharp}(-(0,\lambda''_1),(0,\lambda'_2))-\langle (\nu_{w,\lambda''_1})_{Y^{\sharp}}-|y_1|_{Y^{\sharp}},(0,\lambda'_2)\rangle^{\sharp}\\
=&f^{\sharp}(-(0,\lambda''_1),(0,\lambda'_2))-\langle (\nu_{w,\lambda''_1})_Y-|y_1|_Y,\lambda'_2\rangle.
\end{align*}
Summarizing above results together, we obtain the desired expression.
\end{proof}

\begin{lemma}\label{g and tilde g}
Let $w\in W,\lambda'_1,\lambda''_1,\lambda'_2,\lambda''_2\in X^+$ and $\lambda_1=\lambda'_1+\lambda''_1,\lambda_2=\lambda'_2+\lambda''_2$. Then there exists a $\mathbf{U}^-$-module homomorphism 
$$\tilde{g}:(\mathbf{f}\theta_{\lambda_2}\mathbf{f})_{-w\lambda_1\odot\lambda_2}\rightarrow (\mathbf{f}\theta_{\lambda'_2}\mathbf{f})_{-w\lambda'_1\odot\lambda'_2}\otimes (\mathbf{f}\theta_{\lambda''_2}\mathbf{f})_{-w\lambda''_1\odot\lambda''_2}$$ 
such that the following diagram commutes $$\xymatrix@C=2cm{(\mathbf{f}\theta_{\lambda_2}\mathbf{f})_{-w\lambda_1\odot\lambda_2} \ar@{->>}[d]_{\underline{\varphi_{-w\lambda_1,\lambda_2}}} \ar@{-->}[r]^-{\tilde{g}} &(\mathbf{f}\theta_{\lambda'_2}\mathbf{f})_{-w\lambda'_1\odot\lambda'_2}\otimes (\mathbf{f}\theta_{\lambda''_2}\mathbf{f})_{-w\lambda''_1\odot\lambda''_2} \ar@{->>}[d]^{\underline{\varphi_{-w\lambda'_1,\lambda'_2}} \otimes \underline{\varphi_{-w\lambda''_1,\lambda''_2}}}\\
{^{\omega}V_w(\lambda_1)}\otimes \Lambda_{\lambda_2} \ar@{^{(}->}[d] &{^{\omega}V_w(\lambda'_1)}\otimes \Lambda_{\lambda'_2}\otimes {^{\omega}V_w(\lambda''_1)}\otimes \Lambda_{\lambda''_2} \ar@{^{(}->}[d] \\{^{\omega}\Lambda_{\lambda_1}}\otimes \Lambda_{\lambda_2} \ar[r]^-{g} &{^{\omega}\Lambda_{\lambda'_1}}\otimes  \Lambda_{\lambda'_2}\otimes {^{\omega}\Lambda_{\lambda''_1}}\otimes \Lambda_{\lambda''_2}.}$$
\end{lemma}
\begin{proof}
We identify $(\mathbf{f}\theta_{\lambda_2}\mathbf{f})_{-w\lambda_1\odot\lambda_2}\cong \mathbf{f}\theta_{\lambda_2}\mathbf{f}$ via $z^-v_{-w\lambda_1\odot\lambda_2}\mapsto z$ for any $z\in \mathbf{f}\theta_{\lambda_2}\mathbf{f}$, and define 
\begin{align*}
\tilde{g}:\ &(\mathbf{f}\theta_{\lambda_2}\mathbf{f})_{-w\lambda_1\odot\lambda_2}\cong \mathbf{f}\theta_{\lambda_2}\mathbf{f}\hookrightarrow \tilde{\mathbf{f}}\cong \tilde{\mathbf{U}}^-\xrightarrow{\Delta} \tilde{\mathbf{U}}\otimes \tilde{\mathbf{U}}\xrightarrow{p_1\otimes p_2}\\
&\bigoplus_{\nu'\in \mathbb{N}[I]}\tilde{\mathbf{U}}(-\nu'-|\theta_{\lambda'_2}|) \otimes \bigoplus_{\nu''\in \mathbb{N}[I]}\tilde{\mathbf{U}}(-\nu''-|\theta_{\lambda''_2}|)\xrightarrow{\kappa} M_{-w\lambda'_1\odot \lambda'_2}\otimes M_{-w\lambda''_1\odot \lambda''_2},
\end{align*}
where $\Delta$ is the comultiplication of $\tilde{\mathbf{U}}$, $p_1,p_2$ are natural projection from $\tilde{\mathbf{U}}$ to the direct sums of its certain homogeneous components, and $\kappa$ is defined by
$$\kappa(u'\otimes u'')=v^{\langle (\nu_{w,\lambda''_1})_Y,\lambda'_2\rangle}u'v_{-w\lambda'_1\odot \lambda'_2}\otimes u''v_{-w\lambda''_1\odot \lambda''_2}\ \textrm{for any}\ u',u''\in \tilde{\mathbf{U}}.$$
By definition, $\tilde{g}$ is a $\mathbf{U}^-$-module homomorphism.

Note that $(\mathbf{f}\theta_{\lambda_2}\mathbf{f})_{-w\lambda_1\odot\lambda_2}$ is generated by $\theta_{\lambda_2}^-y^-v_{-w\lambda_1\odot\lambda_2}$ for any $y\in \mathbf{f}$ as a $\mathbf{U}^-$-module. We write $r(y)=\sum y_1\otimes y_2$ with homogeneous $y_1,y_2\in \mathbf{f}$. By \cite[Lemma 1.4.2 \& \S 3.1.5 \& Lemma 1.2.11]{Lusztig-1993}, we have
\begin{align*}
\tilde{g}(\theta_{\lambda_2}^-y^-v_{-w\lambda_1\odot\lambda_2})=&v^{\langle (\nu_{w,\lambda''_1})_Y,\lambda'_2\rangle}\sum v^{\sum_{i\in I}\langle i_Y,\lambda'_2\rangle\langle i_Y,\lambda''_2\rangle-|\theta_{\lambda'_2}|\cdot |\theta_{\lambda''_2}|-|y_2|\cdot |y_1|}\\
& \times \theta_{\lambda'_2}^-y_2^-v_{-w\lambda'_1\odot\lambda'_2}\otimes K_{-|\theta_{\lambda'_2}|_{\tilde{Y}}}\theta_{\lambda''_2}^-K_{-|y_2|_{\tilde{Y}}}y_1^-v_{-w\lambda''_1\odot\lambda''_2}.
\end{align*}
We calculate $K_{-|\theta_{\lambda'_2}|_{\tilde{Y}}}\theta_{\lambda''_2}^-K_{-|y_2|_{\tilde{Y}}}y_1^-v_{-w\lambda''_1\odot\lambda''_2}$ as follows. Moving $K_{-|\theta_{\lambda'_2}|_{\tilde{Y}}},K_{-|y_2|_{\tilde{Y}}}$ past $\theta_{\lambda''_2}^-,y_1^-$ by the commutative relations
\begin{align*}
&K_{-|\theta_{\lambda'_2}|_{\tilde{Y}}}\theta_{\lambda''_2}^-=v^{|\theta_{\lambda'_2}|\cdot |\theta_{\lambda''_2}|}\theta_{\lambda''_2}^-K_{-|\theta_{\lambda'_2}|_{\tilde{Y}}},
K_{-|\theta_{\lambda'_2}|_{\tilde{Y}}}y_1^-=v^{|\theta_{\lambda'_2}|\cdot |y_1|}y_1^-K_{-|\theta_{\lambda'_2}|_{\tilde{Y}}},\\
&K_{-|y_2|_{\tilde{Y}}}y_1^-=v^{|y_2|\cdot|y_1|}y_1^-K_{-|y_2|_{\tilde{Y}}}.
\end{align*}
Since $i'\cdot j=-\delta_{i,j}$ for any $i,j\in I$ in the thickening Cartan datum, we have $|\theta_{\lambda'_2}|\cdot |y_1|=-\langle|y_1|_Y,\lambda'_2\rangle$, and so $K_{-|\theta_{\lambda'_2}|_{\tilde{Y}}}y_1^-=v^{-\langle|y_1|_Y,\lambda'_2\rangle}y_1^-K_{-|\theta_{\lambda'_2}|_{\tilde{Y}}}$. The action of $K_{-|\theta_{\lambda'_2}|_{\tilde{Y}}}K_{-|y_2|_{\tilde{Y}}}$ on $v_{-w\lambda''_1\odot\lambda''_2}$ is
$$K_{-|\theta_{\lambda'_2}|_{\tilde{Y}}}K_{-|y_2|_{\tilde{Y}}}v_{-w\lambda''_1\odot\lambda''_2}=v^{\langle -|y_2|_{\tilde{Y}}-|\theta_{\lambda'_2}|_{\tilde{Y}},-w\lambda''_1\odot\lambda''_2\rangle}v_{-w\lambda''_1\odot\lambda''_2}.$$
By \eqref{definition of odot}, we have $\langle-|y_2|_{\tilde{Y}}-|\theta_{\lambda'_2}|_{\tilde{Y}},-w\lambda''_1\odot\lambda''_2\rangle=\langle|y_2|_Y,w\lambda''_1\rangle-\sum_{i\in I}\langle i_Y,\lambda'_2\rangle\langle i_Y,\lambda''_2\rangle$, and so $K_{-|\theta_{\lambda'_2}|_{\tilde{Y}}}K_{-|y_2|_{\tilde{Y}}}v_{-w\lambda''_1\odot\lambda''_2}
=v^{\langle |y_2|_Y,w\lambda''_1\rangle-\sum_{i\in I}\langle i_Y,\lambda'_2\rangle\langle i_Y,\lambda''_2\rangle}v_{-w\lambda''_1\odot\lambda''_2}$. Hence
\begin{equation}\label{im tilde g}
\begin{aligned}
\tilde{g}(\theta_{\lambda_2}^-y^-v_{-w\lambda_1\odot\lambda_2})
=&v^{\langle (\nu_{w,\lambda''_1})_Y,\lambda'_2\rangle}\sum v^{\langle |y_2|_Y,w\lambda''_1\rangle-\langle |y_1|_Y,\lambda'_2\rangle}\\
& \times \theta_{\lambda'_2}^-y_2^-v_{-w\lambda'_1\odot\lambda'_2}\otimes \theta_{\lambda''_2}^-y_1^-v_{-w\lambda''_1\odot\lambda''_2}\\
\in &(\mathbf{f}\theta_{\lambda'_2}\mathbf{f})_{-w\lambda'_1\odot\lambda'_2}\otimes (\mathbf{f}\theta_{\lambda''_2}\mathbf{f})_{-w\lambda''_1\odot\lambda''_2},   
\end{aligned}
\end{equation}
and so $\mathrm{im}\,\tilde{g}\subset(\mathbf{f}\theta_{\lambda'_2}\mathbf{f})_{-w\lambda'_1\odot\lambda'_2}\otimes (\mathbf{f}\theta_{\lambda''_2}\mathbf{f})_{-w\lambda''_1\odot\lambda''_2}$. 

Since all maps $\underline{\varphi_{-w\lambda_1,\lambda_2}},\underline{\varphi_{-w\lambda'_1,\lambda'_2}}\otimes\underline{\varphi_{-w\lambda''_1,\lambda''_2}},g,\tilde{g}$ in the diagram are $\mathbf{U}^-$-module homomorphisms, to prove $g\underline{\varphi_{-w\lambda_1,\lambda_2}}=(\underline{\varphi_{-w\lambda'_1,\lambda'_2}}\otimes\underline{\varphi_{-w\lambda''_1,\lambda''_2}})\tilde{g}$, it suffices to verify it on the $\mathbf{U}^-$-module generators of $(\mathbf{f}\theta_{\lambda_2}\mathbf{f})_{-w\lambda_1\odot\lambda_2}$, that is, $\theta_{\lambda_2}^-y^-v_{-w\lambda_1\odot\lambda_2}$ for any $y\in \mathbf{f}$. By \eqref{im tilde g} and \eqref{im varphi}, we have 
\begin{align*}
&(\underline{\varphi_{-w\lambda'_1,\lambda'_2}}\otimes\underline{\varphi_{-w\lambda''_1,\lambda''_2}})\tilde{g}(\theta_{\lambda_2}^-y^-v_{-w\lambda_1\odot\lambda_2})\\
=&v^{\langle (\nu_{w,\lambda''_1})_Y,\lambda'_2\rangle}\sum v^{\langle |y_2|_Y,w\lambda''_1\rangle-\langle |y_1|_Y,\lambda'_2\rangle} y_2^-\xi_{-w\lambda'_1}\otimes \eta_{\lambda'_2}\otimes y_1^-\xi_{-w\lambda''_1}\otimes \eta_{\lambda''_2}.
\end{align*}
It coincides with $g\underline{\varphi_{-w\lambda_1,\lambda_2}}(\theta_{\lambda_2}^-y^-v_{-w\lambda_1\odot\lambda_2})$ by Lemma \ref{image of gvarphi}, as desired.
\end{proof}

For any $\tilde{b}\in \tilde{\mathbf{B}}$, we write $r(\tilde{b})=\sum_{\tilde{b}',\tilde{b}''\in \tilde{\mathbf{B}}}c^{\tilde{b}}_{\tilde{b}',\tilde{b}''}\tilde{b}'\otimes \tilde{b}''$.

\begin{theorem}\label{thm:str-const-comult}
Let $\zeta',\zeta''\in X$ and $\dot{b}\in \dot{\mathbf{B}}\cap \dot{\mathbf{U}}1_{\zeta'+\zeta''}$. Then we have 
\begin{align*}
\dot{\Delta}_{\zeta',\zeta''}(\dot{b})=v^{\langle (\nu_{w,\lambda''_1})_Y,\lambda'_2\rangle}\sum v^{\langle -|\mathrm{th}_{\lambda'_1,\lambda'_2,w}(\dot{b}')|_{\tilde{Y}},-w\lambda''_1\odot\lambda''_2\rangle} c^{\mathrm{th}_{\lambda_1,\lambda_2,w}(\dot{b})}_{\mathrm{th}_{\lambda''_1,\lambda''_2,w}(\dot{b}''),\mathrm{th}_{\lambda'_1,\lambda'_2,w}(\dot{b}')}\dot{b}'\otimes \dot{b}'',
\end{align*}
where the summation is taken over $\dot{b}'\in\dot{\mathbf{B}}\cap \dot{\mathbf{U}}1_{\zeta'}$ and $\dot{b}''\in\dot{\mathbf{B}}\cap \dot{\mathbf{U}}1_{\zeta''}$, for sufficiently regular $\lambda'_1,\lambda''_1,\lambda'_2,\lambda''_2\in X^+$ and sufficiently large $w\in W$ such that 
$\mathrm{th}_{\lambda_1,\lambda_2,w}(\dot{b}),\mathrm{th}_{\lambda'_1,\lambda'_2,w}(\dot{b}'),\mathrm{th}_{\lambda''_1,\lambda''_2,w}(\dot{b}'')\neq 0$.
\end{theorem}

\begin{proof}
We retain the notation from \eqref{structure constant dotc}, that is, $\dot{\Delta}_{\zeta',\zeta''}(\dot{b})=\sum\dot{c}_{\dot{b}',\dot{b}''}^{\dot{b}}\dot{b}'\otimes\dot{b}''$ and $\dot{c}_{\dot{b}',\dot{b}''}^{\dot{b}}\neq 0$ for any $\dot{b}',\dot{b}''$ in the summation. Choosing sufficiently regular $\lambda'_1,\lambda''_1,\lambda'_2,\lambda''_2\in X^+$ such that $\zeta'=\lambda'_2-\lambda'_1,\zeta''=\lambda''_2-\lambda''_1$ and $\dot{b}'(\xi_{-\lambda'_1}\otimes\eta_{\lambda'_2})\neq 0,\dot{b}''(\xi_{-\lambda''_1}\otimes\eta_{\lambda''_2})\neq 0$ for any $\dot{b}',\dot{b}''$ in the summation. By Lemma \ref{lem:g-comult-relation}, we have 
\begin{equation} \label{eq:b-action-match}
g(\dot{b}(\xi_{-\lambda_1}\otimes\eta_{\lambda_2})) = \sum \dot{c}_{\dot{b}',\dot{b}''}^{\dot{b}} (\dot{b}'(\xi_{-\lambda'_1}\otimes\eta_{\lambda'_2}))\otimes(\dot{b}''(\xi_{-\lambda''_1}\otimes\eta_{\lambda''_2})).
\end{equation}
By Theorem \ref{25.2.1}, $\{(\dot{b}'(\xi_{-\lambda'_1}\otimes \eta_{\lambda'_2}))\otimes (\dot{b}''(\xi_{-\lambda''_1}\otimes \eta_{\lambda''_2}));\dot{b}',\dot{b}''\in \dot{\mathbf{B}},\dot{c}^{\dot{b}}_{\dot{b}',\dot{b}''}\neq 0\}$ is a linearly independent subset of the basis $\mathbf{B}({{^\omega}\Lambda_{\lambda'_1}}\otimes \Lambda_{\lambda'_2})\otimes  \mathbf{B}({{^\omega}\Lambda_{\lambda''_1}}\otimes \Lambda_{\lambda''_2})$. This uniquely isolates the structure constants $\dot{c}_{\dot{b}',\dot{b}''}^{\dot{b}}$ as the exact basis coefficients.

Since the right hand side of \eqref{eq:b-action-match} is non-zero, we have $\dot{b}(\xi_{-\lambda_1}\otimes\eta_{\lambda_2})\ne0$. By Theorem \ref{25.2.1}, $\dot{b}(\xi_{-\lambda_1}\otimes\eta_{\lambda_2})\in \mathbf{B}({}^{\omega}\Lambda_{\lambda_1}\otimes\Lambda_{\lambda_2})$. We fix a sufficiently large $w\in W$ such that $\dot{b}(\xi_{-\lambda_1}\otimes\eta_{\lambda_2})\in \mathbf{B}({}^{\omega}V_{w}(\lambda_1)\otimes\Lambda_{\lambda_2})$. By Lemma \ref{chi extreme weight vector}, $\chi_{-\lambda'_1,-\lambda''_1}({}^{\omega}V_{w}(\lambda_1))\subset {}^{\omega}V_{w}(\lambda'_1)\otimes {}^{\omega}V_{w}(\lambda''_1)$. Together with the fact that the quasi-$\mathcal{R}$-matrix $\Theta^{\sharp}$ preserves the subspace $({^{\omega}p_{\lambda''_1}}\otimes p_{\lambda'_2})({}^{\omega}V_w(\lambda''_1)\otimes\Lambda_{\lambda'_2})\subset {^{\omega}\Lambda_{(0,\lambda''_1)}\otimes \Lambda_{(0,\lambda'_2)}}$ by similar argument as the proof of Proposition \ref{canonical basis of demazure otimes highest weight}, we have $$g(\dot{b}(\xi_{-\lambda_1}\otimes\eta_{\lambda_2}))\in{}^{\omega}V_{w}(\lambda'_1)\otimes\Lambda_{\lambda'_2}\otimes{}^{\omega}V_{w}(\lambda''_1)\otimes\Lambda_{\lambda''_2}.$$ 
Consequently, by Proposition \ref{canonical basis of demazure otimes highest weight}, we have $\dot{b}'(\xi_{-\lambda'_1}\otimes\eta_{\lambda'_2})\in \mathbf{B}({}^{\omega}V_{w}(\lambda'_1)\otimes\Lambda_{\lambda'_2})$ and $\dot{b}''(\xi_{-\lambda''_1}\otimes\eta_{\lambda''_2})\in \mathbf{B}({}^{\omega}V_{w}(\lambda''_1)\otimes\Lambda_{\lambda''_2})$ for all terms in the summation of \eqref{eq:b-action-match}.

Let $\tilde{b}=\mathrm{th}_{\lambda_1,\lambda_2,w}(\dot{b})\in \mathbf{B}(\mathbf{f}\theta_{\lambda_2}\mathbf{f})$. By \eqref{definition of thickening map} and Lemma \ref{g and tilde g}, we have 
\begin{align*}
g(\dot{b}(\xi_{-\lambda_1}\otimes \eta_{\lambda_2}))=g\underline{\varphi_{-w\lambda_1,\lambda_2}}(\tilde{b}^-v_{-w\lambda_1\odot\lambda_2})
=(\underline{\varphi_{-w\lambda'_1,\lambda'_2}}\otimes \underline{\varphi_{-w\lambda''_1,\lambda''_2}})\tilde{g}(\tilde{b}^-v_{-w\lambda_1\odot\lambda_2}).
\end{align*}
Then by the definition of $\tilde{g}$ and \cite[\S 3.1.5 \& Lemma 1.2.11]{Lusztig-1993}, we have 
\begin{align*}
\tilde{g}(\tilde{b}^-v_{-w\lambda_1\odot\lambda_2})=&v^{\langle (\nu_{w,\lambda''_1})_Y,\lambda'_2\rangle}\sum v^{-|\tilde{b}'|\cdot |\tilde{b}''|}c^{\tilde{b}}_{\tilde{b}'',\tilde{b}'}\tilde{b}'^-v_{-w\lambda'_1\odot\lambda'_2}\otimes K_{-|\tilde{b}'|_{\tilde{Y}}}\tilde{b}''^-v_{-w\lambda''_1\odot\lambda''_2}\\
=&v^{\langle (\nu_{w,\lambda''_1})_Y,\lambda'_2\rangle}\sum v^{\langle -|\tilde{b}'|_{\tilde{Y}},-w\lambda''_1\odot\lambda''_2\rangle}c^{\tilde{b}}_{\tilde{b}'',\tilde{b}'}\tilde{b}'^-v_{-w\lambda'_1\odot\lambda'_2}\otimes \tilde{b}''^-v_{-w\lambda''_1\odot\lambda''_2},
\end{align*}
where the summation is taken over $\tilde{b}',\tilde{b}''\in \tilde{\mathbf{B}}$ with $|\tilde{b}'|-|\theta_{\lambda'_2}|,|\tilde{b}''|-|\theta_{\lambda''_2}|\in \mathbb{N}[I]$. Indeed, the summation is taken over $\tilde{b}'\in \mathbf{B}(\mathbf{f}\theta_{\lambda'_2}\mathbf{f}),\tilde{b}''\in \mathbf{B}(\mathbf{f}\theta_{\lambda''_2}\mathbf{f})$, since $\mathrm{im}\,\tilde{g}\subset (\mathbf{f}\theta_{\lambda'_2}\mathbf{f})_{-w\lambda'_1\odot\lambda'_2}\otimes (\mathbf{f}\theta_{\lambda''_2}\mathbf{f})_{-w\lambda''_1\odot\lambda''_2}$ and $\mathbf{B}(\mathbf{f}\theta_{\lambda'}\mathbf{f})=\tilde{\mathbf{B}}\cap (\mathbf{f}\theta_{\lambda'}\mathbf{f}),\mathbf{B}(\mathbf{f}\theta_{\lambda''}\mathbf{f})=\tilde{\mathbf{B}}\cap (\mathbf{f}\theta_{\lambda''}\mathbf{f})$.
Hence we have
\begin{align*}
g(\dot{b}(\xi_{-\lambda_1}\otimes \eta_{\lambda_2}))=&v^{\langle (\nu_{w,\lambda''_1})_Y,\lambda'_2\rangle}\sum v^{\langle -|\tilde{b}'|_{\tilde{Y}},-w\lambda''_1\odot\lambda''_2\rangle}c^{\tilde{b}}_{\tilde{b}'',\tilde{b}'}\\
&\times (\underline{\varphi_{-w\lambda'_1,\lambda'_2}}(\tilde{b}'^-v_{-w\lambda'_1\odot\lambda'_2})\otimes \underline{\varphi_{-w\lambda''_1,\lambda''_2}}(\tilde{b}''^-v_{-w\lambda'_1\odot\lambda''_2})),
\end{align*}
where the summation is taken over $\tilde{b}'^-v_{-w\lambda'_1\odot\lambda'_2}\in \mathbf{B}((\mathbf{f}\theta_{\lambda'_2}\mathbf{f})_{-w\lambda'_1\odot\lambda'_2})\setminus\mathrm{ker}\,\underline{\varphi_{-w\lambda'_1,\lambda'_2}}$ and $\tilde{b}''^-v_{-w\lambda''_1\odot\lambda''_2}\in \mathbf{B}((\mathbf{f}\theta_{\lambda''_2}\mathbf{f})_{-w\lambda''_1\odot\lambda''_2})\setminus\mathrm{ker}\,\underline{\varphi_{-w\lambda''_1,\lambda''_2}}$. Comparing it with \eqref{eq:b-action-match}, by Theorem \ref{thm:structure-trans} and \eqref{definition of thickening map}, we obtain 
\begin{equation}\label{two structure constants}
\dot{c}^{\dot{b}}_{\dot{b}',\dot{b}''}=v^{\langle (\nu_{w,\lambda''_1})_Y,\lambda'_2\rangle} v^{\langle -|\mathrm{th}_{\lambda'_1,\lambda'_2,w}(\dot{b}')|_{\tilde{Y}},-w\lambda''_1\odot\lambda''_2\rangle}c^{\mathrm{th}_{\lambda_1,\lambda_2,w}(\dot{b})}_{\mathrm{th}_{\lambda''_1,\lambda''_2,w}(\dot{b}''),\mathrm{th}_{\lambda'_1,\lambda'_2,w}(\dot{b}')},
\end{equation}
as desired.
\end{proof}

Let $\dot{b}\in \dot{\mathbf{B}}$ and $\lambda_1,\lambda_2\in X^+$ with $\dot{b}(\xi_{-\lambda_1}\otimes \eta_{\lambda_2})\in \mathbf{B}({^{\omega}\Lambda_{\lambda_1}}\otimes \Lambda_{\lambda_2})$. Consider the special case $\zeta'=-\lambda_1$ and $\zeta''=\lambda_2$ in Theorem \ref{thm:str-const-comult}. For any sufficiently large $w\in W$, we have $\dot{\Delta}_{-\lambda_1,\lambda_2}(\dot{b})=\sum c^{\mathrm{th}_{\lambda_1,\lambda_2,w}(\dot{b})}_{\mathrm{th}_{0,\lambda_2,w}(\dot{b}''),\mathrm{th}_{\lambda_1,0,w}(\dot{b}')}\dot{b}'\otimes \dot{b}''$.
Then 
$$\dot{b}(\xi_{-\lambda_1}\otimes \eta_{\lambda_2})=\sum c^{\mathrm{th}_{\lambda_1,\lambda_2,w}(\dot{b})}_{\mathrm{th}_{0,\lambda_2,w}(\dot{b}''),\mathrm{th}_{\lambda_1,0,w}(\dot{b}')}\dot{b}'\xi_{-\lambda_1}\otimes \dot{b}''\eta_{\lambda_2}.$$
For any non-zero terms in the summation, by Theorem \ref{25.2.1} and definition, if $\dot{b}'\xi_{-\lambda_1}={b'_1}^-\xi_{-w\lambda_1}\in \mathbf{B}({^{\omega}V_w(\lambda_1)}),\dot{b}''\eta_{\lambda_2}=b_2^-\eta_{\lambda_2}\in \mathbf{B}(\Lambda_{\lambda_2})$, then $\mathrm{th}_{\lambda_1,0,w}(\dot{b}')=b'_1, \mathrm{th}_{0,\lambda_2,w}(\dot{b}'')=b_2\theta_{\lambda_2}$. Hence
$$\dot{b}(\xi_{-\lambda_1}\otimes \eta_{\lambda_2})=\sum c^{\mathrm{th}_{\lambda_1,\lambda_2,w}(\dot{b})}_{b_2\theta_{\lambda_2},b'_1}{b'_1}^-\xi_{-w\lambda_1}\otimes b_2^-\eta_{\lambda_2}.$$
This is the relation between the transition matrix from the pure tensor basis $\mathbf{B}({^{\omega}\Lambda_{\lambda_1}})\otimes \mathbf{B}(\Lambda_{\lambda_2})$ to the canonical basis $\mathbf{B}({^{\omega}\Lambda_{\lambda_1}}\otimes \Lambda_{\lambda_2})$ and the structure constant of the comultiplication in $\tilde{\mathbf{U}}^-$ which has been established in Theorem \ref{thm:structure-trans} (2).

\subsection{Example}
We present the example for $\mathfrak{sl}_2$ and retain the notations in Example \ref{Example of thickening map}. Recall that $\dot{\mathbf{B}}=\{\theta_i^{(k)}\diamondsuit_t\,\theta_i^{(l)};t\in \mathbb{Z},k,l\in \mathbb{N}\}$. We only consider the case $t\leqslant l-k$, and the case $t>l-k$ is similar. In this case, $\theta_i^{(k)}\diamondsuit_t\,\theta_i^{(l)}=E_i^{(k)}F_i^{(l)}1_t\in 1_{t+2k-2l}\dot{\mathbf{U}}1_t$. Let $s',s'',t',t''\in \mathbb{Z}$ with $t+2k-2l=s'+s'',t=t'+t''$. We express ${_{s',s''}\dot{\Delta}_{t',t''}}(\theta_i^{(k)}\diamondsuit_t\,\theta_i^{(l)})$ as the linear combination of $\dot{\mathbf{B}}\otimes \dot{\mathbf{B}}$.

Let $\mathcal{S}=\{(k',k'',l',l'')\in \mathbb{N}^4;k'+k''=k,l'+l''=l,t'+2k'-2l'=s',t''+2k''-2l''=s''\}$. For any $(k',k'',l',l'')\in \mathcal{S}$, we set
$$\alpha_{k',k'',l',l''}=k'k''+l'l''-2k''l'+k''t'-l't''.$$
Then we divide $\mathcal{S}=\mathcal{S}_1\sqcup \mathcal{S}_2\sqcup \mathcal{S}_3$ into the union of subsets, where elements in $\mathcal{S}_1$ satisfy $t'\leqslant l'-k',t''\leqslant l''-k''$, elements in $\mathcal{S}_2$ satisfy $t'>l'-k',t''\leqslant l''-k''$, and elements in $\mathcal{S}_3$ satisfy $t'\leqslant l'-k',t''> l''-k''$.
By definition and \cite[\S 23.1.3]{Lusztig-1993}, 
\begin{align*}
&{_{s',s''}\dot{\Delta}_{t',t''}}(\theta_i^{(k)}\diamondsuit_t\,\theta_i^{(l)})=\sum_{\mathcal{S}} v^{\alpha_{k',k'',l',l''}} (E_i^{(k')}F_i^{(l')}1_{t'}) \otimes  (E_i^{(k'')}F_i^{(l'')}1_{t''})\\
=&\sum_{\mathcal{S}_1} v^{\alpha_{k',k'',l',l''}}(\theta_i^{(k')}\diamondsuit_{t'}\theta_i^{(l')})\otimes (\theta_i^{(k'')}\diamondsuit_{t''}\theta_i^{(l'')})\\
+&\sum_{\mathcal{S}_2} v^{\alpha_{k',k'',l',l''}}\sum_{a'\geqslant 0}\begin{bmatrix}
t'-l'+k'\\ a'   
\end{bmatrix}(\theta_i^{(k'-a')}\diamondsuit_{t'}\theta_i^{(l'-a')})\otimes (\theta_i^{(k'')}\diamondsuit_{t''}\theta_i^{(l'')})\\
+&\sum_{\mathcal{S}_3} v^{\alpha_{k',k'',l',l''}}\sum_{a''\geqslant 0}\begin{bmatrix}
t''-l''+k''\\ a''   
\end{bmatrix}(\theta_i^{(k')}\diamondsuit_{t'}\theta_i^{(l')})\otimes (\theta_i^{(k''-a'')}\diamondsuit_{t''}\theta_i^{(l''-a'')}),
\end{align*}
where we denote $\begin{bmatrix}
q\\ p    
\end{bmatrix}=0$ for $q<p$ in $\mathbb{N}$ and $\theta_i^{(n)}=0$ for $n<0$.

Recall that  
$$\mathrm{th}_{m,n,-1}(\theta_i^{(k)}\diamondsuit_{n-m}\,\theta_i^{(l)})=\begin{cases}
\theta_{i'}^{(n-l)}\theta_i^{(m-k+l)}\theta_{i'}^{(l)}\ &\textrm{if}\ n-m\leqslant l-k;\\
\theta_i^{(l)}\theta_{i'}^{(n)}\theta_i^{(m-k)}\ &\textrm{if}\ n-m\geqslant l-k.
\end{cases}$$
We take large enough $m',m'',n',n''\in \mathbb{N}$ with $t'=n'-m',t''=n''-m''$, and let $n=n'+n'',m=m'+m''$. Then $t=n-m$. We express $r(\mathrm{th}_{m,n,-1}(\theta_i^{(k)}\diamondsuit_{n-m}\,\theta_i^{(l)}))$ as the linear combination of $\tilde{\mathbf{B}}\otimes \tilde{\mathbf{B}}$.

Let $\tilde{\mathcal{S}}=\{(p',p'',q',q'',r',r'')\in \mathbb{N}^6;p'+p''=n-l,q'+q''=m-k+l,r'+r''=l\}$. For any $(p',p'',q',q'',r',r'')\in \tilde{\mathcal{S}}$, we set 
$$\beta_{p',p'',q',q'',r',r''}=p'p''+q'q''+r'r''-p''q'+2p''r'-q''r'.$$
Then we divide $\tilde{\mathcal{S}}=\tilde{\mathcal{S}}_1\sqcup \tilde{\mathcal{S}}_2\sqcup \tilde{\mathcal{S}}_3$ into the union of subsets, where elements in $\tilde{\mathcal{S}}_1$ satisfy $q'\geqslant p'+r',q''\geqslant p''+r''$, elements in $\tilde{\mathcal{S}}_2$ satisfy $q'< p'+r',q''\geqslant p''+r''$, and elements in $\tilde{\mathcal{S}}_3$ satisfy $q'\geqslant p'+r',q''< p''+r''$. By definition and \cite[Lemma 1.4.2\& Lemma 42.1.2(d)]{Lusztig-1993},
\begin{align*}
&r(\mathrm{th}_{m,n,-1}(\theta_i^{(k)}\diamondsuit_{n-m}\,\theta_i^{(l)}))=\sum_{\tilde{\mathcal{S}}_1} v^{\beta_{p',p'',q',q'',r',r''}} \theta_{i'}^{(p')}\theta_i^{(q')}\theta_{i'}^{(r')}\otimes \theta_{i'}^{(p'')}\theta_i^{(q'')}\theta_{i'}^{(r'')}\\
&+\sum_{\tilde{\mathcal{S}}_2} v^{\beta_{p',p'',q',q'',r',r''}}\sum_{a'\geqslant 0} \begin{bmatrix}
p'+r'-q'\\a'
\end{bmatrix} \theta_i^{(r'-a')}\theta_{i'}^{(p'+r')}\theta_i^{(q'-r'+a')}\otimes \theta_{i'}^{(p'')}\theta_i^{(q'')}\theta_{i'}^{(r'')}\\
&+\sum_{\tilde{\mathcal{S}}_3} v^{\beta_{p',p'',q',q'',r',r''}}\sum_{a''\geqslant 0} \begin{bmatrix}
p''+r''-q''\\a''
\end{bmatrix} \theta_{i'}^{(p')}\theta_i^{(q')}\theta_{i'}^{(r')}\otimes \theta_i^{(r''-a'')}\theta_{i'}^{(p''+r'')}\theta_i^{(q''-r''+a'')}.
\end{align*}

We compare the coefficients to verify \eqref{two structure constants}, that is,
$$\dot{c}^{\dot{b}}_{\dot{b}',\dot{b}''}=v^{\langle (\nu_{w,\lambda''_1})_Y,\lambda'_2\rangle} v^{\langle -|\mathrm{th}_{\lambda'_1,\lambda'_2,w}(\dot{b}')|_{\tilde{Y}},-w\lambda''_1\odot\lambda''_2\rangle}c^{\mathrm{th}_{\lambda_1,\lambda_2,w}(\dot{b})}_{\mathrm{th}_{\lambda''_1,\lambda''_2,w}(\dot{b}''),\mathrm{th}_{\lambda'_1,\lambda'_2,w}(\dot{b}')}.$$
For finite type, the fixed root datum is already $X$-regular. We have  $(\nu_{-1,m''})_X=2m'',(\nu_{-1,m''})_Y=m''$, and so $\langle (\nu_{w,\lambda''_1})_Y,\lambda'_2\rangle=m''n'$. We have $|\mathrm{th}_{m',n',-1}(\dot{b}')|_{\tilde{Y}}=(m'-k'+l')i_{\tilde{Y}}+n'i'_{\tilde{X}}$, and so $\langle -|\mathrm{th}_{m',n',-1}(\dot{b}')|,m''\odot n''\rangle=-(m'+l'-k')m''-n'n''$. It is routine to check that 
$$\alpha_{k',k'',l',l''}=m''n'-(m'+l'-k')m''-n'n''+\beta_{n''-l'',n'-l',m''+l''-k'',m'+l'-k',l'',l'}.$$
So the exponents of $v$ match perfectly. As for the quantum binomial coefficients, for any $a'$ in the linear combination of $r(\mathrm{th}_{m,n,-1}(\theta_i^{(k)}\diamondsuit_{n-m}\,\theta_i^{(l)}))$, notice that the scalar $p'+r'-q'$ is equal to difference between the exponents of $\theta_{i'}$ and $\theta_i$ in $\theta_i^{(r'-a')}\theta_{i'}^{(p'+r')}\theta_i^{(q'-r'+a')}$. This difference in $\theta_i^{(l')}\theta_{i'}^{(n')}\theta_i^{(m'-k')}$ is $n'-(l'+(m'-k'))=t'-l'+k'$ matching the scalar in the quantum binomial coefficients in the linear combination of ${_{s',s''}\dot{\Delta}_{t',t''}}(\theta_i^{(k)}\diamondsuit_t\,\theta_i^{(l)})$. Similarly, the quantum binomial coefficients also match for $t'',k'',l'',n'',m'',p'',q'',r'',a''$. This provides a direct verification of \eqref{two structure constants} and Theorem \ref{thm:str-const-comult} for $\mathfrak{sl}_2$.

\section{Positivity of canonical basis of modified quantum group}\label{Positivity of canonical basis of modified quantum group}

\subsection{Positivity of transition matrix} 
Let $\lambda_1,\lambda_2\in X^+$. In this subsection, we establish the positivity property of the transition matrix from the pure tensor basis $\mathbf{B}({^{\omega}\Lambda_{\lambda_1}})\otimes \mathbf{B}(\Lambda_{\lambda_2})$ to the canonical basis $\mathbf{B}({^{\omega}\Lambda_{\lambda_1}}\otimes \Lambda_{\lambda_2})$ of ${^{\omega}\Lambda_{\lambda_1}}\otimes \Lambda_{\lambda_2}$.

\begin{theorem}\label{thm:trans-pos}
Let $\lambda_1,\lambda_2\in X^+$. Then we have 
$$\mathbf{B}({^{\omega}\Lambda_{\lambda_1}}\otimes \Lambda_{\lambda_2})\subset \mathbb{N}[v^{-1}][\mathbf{B}({^{\omega}\Lambda_{\lambda_1}})\otimes \mathbf{B}(\Lambda_{\lambda_2})].$$
\end{theorem}

\begin{proof}
Let $w\in W$. We show that $$\mathbf{B}({^{\omega}V_w(\lambda_1)}\otimes \Lambda_{\lambda_2})\subset \mathbb{N}[v, v^{-1}][\mathbf{B}({^{\omega}V_w(\lambda_1)})\otimes \mathbf{B}(\Lambda_{\lambda_2})].$$ 
For any $\tilde{b}\in \mathbf{B}(\mathbf{f}\theta_{\lambda_2}\mathbf{f})$, we write $r(\tilde{b})=\sum_{\tilde{b}_1,\tilde{b}_2\in \tilde{\mathbf{B}}}c_{\tilde{b}_1,\tilde{b}_2}^{\tilde{b}}\tilde{b}_1\otimes\tilde{b}_2$. Then $c_{\tilde{b}_1,\tilde{b}_2}^{\tilde{b}}\in \mathbb{N}[v,v^{-1}]$ by the positivity property of $\tilde{\mathbf{B}}$ with respect to the comultiplication in Theorem \ref{14.4.13}. By Theorem \ref{thm:structure-trans}, there is a bijection
$\underline{\varphi_{-w\lambda_1,\lambda_2}}:\mathbf{B}((\mathbf{f}\theta_{\lambda_2}\mathbf{f})_{-w\lambda_1\odot\lambda_2})\setminus\mathrm{ker}\,\underline{\varphi_{-w\lambda_1,\lambda_2}}\rightarrow \mathbf{B}({^{\omega}V_w(\lambda_1)}\otimes \Lambda_{\lambda_2})$ such that
$$\underline{\varphi_{-w\lambda_1,\lambda_2}}(\tilde{b}^-v_{-w\lambda_1\odot\lambda_2})=\sum c_{b_2\theta_{\lambda_2},b_1}^{\tilde{b}}b_1^-\xi_{-w\lambda_1}\otimes b_2^-\eta_{\lambda_2}$$ 
for any $\tilde{b}^-v_{-w\lambda_1\odot\lambda_2}\in \mathbf{B}((\mathbf{f}\theta_{\lambda_2}\mathbf{f})_{-w\lambda_1\odot\lambda_2})\setminus\mathrm{ker}\,\underline{\varphi_{-w\lambda_1,\lambda_2}}$, where the summation is taken over $b_1^-\xi_{-w\lambda_1}\in \mathbf{B}({^{\omega}V_w(\lambda_1)})$ and $b_2^-\eta_{\lambda_2}\in \mathbf{B}(\Lambda_{\lambda_2})$. 
Consequently, every canonical basis element in $\mathbf{B}(^{\omega}V_{w}(\lambda_1)\otimes\Lambda_{\lambda_2})$ can be written as a $(\mathbb{N}[v,v^{-1}])$-linear combination of pure tensor basis elements in $\mathbf{B}({^{\omega}V_w(\lambda_1)})\otimes \mathbf{B}(\Lambda_{\lambda_2})$.

We have $\mathbf{B}({^{\omega}\Lambda_{\lambda_1}})\otimes \mathbf{B}(\Lambda_{\lambda_2})=\bigcup_{w\in W}\mathbf{B}({^{\omega}V_w(\lambda_1)})\otimes \mathbf{B}(\Lambda_{\lambda_2})$. By Proposition \ref{canonical basis of demazure otimes highest weight}, we have $\mathbf{B}(^{\omega}\Lambda_{\lambda_1}\otimes\Lambda_{\lambda_2})=\bigcup_{w\in W}\mathbf{B}(^{\omega}V_{w}(\lambda_1)\otimes\Lambda_{\lambda_2})$, and so $\mathbf{B}({^{\omega}\Lambda_{\lambda_1}}\otimes \Lambda_{\lambda_2})\subset \mathbb{N}[v,v^{-1}][\mathbf{B}({^{\omega}\Lambda_{\lambda_1}})\otimes \mathbf{B}(\Lambda_{\lambda_2})]$.

By Theorem \ref{24.3.3}, the transition matrix from $\mathbf{B}(^{\omega}\Lambda_{\lambda_1})\otimes \mathbf{B}(\Lambda_{\lambda_2})$ to $\mathbf{B}(^{\omega}\Lambda_{\lambda_1}\otimes\Lambda_{\lambda_2})$ is unitriangular with off-diagonal entries strictly in $v^{-1}\mathbb{Z}[v^{-1}]$. Thus the entries of the transition matrix are contained in $\mathbb{N}[v,v^{-1}]\cap\mathbb{Z}[v^{-1}]=\mathbb{N}[v^{-1}]$. 
\end{proof}

\subsection{Positivity of bilinear form}
Lusztig introduced in \cite[Theorem 26.1.2]{Lusztig-1993} the symmetric bilinear form $$(\,,\,)_{\dot{\mathbf{U}}}:\dot{\mathbf{U}}\times \dot{\mathbf{U}}\rightarrow \mathbb{Q}(v).$$
He conjectured that the canonical basis $\dot{\mathbf{B}}$ admits the positivity property with respect to this bilinear form. It was proved by McGerty \cite{McGerty-2012} for finite type A and affine type A using the geometric realization. In this subsection, we establish this positivity for any symmetric Cartan datum. 

\begin{theorem}\label{positivity of bilinear form}
Let $\dot{b},\dot{b}'\in \dot{\mathbf{B}}$. Then we have
$$(\dot{b},\dot{b}')_{\dot{\mathbf{U}}}\in \delta_{\dot{b},\dot{b}'}+v^{-1}\mathbb{N}[[v^{-1}]].$$
\end{theorem}
\begin{proof}
Let $\dot{b}\in \dot{\mathbf{B}}\cap \dot{\mathbf{U}}1_{\zeta},\dot{b}'\in \dot{\mathbf{B}}\cap \dot{\mathbf{U}}1_{\zeta'}$ with $\zeta,\zeta'\in X$. We may suppose that $\zeta=\zeta'$, otherwise,  $(\dot{b},\dot{b}')_{\dot{\mathbf{U}}}=0$ by \cite[Theorem 26.1.2(a)]{Lusztig-1993}. By \cite[Proposition 26.2.3]{Lusztig-1993}, $(\dot{b},\dot{b}')_{\dot{\mathbf{U}}}$ is the limit of $(\dot{b}(\xi_{-\lambda_1}\otimes \eta_{\lambda_2}),\dot{b}'(\xi_{-\lambda_1}\otimes \eta_{\lambda_2}))_{{{^{\omega}\Lambda_{\lambda_1}}\otimes \Lambda_{\lambda_2}}}$, when $\langle i,\lambda_1\rangle,\langle i,\lambda_2\rangle\rightarrow \infty$ for any $i\in I$ and $\zeta=\lambda_2-\lambda_1$, where $(\,,\,)_{{^{\omega}\Lambda_{\lambda_1}}\otimes \Lambda_{\lambda_2}}$ is defined by 
$$(m_1\otimes m_2,m'_1\otimes m'_2)_{{^{\omega}\Lambda_{\lambda_1}}\otimes \Lambda_{\lambda_2}}=(m_1,m'_1)_{{^{\omega}\Lambda_{\lambda_1}}}(m_2,m'_2)_{\Lambda_{\lambda_2}}$$
for any $m_1,m'_1\in {^{\omega}\Lambda_{\lambda_1}},m_2,m'_2\in \Lambda_{\lambda_2}$, and $(\,,\,)_{{^{\omega}\Lambda_{\lambda_1}}},(\,,\,)_{\Lambda_{\lambda_2}}$ are given by \cite[Proposition 19.1.2]{Lusztig-1993}. By \cite[Theorem 4.16]{Zheng-2014}, the canonical bases of ${^{\omega}\Lambda_{\lambda_1}}$ and $\Lambda_{\lambda_2}$ have the positivities with respect to $(\,,\,)_{{^{\omega}\Lambda_{\lambda_1}}}$ and $(\,,\,)_{\Lambda_{\lambda_2}}$, that is, 
$$(\mathbf{B}({^{\omega}\Lambda_{\lambda_1}}),\mathbf{B}({^{\omega}\Lambda_{\lambda_1}}))_{{^{\omega}\Lambda_{\lambda_1}}}, (\mathbf{B}(\Lambda_{\lambda_2}),\mathbf{B}(\Lambda_{\lambda_2}))_{\Lambda_{\lambda_2}}\subset \mathbb{N}[v^{-1}].$$ 
By Theorem \ref{25.2.1} and Theorem \ref{thm:structure-trans}, we have
\begin{align*}
\dot{b}(\xi_{-\lambda_1}\otimes \eta_{\lambda_2}),\dot{b}'(\xi_{-\lambda_1}\otimes \eta_{\lambda_2})\in \mathbf{B}({^{\omega}\Lambda_{\lambda_1}}\otimes \Lambda_{\lambda_2})\subset \mathbb{N}[v^{-1}][\mathbf{B}({^{\omega}\Lambda_{\lambda_1}})\otimes \mathbf{B}(\Lambda_{\lambda_2})].
\end{align*}
Hence $(\dot{b}(\xi_{-\lambda_1}\otimes \eta_{\lambda_2}),\dot{b}'(\xi_{-\lambda_1}\otimes \eta_{\lambda_2}))_{{{^{\omega}\Lambda_{\lambda_1}}\otimes \Lambda_{\lambda_2}}}\in \mathbb{N}[v^{-1}]$, and so $(\dot{b},\dot{b}')_{\dot{\mathbf{U}}}\in \mathbb{N}[[v^{-1}]]$. By \cite[Theorem 26.3.1]{Lusztig-1993}, we have $(\dot{b},\dot{b}')_{\dot{\mathbf{U}}}\in \delta_{\dot{b},\dot{b}'}+v^{-1}\mathbb{Z}[[v^{-1}]]$. Hence 
$(\dot{b},\dot{b}')_{\dot{\mathbf{U}}}\in \delta_{\dot{b},\dot{b}'}+v^{-1}\mathbb{N}[[v^{-1}]]$.
\end{proof}

\subsection{Positivity of comultiplication}
By leveraging the relation between the structure constants of the comultiplication of $\dot{\mathbf{U}}$ and the negative part of the larger quantum group established in Theorem \ref{thm:str-const-comult}, we deduce the positivity of the comultiplication.

The comultiplication of $\dot{\mathbf{U}}$ can be regarded as a single map
$$\dot{\Delta}:\dot{\mathbf{U}}\rightarrow \prod_{\gamma',\gamma'',\zeta',\zeta''\in X}{_{\gamma'}\mathbf{U}_{\zeta'}}\otimes {_{\gamma''}\mathbf{U}_{\zeta''}}\hookrightarrow \dot{\mathbf{U}}\hat{\otimes}\dot{\mathbf{U}},$$
where $\dot{\mathbf{U}}\hat{\otimes}\dot{\mathbf{U}}$ is the completion of $\dot{\mathbf{U}}\otimes\dot{\mathbf{U}}$ consisting of formal linear combinations $\sum_{\dot{b},\dot{b}'\in \dot{\mathbf{B}}}k_{\dot{b},\dot{b}'}\dot{b}\otimes \dot{b}'$ with $k_{\dot{b},\dot{b}'}\in \mathbb{Q}(v)$ (see \cite[\S 23.1.5]{Lusztig-1993} and \cite[\S 1.11]{Lusztig-2009}). We denote
$$\mathbb{N}[v, v^{-1}][\dot{\mathbf{B}}\hat{\otimes} \dot{\mathbf{B}}]=\{\sum_{\dot{b},\dot{b}'\in \dot{\mathbf{B}}}k_{\dot{b},\dot{b}'}\dot{b}\otimes \dot{b}'\in \dot{\mathbf{U}}\hat{\otimes}\dot{\mathbf{U}};k_{\dot{b},\dot{b}'}\in \mathbb{N}[v,v^{-1}],\dot{b},\dot{b}'\in \dot{\mathbf{B}}\}.$$

\begin{theorem} \label{thm:comult-positivity}
Let $\dot{b} \in \dot{\mathbf{B}}$. Then we have
$$\Delta(\dot{b}) \in \mathbb{N}[v, v^{-1}][\dot{\mathbf{B}}\hat{\otimes} \dot{\mathbf{B}}].$$
\end{theorem}
\begin{proof}
For any $\tilde{b}\in \mathbf{B}(\mathbf{f}\theta_{\lambda_2}\mathbf{f})$, we write $r(\tilde{b})=\sum_{\tilde{b}_1,\tilde{b}_2\in \tilde{\mathbf{B}}}c_{\tilde{b}_1,\tilde{b}_2}^{\tilde{b}}\tilde{b}_1\otimes\tilde{b}_2$. Then $c_{\tilde{b}_1,\tilde{b}_2}^{\tilde{b}}\in \mathbb{N}[v,v^{-1}]$ by the positivity property of $\tilde{\mathbf{B}}$ with respect to the comultiplication in Theorem \ref{14.4.13}. Let $\zeta',\zeta''\in X$ and $\dot{b}\in \dot{\mathbf{B}}\cap \dot{\mathbf{U}}1_{\zeta'+\zeta''}$. By Theorem \ref{thm:str-const-comult}, we have 
\begin{align*}
\dot{\Delta}_{\zeta',\zeta''}(\dot{b})=v^{\langle (\nu_{w,\lambda''_1})_Y,\lambda'_2\rangle}\sum v^{\langle -|\mathrm{th}_{\lambda'_1,\lambda'_2,w}(\dot{b}')|_{\tilde{Y}},-w\lambda''_1\odot\lambda''_2\rangle} c^{\mathrm{th}_{\lambda_1,\lambda_2,w}(\dot{b})}_{\mathrm{th}_{\lambda''_1,\lambda''_2,w}(\dot{b}''),\mathrm{th}_{\lambda'_1,\lambda'_2,w}(\dot{b}')}\dot{b}'\otimes \dot{b}'',
\end{align*}
where the summation is taken over $\dot{b}'\in\dot{\mathbf{B}}\cap \dot{\mathbf{U}}1_{\zeta'}$ and $\dot{b}''\in\dot{\mathbf{B}}\cap \dot{\mathbf{U}}1_{\zeta''}$, for sufficiently regular $\lambda'_1,\lambda''_1,\lambda'_2,\lambda''_2\in X^+$ and sufficiently large $w\in W$ satisfy certain conditions. Hence we have $\dot{\Delta}_{\zeta',\zeta''}(\dot{b})\in \mathbb{N}[v,v^{-1}[\dot{\mathbf{B}}\otimes \dot{\mathbf{B}}]$. Summing over all blocks, we conclude that $\dot{\Delta}(\dot{b})\in\mathbb{N}[v,v^{-1}][\dot{\mathbf{B}}\hat{\otimes}\dot{\mathbf{B}}]$.
\end{proof}

\subsection{The involution $\omega$ and canonical basis}
The algebra automorphism $\omega:\mathbf{U}\rightarrow \mathbf{U}$ induces a linear isomorphism ${_{\gamma}\mathbf{U}_{\zeta}}\rightarrow {_{-\gamma}\mathbf{U}_{-\zeta}}$ for any $\gamma,\zeta\in X$. Taking direct sums, we obtain an algebra automorphism $\omega:\dot{\mathbf{U}}\rightarrow \dot{\mathbf{U}}$ (see \cite[\S 23.1.6]{Lusztig-1993}). By \cite[Corollary 26.3.2]{Lusztig-1993}, $\omega(\dot{\mathbf{B}})\subset \pm\dot{\mathbf{B}}$. Lusztig conjectured that $\omega(\dot{\mathbf{B}})=\dot{\mathbf{B}}$, and he proved in \cite[Proposition 3.16]{Lusztig-2023} for finite type. In this subsection, we verify this conjecture in general. 

Let $\lambda_1,\lambda_2\in X^+$. We define the linear map
$$\Omega_{-\lambda_1,\lambda_2}:{^{\omega}\Lambda_{\lambda_1}}\otimes \Lambda_{\lambda_2}\rightarrow {^{\omega}({^{\omega}\Lambda_{\lambda_2}}\otimes \Lambda_{\lambda_1})}$$
by $u_1\xi_{-\lambda_1}\otimes u_2\eta_{\lambda_2} \mapsto {}^{\omega}(\omega(u_2)\xi_{-\lambda_2}\otimes \omega(u_1)\eta_{\lambda_1})$ for any $u_1,u_2\in \mathbf{U}$. Similarly, we have the linear map $\Omega_{-\lambda_2, \lambda_1}:{^{\omega}\Lambda_{\lambda_2}}\otimes \Lambda_{\lambda_1}\rightarrow {^{\omega}({^{\omega}\Lambda_{\lambda_1}}\otimes \Lambda_{\lambda_2})}$, and it induces the linear map ${^{\omega}\Omega_{-\lambda_2,\lambda_1}}:{^{\omega}({^{\omega}\Lambda_{\lambda_2}}\otimes \Lambda_{\lambda_1})}\rightarrow {^{\omega}\Lambda_{\lambda_1}}\otimes \Lambda_{\lambda_2}$ by ${^{\omega}\Omega_{-\lambda_2,\lambda_1}}({^{\omega}m})={^{\omega}(\Omega_{-\lambda_2,\lambda_1}(m))}$ for any $m\in {^{\omega}\Lambda_{\lambda_2}}\otimes \Lambda_{\lambda_1}$. Then it is easy to check that ${^{\omega}\Omega_{-\lambda_2,\lambda_1}}$ is inverse to $\Omega_{-\lambda_1,\lambda_2}$. Hence $\Omega_{-\lambda_1,\lambda_2}$ is a linear isomorphism.

\begin{lemma}\label{isomorphism Omega}
Let $\lambda_1,\lambda_2\in X^+$. Then $\Omega_{-\lambda_1,\lambda_2}:{^{\omega}\Lambda_{\lambda_1}}\otimes \Lambda_{\lambda_2}\rightarrow {^{\omega}({^{\omega}\Lambda_{\lambda_2}}\otimes \Lambda_{\lambda_1})}$ is a $\mathbf{U}$-module isomorphism such that
$$\Omega_{-\lambda_1,\lambda_2}(b_1^+\xi_{-\lambda_1}\diamondsuit b_2^-\eta_{\lambda_2})={}^{\omega}(b_2^+\xi_{-\lambda_2}\diamondsuit b_1^-\eta_{\lambda_1})$$
for any $b_1\in \mathbf{B}(\lambda_1)$ and $b_2\in \mathbf{B}(\lambda_2)$.
\end{lemma}
\begin{proof}
We simply denote $\Omega_{-\lambda_1,\lambda_2}=\Omega$, 

Let $\tau:\mathbf{U}\otimes \mathbf{U}\rightarrow \mathbf{U}\otimes \mathbf{U},u_1\otimes u_2\mapsto u_2\otimes u_1$ be the natural linear isomorphism. Then we have $(\omega\otimes \omega)\Delta=\tau\Delta\omega$ (see the commutative diagram in the proof of \cite[Proposition 25.1.3]{Lusztig-1993}). For any $u\in \mathbf{U}$, we write $\Delta(u)=\sum u_1\otimes u_2\in \mathbf{U}\otimes \mathbf{U}$, and then $\Delta\omega(u)=\sum \omega(u_2)\otimes \omega(u_1)$. For any $u'_1,u'_2\in \mathbf{U}$, we have 
\begin{align*}
\Omega(u(u'_1\xi_{-\lambda_1}\otimes u'_2\eta_{\lambda_2}))=&\Omega(\sum u_1u'_1\xi_{-\lambda_1}\otimes u_2u'_2\eta_{\lambda_2})\\
=&\sum {}^{\omega}(\omega(u_2)\omega(u'_2)\xi_{-\lambda_2}\otimes \omega(u_1)\omega(u'_1)\eta_{\lambda_1}),\\
u\Omega(u'_1\xi_{-\lambda_1}\otimes u'_2\eta_{\lambda_2})=&u({^{\omega}(\omega(u'_2)\xi_{-\lambda_2}\otimes \omega(u'_1)\eta_{\lambda_1})})\\
=&{^{\omega}((\omega(u))(\omega(u'_2)\xi_{-\lambda_2}\otimes \omega(u'_1)\eta_{\lambda_1}))}\\=
&\sum {}^{\omega}(\omega(u_2)\omega(u'_2)\xi_{-\lambda_2}\otimes \omega(u_1)\omega(u'_1)\eta_{\lambda_1}).
\end{align*}
Hence the linear isomorphism $\Omega$ is a $\mathbf{U}$-module isomorphism.

Let $\Psi$ and $\Psi'$ be the involutions on ${^{\omega}\Lambda_{\lambda_1}}\otimes \Lambda_{\lambda_2}$ and ${^{\omega}\Lambda_{\lambda_2}}\otimes \Lambda_{\lambda_1}$ defined by \eqref{Phi-involution} respectively. For any $u_1,u_2\in \mathbf{U}$, on the one hand, we have
\begin{align*}
\Psi'({^{\omega}(\Omega(u_1\xi_{-\lambda_1}\otimes u_2\eta_{\lambda_2})}))=&\Psi'(\omega(u_2)\xi_{-\lambda_2}\otimes \omega(u_1)\eta_{\lambda_1})\\
=&\sum_{\nu\in \mathbb{N}[I]}(-v)^{\mathrm{tr}\,\nu}\sum_{b\in B_{\nu}}b^-\overline{\omega(u_2)\xi_{-\lambda_2}}\otimes b^{*+}\overline{\omega(u_1)\eta_{\lambda_1}}\\
=&\sum_{\nu\in \mathbb{N}[I]}(-v)^{\mathrm{tr}\,\nu}\sum_{b\in B_{\nu}}b^-\omega(\overline{u_2})\xi_{-\lambda_2}\otimes b^{*+}\omega(\overline{u_1})\eta_{\lambda_1};
\end{align*}
on the other hand, we have 
\begin{align*}
{^{\omega}(\Omega\Psi(u_1\xi_{-\lambda_1}\otimes u_2\eta_{\lambda_2}))}=&{^{\omega}(\Omega(\sum_{\nu\in \mathbb{N}[I]}(-v)^{\mathrm{tr}\,\nu}\sum_{b\in B_{\nu}}b^-\overline{u_1\xi_{-\lambda_1}}\otimes b^{*+}\overline{u_2\eta_{\lambda_2}}))}\\
=&\sum_{\nu\in \mathbb{N}[I]}(-v)^{\mathrm{tr}\,\nu}\sum_{b\in B_{\nu}}b^{*-}\omega(\overline{u_2})\xi_{-\lambda_2}\otimes b^+\omega(\overline{u_1})\eta_{\lambda_1}.
\end{align*}

Notice that both $\{b; b \in B_{\nu}\}$ and $\{b^*; b \in B_{\nu}\}$ are bases of $\mathbf{f}_{\nu}$ which are dual to each other. So $\sum_{b\in B_{\nu}}b^-\omega(\overline{u_2})\xi_{-\lambda_2}\otimes b^{*+}\omega(\overline{u_1})\eta_{\lambda_1}=\sum_{b\in B_{\nu}}b^{*-}\omega(\overline{u_2})\xi_{-\lambda_2}\otimes b^+\omega(\overline{u_1})\eta_{\lambda_1}$, and $\Psi'({^{\omega}(\Omega(u_1\xi_{-\lambda_1}\otimes u_2\eta_{\lambda_2})}))={^{\omega}(\Omega\Psi(u_1\xi_{-\lambda_1}\otimes u_2\eta_{\lambda_2}))}$. 

Let $b_1\in \mathbf{B}(\lambda_1)$ and $b_2\in \mathbf{B}(\lambda_2)$. By Theorem \ref{24.3.3}, we have $\Psi(b_1^+\xi_{-\lambda_1}\diamondsuit b_2^-\eta_{\lambda_2})=b_1^+\xi_{-\lambda_1}\diamondsuit b_2^-\eta_{\lambda_2}\in \mathbb{Z}[v^{-1}][\mathbf{B}({^{\omega}\Lambda_{\lambda_1}})\otimes \mathbf{B}(\Lambda_{\lambda_2})]$ and $b_1^+\xi_{-\lambda_1}\diamondsuit b_2^-\eta_{\lambda_2}-b_1^+\xi_{-\lambda_1}\otimes b_2^-\eta_{\lambda_2}\in v^{-1}\mathbb{Z}[v^{-1}][\mathbf{B}({^{\omega}\Lambda_{\lambda_1}})\otimes \mathbf{B}(\Lambda_{\lambda_2})]$. Then 
\begin{align*}
&\Psi'({^{\omega}(\Omega(b_1^+\xi_{-\lambda_1}\diamondsuit b_2^-\eta_{\lambda_2})}))={^{\omega}(\Omega(b_1^+\xi_{-\lambda_1}\diamondsuit b_2^-\eta_{\lambda_2})})\in \mathbb{Z}[v^{-1}][\mathbf{B}({^{\omega}\Lambda_{\lambda_2}})\otimes \mathbf{B}(\Lambda_{\lambda_1})];\\
&{^{\omega}(\Omega(b_1^+\xi_{-\lambda_1}\diamondsuit b_2^-\eta_{\lambda_2}))}-b_2^+\xi_{-\lambda_2}\otimes b_1^-\eta_{\lambda_1}\in v^{-1}\mathbb{Z}[v^{-1}][\mathbf{B}({^{\omega}\Lambda_{\lambda_2}})\otimes \mathbf{B}(\Lambda_{\lambda_1})],
\end{align*}
Hence ${^{\omega}(\Omega(b_1^+\xi_{-\lambda_1}\diamondsuit b_2^-\eta_{\lambda_2}))}$ satisfies the characterizing conditions in Theorem \ref{24.3.3}, and so ${^{\omega}(\Omega(b_1^+\xi_{-\lambda_1}\diamondsuit b_2^-\eta_{\lambda_2}))}=b_2^+\xi_{-\lambda_2}\diamondsuit b_1^-\eta_{\lambda_1}$, that is, we have $\Omega(b_1^+\xi_{-\lambda_1}\diamondsuit b_2^-\eta_{\lambda_2})={}^{\omega}(b_2^+\xi_{-\lambda_2}\diamondsuit b_1^-\eta_{\lambda_1})$.
\end{proof}

Now we show that the involution $\omega$ preserves the canonical basis. 

\begin{proposition}\label{prop:omega}
Let $\zeta\in X$ and $b_1,b_2\in \mathbf{B}$. Then $\omega(b_1\diamondsuit_{\zeta} b_2)=b_2\diamondsuit_{-\zeta}\, b_1$.
\end{proposition}
\begin{proof}
By Theorem Theorem \ref{25.2.1}, we have $\omega(b_1\diamondsuit_{\zeta} b_2)\in \omega({_{\mathcal{A}}\dot{\mathbf{U}}1_{\zeta}})={_{\mathcal{A}}\dot{\mathbf{U}}1_{-\zeta}}$. For any $\lambda_1,\lambda_2\in X^+$ with $\zeta=\lambda_2-\lambda_1$ and $b_1\in \mathbf{B}(\lambda_1),b_2\in \mathbf{B}(\lambda_2)$, we have 
\begin{align*}
&{}^{\omega}(\omega(b_1 \diamondsuit_{\zeta} b_2)(\xi_{-\lambda_2}\otimes \eta_{\lambda_1}))=(b_1\diamondsuit_{\zeta} b_2)({^{\omega}(\xi_{-\lambda_2}\otimes \eta_{\lambda_1}}))=(b_1\diamondsuit_{\zeta} b_2)\Omega_{-\lambda_1,\lambda_2}(\xi_{-\lambda_1}\otimes \eta_{\lambda_2})\\
&=\Omega_{-\lambda_1,\lambda_2}((b_1\diamondsuit_{\zeta} b_2)(\xi_{-\lambda_1}\otimes \eta_{\lambda_2}))=\Omega_{-\lambda_1,\lambda_2}(b_1^+\xi_{-\lambda_1}\diamondsuit b_2^-\eta_{\lambda_2})={}^{\omega}(b_2^+\xi_{-\lambda_2}\diamondsuit b_1^-\eta_{\lambda_1}).
\end{align*}
So $\omega(b_1 \diamondsuit_{\zeta} b_2)(\xi_{-\lambda_2}\otimes \eta_{\lambda_1})=b_2^+\xi_{-\lambda_2}\diamondsuit b_1^-\eta_{\lambda_1}$. Hence $\omega(b_1\diamondsuit_{\zeta} b_2)$ satisfies the characterizing conditions in Theorem \ref{25.2.1} (1), and so $\omega(b_1\diamondsuit_{\zeta} b_2)=b_2\diamondsuit_{-\zeta}\, b_1$.
\end{proof}

\subsection{Positivity of multiplication}

Let $J\subset I$ be a spherical subset. Then the spherical parabolic subalgebras $\dot{\mathbf{U}}_J$ and ${^{\omega}\dot{\mathbf{U}}_J}$ have bases $\dot{\mathbf{B}}\cap \dot{\mathbf{U}}_J$ and $\dot{\mathbf{B}}\cap {^{\omega}\dot{\mathbf{U}}_J}$ respectively. A canonical basis element of $\dot{\mathbf{B}}$ is called {\it spherical parabolic}, if it is contained in $\dot{\mathbf{B}}\cap\dot{\mathbf{U}}_J$ or $\dot{\mathbf{B}}\cap{^{\omega}\dot{\mathbf{U}}_J}$ for some spherical subset $J\subset I$.
\begin{theorem}\label{thm:mult}
For any $\dot{b},\dot{b}'\in \dot{\mathbf{B}}$, if one of them is spherical parabolic, then 
$$\dot{b}\dot{b}'\in \mathbb{N}[v,v^{-1}][\dot{\mathbf{B}}].$$
\end{theorem}
\begin{proof}
For any $b_1,b_2\in \tilde{\tilde{\mathbf{B}}}$, we write $b_1b_2=\sum_{b\in \tilde{\tilde{\mathbf{B}}}}m_{b_1,b_2}^b b$. Then $m_{b_1,b_2}^b\in \mathbb{N}[v,v^{-1}]$ be the positivity property of $\tilde{\tilde{\mathbf{B}}}$ with respect to the multiplication in Theorem \ref{14.4.13}. 

(a) If $\dot{b}\in \dot{\mathbf{B}}\cap \dot{\mathbf{U}}_J$ is spherical parabolic, then for any $\dot{b}'\in \dot{\mathbf{B}}$, by Theorem \ref{thm:structure-multiplication}, 
$$\dot{b}\dot{b}'=\sum_{\dot{b}''\in \dot{\mathbf{B}}}m_{\sigma(\mathrm{th}_{\lambda_1,\lambda_2,w}(\dot{b}')),\widetilde{\mathrm{th}}_{\tilde{\lambda}_1,\tilde{\lambda}_2,\tilde{w}}(\sigma(\dot{\tilde{b}}))}^{\widetilde{\mathrm{th}}_{\tilde{\lambda}_1,\tilde{\lambda}_2,\tilde{w}}(\sigma(1\diamondsuit_{-w\lambda_1\odot\lambda_2}\,\mathrm{th}_{\lambda_1,\lambda_2,w}(\dot{b}'')))}\dot{b}''\in \mathbb{N}[v,v^{-1}][\dot{\mathbf{B}}]$$
for sufficiently regular $\l_1, \l_2 \in X^+$, $\tilde{\lambda}_1,\tilde{\lambda}_2\in \tilde{X}^+$ and sufficiently large $w\in W,\tilde{w}\in \tilde{W}$ satisfy certain conditions.

(b) By Proposition \ref{prop:omega}, $\omega(\dot{\mathbf{B}})=\dot{\mathbf{B}}$. If $\dot{b}\in \mathbf{B}\cap {^{\omega}\dot{\mathbf{U}}_J}$ is spherical parabolic, then $\omega(\dot{b})\in \mathbf{B}\cap \dot{\mathbf{U}}_J$ is spherical parabolic by definition. For any $\dot{b}'\in \dot{\mathbf{B}}$, by (a), we have $\omega(\dot{b})\omega(\dot{b}')\in \mathbb{N}[v,v^{-1}][\dot{\mathbf{B}}]$. Applying $\omega$, we obtain $\dot{b}\dot{b}'\in \mathbb{N}[v,v^{-1}][\dot{\mathbf{B}}]$.

(c) By \cite[Theorem 4.3.2]{Kashiwara-1994}, $\sigma(\dot{\mathbf{B}})=\dot{\mathbf{B}}$. If $\dot{b}'\in \dot{\mathbf{B}}\cap \dot{\mathbf{U}}_J$ or $\dot{\mathbf{B}}\cap{^{\omega}\dot{\mathbf{U}}_J}$ is spherical parabolic, then $\sigma(\dot{b}')\in \dot{\mathbf{B}}\cap \dot{\mathbf{U}}_J$ or $\dot{\mathbf{B}}\cap{^{\omega}\dot{\mathbf{U}}_J}$ is spherical parabolic. For any $\dot{b}\in \dot{\mathbf{B}}$, by (a) or (b), we have $\sigma(\dot{b}')\sigma(\dot{b})\in \mathbb{N}[v,v^{-1}][\dot{\mathbf{B}}]$. Applying $\sigma$, we obtain $\dot{b}\dot{b}'\in \mathbb{N}[v,v^{-1}][\dot{\mathbf{B}}]$.
\end{proof}

\subsection{Generalization to arbitrary root datum}

For the $Y$-regular root datum $(Y,X,\langle\,,\,\rangle,\ldots)$ of type $(I,\cdot)$, the canonical basis $\dot{\mathbf{B}}$ of the modified quantum group $\dot{\mathbf{U}}$ is constructed by gluing the canonical bases of tensor products (see Theorem \ref{25.2.1}). In general, for arbitrary root datum $(Y',X',\langle\,,\,\rangle',\ldots)$ of type $(I,\cdot)$, the canonical basis $\dot{\mathbf{B}}'$ of the associated modified quantum group $\dot{\mathbf{U}}'$ is constructed from $\dot{\mathbf{B}}$. In this subsection, we generalize the positivity properties of $\dot{\mathbf{B}}$ to $\dot{\mathbf{B}}'$.

There is a unique morphism $\alpha:Y\rightarrow Y',\beta:X'\rightarrow X$ of root data from the simply connected root datum $(Y,X,\langle\,,\,\rangle,\ldots)$ to $(Y',X',\langle\,,\,\rangle',\ldots)$ (see \cite[\S 2.2.2]{Lusztig-1993}). By \cite[\S 3.1.2 \& \S 23.2.5 \& \S 25.2.2]{Lusztig-1993}, the natural algebra homomorphism $\phi:\mathbf{U}\rightarrow \mathbf{U}'$ defined by $\phi(E_i)=E_i,\phi(F_i),\phi(K_{\mu})=K_{\alpha(\mu)}$ for any $i\in I$ and $\mu\in Y$ induces a linear isomorphism $\dot{\phi}_{\zeta}:\dot{\mathbf{U}}1_{\beta(\zeta')}\rightarrow \dot{\mathbf{U}}'1_{\zeta'}$ for any $\zeta'\in X'$ such that $\dot{\phi}_{\zeta'}(x^+y^-1_{\beta(\zeta')})=x^+y^-1_{\zeta'}$ for any $x,y\in \mathbf{f}$. For any $\zeta'\in X'$ and $b_1,b_2\in \mathbf{B}$, we denote $b_1\diamondsuit_{\zeta'}b_2=\dot{\phi}_{\zeta'}(b_1\diamondsuit_{\beta(\zeta')}b_2)$, where $b_1\diamondsuit_{\beta(\zeta')}b_2\in \dot{\mathbf{B}}\cap \dot{\mathbf{U}}1_{\beta(\zeta')}$ is the canonical basis element of $\dot{\mathbf{U}}$ given by Theorem \ref{25.2.1}. The canonical basis of $\dot{\mathbf{U}}'$ is 
$$\dot{\mathbf{B}}':=\{b_1\diamondsuit_{\zeta'}b_2;\zeta'\in X',b_1,b_2\in \mathbf{B}\}.$$

Taking the direct sum of $\dot{\phi}_{\zeta}$ over $\zeta'\in X'$, we obtain a linear isomorphism $$\dot{\phi}=\bigoplus_{\zeta'\in X'}\dot{\phi}_{\zeta}:\bigoplus_{\zeta'\in X'}\dot{\mathbf{U}}1_{\beta(\zeta')}\rightarrow \bigoplus_{\zeta'\in X'}\dot{\mathbf{U}}'1_{\zeta'}=\dot{\mathbf{U}}'.$$
By definition, we have $\dot{\phi}(\bigcup_{\zeta'\in X'}(\dot{\mathbf{B}}\cap \dot{\mathbf{U}}1_{\beta(\zeta')}))=\dot{\mathbf{B}}'$. Moreover, if $\dot{b}\in \dot{\mathbf{B}}\cap \dot{\mathbf{U}}1_{\beta(\zeta')}$ is spherical parabolic, then $\dot{\phi}(\dot{b})$ is also spherical parabolic. 

By definition, $\phi:\mathbf{U}\rightarrow \mathbf{U}'$ is a bialgebra homomorphism preserving the multiplication and the comultiplication. Hence $\dot{\phi}$ also preserve the multiplication and the comultiplication. Let $(\,,\,)_{\dot{\mathbf{U}}}$ and $(\,,\,)_{\dot{\mathbf{U}}'}$ be the unique bilinear forms on $\dot{\mathbf{U}}$ and $\dot{\mathbf{U}}'$ respectively (see \cite[Theorem 26.1.2]{Lusztig-1993}). Then the bilinear form $(\,,\,)'_{\dot{\mathbf{U}}'}:\dot{\mathbf{U}}'\otimes \dot{\mathbf{U}}'\rightarrow \mathbb{Q}(v)$ defined by 
$$(u_1,u_2)'_{\dot{\mathbf{U}}'}=(\dot{\phi}^{-1}(u_1),\dot{\phi}^{-1}(u_2))_{\dot{\mathbf{U}}}\ \textrm{for any}\ u_1,u_2\in \dot{\mathbf{U}}'$$
satisfies the  characterizing conditions of $(\,,\,)_{\dot{\mathbf{U}}'}$ in \cite[Theorem 26.1.2(a)-(c)]{Lusztig-1993}. Hence we have $(\dot{\phi}(u_1),\dot{\phi}(u_2))_{\dot{\mathbf{U}}'}=(u_1,u_2)_{\dot{\mathbf{U}}}$ for any $u_1,u_2\in \bigoplus_{\zeta'\in X'}\dot{\mathbf{U}}1_{\beta(\zeta')}$.

As results, the positivity properties of $\dot{\mathbf{B}}$ with respect to the spherical multiplication, the comultiplication and the bilinear form give the same positivity properties of $\dot{\mathbf{B}}'$.

\section{Positivity of canonical basis of tensor product}\label{Positivity of canonical basis of tensor product}

By using the positivity property of the canonical basis of the modified quantum group, we establish the positivity property of the canonical basis of tensor product.

\subsection{Canonical basis of tensor product in general case}\label{Canonical basis of tensor product sequel}

Let $m,n\in \mathbb{N}$ and $\lambda_1,\ldots,\lambda_{m+n}\in X^+$. We denote the sequence and the tensor product.
\begin{align*}
\blambda&=(-\lambda_1,\ldots,-\lambda_m,\lambda_{m+1},\ldots,\lambda_{m+n}),\\ \Lambda(\blambda)&={^{\omega}\Lambda_{\lambda_1}}\otimes \cdots \otimes {^{\omega}\Lambda_{\lambda_m}}\otimes \Lambda_{\lambda_{m+1}}\otimes\cdots\otimes \Lambda_{\lambda_{m+n}}.
\end{align*}
We define the $\Psi$-involution on $\Lambda(\blambda)$ by induction on $m,n\in \mathbb{N}$. 

For any $\mathbf{U}$-module $M$ and $\lambda\in X^+$, by \cite[\S 24.1.1]{Lusztig-1993}, the quasi-$\mathcal{R}$-matrix $\Theta$ induces well-defined linear transformations on the tensor products $^{\omega}\Lambda_{\lambda}\otimes M$ and $M\otimes \Lambda_{\lambda}$. The $\Psi$-involution on ${^{\omega}}\Lambda_{\lambda_1}\otimes \Lambda_{\lambda_2}$ is defined by \eqref{Phi-involution}. The $\Psi$-involutions on ${^{\omega}}\Lambda_{\lambda_1}\otimes {^{\omega}\Lambda_{\lambda_2}}$ and $\Lambda_{\lambda_1}\otimes \Lambda_{\lambda_2}$ are defined by the same formula, that is $\Psi(m_1\otimes m_2)=\Theta(\overline{m_1}\otimes \overline{m_2})$ for any $m_1\in {^{\omega}\Lambda_{\lambda_1}},m_2\in {^{\omega}\Lambda_{\lambda_2}}$ or $m_1\in \Lambda_{\lambda_1},m_2\in \Lambda_{\lambda_2}$. Suppose that the $\Psi$-involution on $\Lambda(\blambda)$ has been defined, then the $\Psi$-involutions on ${^{\omega}\Lambda_{\lambda}}\otimes \Lambda(\blambda)$ and $\Lambda(\blambda)\otimes \Lambda_{\lambda}$ for any $\lambda$ are defined by 
\begin{align*}
&\Psi(m\otimes m')=\Theta(\overline{m}\otimes \Psi(m'))\ \textrm{for any}\ m\in {^{\omega}\Lambda_{\lambda}},m'\in \Lambda(\lambda);\\
&\Psi(m\otimes m')=\Theta(\Psi(m)\otimes \overline{m'})\ \textrm{for any}\ m\in \Lambda(\blambda),m'\in \Lambda_{\lambda}.
\end{align*}

We denote 
$\mathcal{L}(\blambda)=\mathbb{Z}[v^{-1}][\mathbf{B}(^{\omega}\Lambda_{\lambda_1})\otimes\cdots\otimes \mathbf{B}(^{\omega}\Lambda_{\lambda_m})\otimes \mathbf{B}(\Lambda_{\lambda_{m+1}})\otimes \cdots\otimes \mathbf{B}(\Lambda_{\lambda_{m+n}})]$. The canonical basis of the tensor product $\Lambda(\blambda)$ was established by Bao-Wang in \cite[Theorem 2.9]{Bao-Wang-2016}.

\begin{theorem}
Let $m,n\in \mathbb{N}$ and $\lambda_1,\ldots,\lambda_{m+n}\in X^+$. Then for any $b_s\in \mathbf{B}(^{\omega}\Lambda_s),1\leqslant s\leqslant m$ and $b_{m+t}\in \mathbf{B}(\Lambda_{m+t}),1\leqslant t\leqslant n$, there exists a unique 
$b_1\diamondsuit\cdots\diamondsuit b_{m+n}\in \mathcal{L}(\blambda)$
such that 
$$\Psi(b_1\diamondsuit\cdots\diamondsuit b_{m+n})=b_1\diamondsuit\cdots\diamondsuit b_{m+n}\ \textrm{and}\ b_1\diamondsuit\cdots\diamondsuit b_{m+n}-b_1\otimes\cdots\otimes b_{m+n}\in v^{-1}\mathcal{L}(\blambda).$$  
Moreover, $\mathbf{B}(\Lambda(\blambda)):=\mathbf{B}(^{\omega}\Lambda_{\lambda_1})\diamondsuit\cdots\diamondsuit \mathbf{B}(^{\omega}\Lambda_{\lambda_m})\diamondsuit \mathbf{B}(\Lambda_{\lambda_{m+1}})\diamondsuit \cdots\diamondsuit \mathbf{B}(\Lambda_{\lambda_{m+n}})$ is a basis of $\Lambda(\blambda)$. The transition matrix from $\mathbf{B}(^{\omega}\Lambda_{\lambda_1})\otimes\cdots\otimes \mathbf{B}(^{\omega}\Lambda_{\lambda_m})\otimes \mathbf{B}(\Lambda_{\lambda_{m+1}})\otimes \cdots\otimes \mathbf{B}(\Lambda_{\lambda_{m+n}})$ to $\mathbf{B}(\Lambda(\blambda))$ is unitriangular with off-diagonal entries in $v^{-1}\mathbb{Z}[v^{-1}]$.
\end{theorem}

The canonical basis of $\Lambda(\blambda)$ is $\mathbf{B}(\Lambda(\blambda))$. Moreover, $(\Lambda(\blambda),\mathbf{B}(\Lambda(\blambda)))$ is a based $\mathbf{U}$-module.


\subsection{Realization of tensor product of simple highest weight modules}

In this subsection, we establish the thickening realization of the tensor product of simple highest weight modules.

Let $\lambda_1,\lambda_2\in X^+$. By \eqref{definition of odot}, the thickening product $\lambda_1\odot\lambda_2\in \tilde{X}^+$ is dominant, then we have the Verma module $M_{\lambda_1\odot\lambda_2}$, the simple highest weight module $\Lambda_{\lambda_1\odot\lambda_2}$ of $\tilde{\mathbf{U}}$ with the highest weight $\lambda_1\odot \lambda_2$. We regard them as $\mathbf{U}$-modules via the imbedding $\mathbf{U}\hookrightarrow \tilde{\mathbf{U}}$, and define the $\mathbf{U}$-module $\Lambda_{\lambda_1,\lambda_2}$ to be the image of $(\mathbf{f}\theta_{\lambda_2}\mathbf{f})_{\lambda_1\odot\lambda_2}\subset M_{\lambda_1\odot\lambda_2}$ under the natural projection $\pi_{\lambda_1\odot\lambda_2}:M_{\lambda_1\odot\lambda_2}\rightarrow \Lambda_{\lambda_1\odot\lambda_2}$.

By \cite[Corollary 4.5]{Fang-Lan-2025}, $\mathbf{B}(\Lambda_{\lambda_1,\lambda_2})=:\{b^-\eta_{\lambda_1\odot\lambda_2};b\in \mathbf{B}(\mathbf{f}\theta_{\lambda_2}\mathbf{f})\}\setminus \{0\}$ is a basis of $\Lambda_{\lambda_1,\lambda_2}$, and $(\Lambda_{\lambda_1,\lambda_2},\mathbf{B}(\Lambda_{\lambda_1,\lambda_2}))$ is a based $\mathbf{U}$-submodule of $(\Lambda_{\lambda_1\odot\lambda_2},\mathbf{B}(\Lambda_{\lambda_1\odot\lambda_2}))$.

The following result, first established in \cite[Corollary 7.21]{Li-2014} and \cite[Theorem 3.8 \& Theorem 4.7]{Fang-Lan-2025}, realizes the tensor product of two simple highest weight modules as a based $\mathbf{U}$-submodule of a simple highest weight module of $\tilde{\mathbf{U}}$. Our proof here is different and simpler.

\begin{lemma}\label{morphism underlinepsi of case 2}
Let $\lambda_1,\lambda_2\in X^+$. Then there are based $\mathbf{U}$-module isomorphisms
$$\xymatrix@C=2cm{(\Lambda_{\lambda_1,\lambda_2}, \mathbf{B}(\Lambda_{\lambda_1,\lambda_2})) \ar@<.5ex>[r]^-{\overline{\varphi_{\lambda_1,\lambda_2}}} &(\Lambda_{\lambda_1}\otimes \Lambda_{\lambda_2}, \mathbf{B}(\Lambda_{\lambda_1}\otimes \Lambda_{\lambda_2})).\ar@<.5ex>[l]^-{\overline{\psi_{\lambda_1,\lambda_2}}}}$$
\end{lemma}
\begin{proof}
By Lemma \ref{isomorphism varphi}, there are $\mathbf{U}$-module isomorphisms $\varphi_{\lambda_1,\lambda_2}:(\mathbf{f}\theta_{\lambda_2}\mathbf{f})_{\lambda_1\odot\lambda_2}\rightarrow M_{\lambda_1}\otimes \Lambda_{\lambda}$ and $\psi_{\lambda_1,\lambda_2}:M_{\lambda_1}\otimes \Lambda_{\lambda}\rightarrow (\mathbf{f}\theta_{\lambda_2}\mathbf{f})_{\lambda_1\odot\lambda_2}$ which are inverse to each other.

By \cite[Lemma 3.2 \& Definition 3.3]{Fang-Lan-2025}, we can identify 
\begin{align*}
\Lambda_{\lambda_1,\lambda_2}\cong (\mathbf{f}\theta_{\lambda_2}\mathbf{f})_{\lambda_1\odot\lambda_2}/((\mathbf{f}\theta_{\lambda_2}\mathbf{f})_{\lambda_1\odot\lambda_2}\cap \sum_{i\in I}\tilde{\mathbf{f}}\theta_i^{\langle i_Y,\lambda_1\rangle+1}).
\end{align*}
Let $\pi_{\lambda_1,\lambda_2}:(\mathbf{f}\theta_{\lambda_2}\mathbf{f})_{\lambda_1\odot\lambda_2}\rightarrow \Lambda_{\lambda_1,\lambda_2}$ and $\pi_{\lambda_1}:M_{\lambda_1}\rightarrow \Lambda_{\lambda_1}$ be the natural projections. We simply denote $\pi_{\lambda_1,\lambda_2}=\pi,\varphi_{\lambda_1,\lambda_2}=\varphi,\psi_{\lambda_1,\lambda_2}=\psi$, and consider the following diagram 
$$\xymatrix{\mathrm{ker}\,\pi \ar@{^{(}->}[r] \ar@<.5ex>[d]^{\varphi} &(\mathbf{f}\theta_{\lambda_2}\mathbf{f})_{\lambda_1\odot\lambda_2} \ar@<.5ex>[d]^{\varphi} \ar@{->>}[r]^-{\pi} &\Lambda_{\lambda_1,\lambda_2} \ar@{-->}[d]<.5ex>^{\overline{\varphi}}\\
\mathrm{ker}\,(\pi_{\lambda_1}\otimes \mathrm{Id}) \ar@{^{(}->}[r] \ar@<.5ex>[u]^{\psi} &M_{\lambda_1}\otimes \Lambda_{\lambda_2} \ar@{->>}[r]^-{\pi_{\lambda_1}\otimes \mathrm{Id}}  \ar@<.5ex>[u]^{\psi} &\Lambda_{\lambda_1}\otimes \Lambda_{\lambda_2}. \ar@{-->}[u]<.5ex>^{\overline{\psi}}}$$

We prove $\pi\psi(x^-v_{\lambda_1}\otimes y^-\eta_{\lambda_2})=0$ for any $x\in \sum_{i\in I}\mathbf{f}\theta_i^{\langle i_Y,\lambda_1\rangle+1}$ and homogeneous $y\in \mathbf{f}$ by induction on $\mathrm{tr}\,|y|\in \mathbb{N}$. If $\mathrm{tr}\,|y|=0$, then $y\in \mathbb{Q}(v)$. By \eqref{im psi}, we have $\pi\psi(x^-v_{\lambda_1}\otimes\eta_{\lambda_2})=\pi(\theta_{\lambda_2}^-x^-v_{\lambda_1\odot\lambda_2})=0$. If $\mathrm{tr}\,|y|>0$, then it is enough to prove the statement for any $y=\theta_iy'$, with $i\in I$ and homogeneous $y'\in \mathbf{f}$. Notice that \begin{align*}
F_i(x^-v_{\lambda_1}\otimes y'^-\eta_{\lambda_2})=&x^-v_{\lambda_1}\otimes F_iy'^-\eta_{\lambda_2}+F_ix^-v_{\lambda_1}\otimes K_{-i_Y}y'^-\eta_{\lambda_2}\\
=&x^-v_{\lambda_1}\otimes y^-\eta_{\lambda_2}+(\theta_ix)^-v_{\lambda_1}\otimes v^{\langle -i_Y,\lambda_2-|y'|_X\rangle}y'^-\eta_{\lambda_2}.
\end{align*}
Since $\pi\psi$ commutes with the actions of $F_i$, and $\mathrm{tr}\,|y'|<\mathrm{tr}\,|y|$, by the inductive hypothesis, we have $\pi\psi(x^-v_{\lambda_1}\otimes y^-\eta_{\lambda_2})=0$, as desired. Hence $\pi\psi(\mathrm{ker}\,(\pi_{\lambda_1}\otimes \mathrm{Id}))=0$,
and then $\pi\psi$ induces the $\mathbf{U}$-module homomorphism $\overline{\psi}$. By the proof of \cite[Theorem 3.8]{Fang-Lan-2025}, $(\pi_{\lambda_1}\otimes \mathrm{Id})\varphi(\mathrm{ker}\,\pi)=0$, and $(\pi_{\lambda_1}\otimes \mathrm{Id})\varphi$ induces the $\mathbf{U}$-module homomorphism $\underline{\varphi}$. From the commutative diagram, $\underline{\varphi}$ and $\underline{\psi}$ are the inverses of each other, and both of them are $\mathbf{U}$-module isomorphisms.

By \cite[Theorem 14.4.11]{Lusztig-1993}, $\pi_{\lambda_1}:M_{\lambda_1}\rightarrow \Lambda_{\lambda_1}$ is a based $\mathbf{U}$-module homomorphism, and then so is $\pi_{\lambda_1}\otimes \mathrm{Id}:M_{\lambda_1}\otimes \Lambda_{\lambda_2}\rightarrow \Lambda_{\lambda_1}\otimes \Lambda_{\lambda_2}$. By \cite[Corollary 4.5]{Fang-Lan-2025}, we have $\pi_{\lambda_1,\lambda_2}(\mathbf{B}(\mathbf{f}\theta_{\lambda_2}\mathbf{f})_{\lambda_1\odot\lambda_2})\subset \mathbf{B}(\Lambda_{\lambda_1,\lambda_2})\sqcup \{0\}$, and $\mathbf{B}(\mathbf{f}\theta_{\lambda_2}\mathbf{f})_{\lambda_1\odot\lambda_2}\cap \mathrm{ker}\,\pi_{\lambda_1,\lambda_2}$ is a basis of $\mathrm{ker}\,\pi_{\lambda_1,\lambda_2}$. Combining with Theorem \ref{positivity of transition matrix}, both $\overline{\varphi}$ and $\overline{\psi}$ are based $\mathbf{U}$-module isomorphisms.
\end{proof}

Let $n\geqslant 1$. We define the Cartan datum $(\tilde{I}^n,\cdot)$ and the $\tilde{Y}^n$-regular root datum $(\tilde{Y}^n,\tilde{X}^n,\langle\,,\,\rangle,\ldots)$ of type $(\tilde{I}^n,\cdot)$ by induction. For $n=1$, $(\tilde{I}^1,\cdot)=(\tilde{I},\cdot)$ and  $(\tilde{Y}^1,\tilde{X}^1,\langle\,,\,\rangle,\ldots)=(\tilde{Y},\tilde{X},\langle\,,\,\rangle,\ldots)$ are the thickening data defined in \S \ref{The thickening product}. For $n>1$, $(\tilde{I}^n,\cdot)$ and $(\tilde{Y}^n,\tilde{X}^n,\langle\,,\,\rangle,\ldots)$ are defined to be the thickening data of $(\tilde{I}^{n-1},\cdot)$ and $(\tilde{Y}^{n-1},\tilde{X}^{n-1},\langle\,,\,\rangle,\ldots)$ respectively. For convenience, we denote $(\tilde{I}^0,\cdot)=(I,\cdot)$ and $(\tilde{Y}^0,\tilde{X}^0,\langle\,,\,\rangle,\ldots)=(Y,X,\langle\,,\,\rangle,\ldots)$.

Let $\lambda_1,\ldots,\lambda_n\in X^+$. We inductively define $\lambda_1\odot\cdots\odot\lambda_n\in (\tilde{X}^{n-1})^+$ to be the thickening product of $\lambda_1\odot\cdots\odot\lambda_{n-1}\in (\tilde{X}^{n-2})^+$ and $0\odot\cdots\odot 0\odot\lambda_{n}\in (\tilde{X}^{n-2})^+$.

Let $\tilde{\mathbf{f}}^n$ and $\tilde{\mathbf{U}}^n$ be the algebra and the quantum group associated with $(\tilde{I}^n,\cdot)$ and $(\tilde{Y}^n,\tilde{X}^n,\langle\,,\,\rangle,\ldots)$ defined in \S \ref{The algebra f} and \S \ref{Quantized enveloping algebra} respectively. Then there are natural subalgebra imbeddings $\mathbf{f}\hookrightarrow\tilde{\mathbf{f}}\hookrightarrow\cdots\hookrightarrow \tilde{\mathbf{f}}^n$ and $\mathbf{U}\hookrightarrow\tilde{\mathbf{U}}\hookrightarrow\cdots\hookrightarrow \tilde{\mathbf{U}}^n$.

\begin{proposition}\label{thickening realization of tensor product of simple based modules}
Let $n\geqslant 2,\lambda_1,\ldots,\lambda_n\in X^+$ and $\mathbf{B}(\Lambda(\blambda))$ be the canonical basis of the tensor product $\Lambda(\blambda)=\Lambda_{\lambda_1}\otimes\cdots\otimes \Lambda_{\lambda_n}$. Then there exist a based $\mathbf{U}$-submodule $(\Lambda_{\blambda},\mathbf{B}(\Lambda_{\blambda}))$ of the based $\tilde{\mathbf{U}}^{n-1}$-module $(\Lambda_{\lambda_1\odot\cdots\odot\lambda_n},\mathbf{B}(\Lambda_{\lambda_1\odot\cdots\odot\lambda_n}))$, and based $\mathbf{U}$-module isomorphisms 
$$\xymatrix@C=1.5cm{(\Lambda_{\blambda},\mathbf{B}(\Lambda_{\blambda})) \ar@<.5ex>[r]^-{\overline{\varphi_{\blambda}}} &(\Lambda(\blambda),\mathbf{B}(\Lambda(\blambda))). \ar@<.5ex>[l]^-{\overline{\psi_{\blambda}}}}$$
\end{proposition}
\begin{proof}
We prove the statement by induction on $n$. If $n=2$, it is Lemma \ref{morphism underlinepsi of case 2}. If $n>2$, we denote $\blambda'=(\lambda_1,\ldots,\lambda_{n-1})$, $\blambda''=(0,\ldots,0, \lambda_{n})$ and  
$\tilde{\lambda}'=\lambda_1\odot\cdots\odot\lambda_{n-1},\tilde{\lambda}''=0 \odot \cdots \odot 0 \odot \l_n\in (\tilde{X}^{n-2})^+$. By the inductive hypothesis, there exist based $\mathbf{U}$-submodules $(\Lambda_{\blambda'},\mathbf{B}(\Lambda_{\blambda'})),(\Lambda_{\blambda''},\mathbf{B}(\Lambda_{\blambda''}))$ of the based $\tilde{\mathbf{U}}^{n-2}$-module $(\Lambda_{\tilde{\lambda}'},\mathbf{B}(\Lambda_{\tilde{\lambda}'})),(\Lambda_{\tilde{\lambda}''},\mathbf{B}(\Lambda_{\tilde{\lambda}''}))$ respectively, and based $\mathbf{U}$-module isomorphisms
$$\xymatrix@C=1cm{(\Lambda_{\blambda'},\mathbf{B}(\Lambda_{\blambda'})) \ar@<.5ex>[r]^-{\overline{\varphi_{\blambda'}}} &{(\Lambda(\blambda'),\mathbf{B}(\Lambda(\blambda'))) \ar@<.5ex>[l]^-{\overline{\psi_{\blambda'}}},\ (\Lambda_{\blambda''},\mathbf{B}(\Lambda_{\blambda''})) \ar@<.5ex>[r]^-{\overline{\varphi_{\blambda''}}}} &(\Lambda(\blambda''),\mathbf{B}(\Lambda(\blambda''))). \ar@<.5ex>[l]^-{\overline{\psi_{\blambda''}}}}$$
Taking the tensor product of based modules, we obtain a based $\mathbf{U}$-submodule $(\Lambda_{\blambda'}\otimes \Lambda_{\blambda''},\mathbf{B}(\Lambda_{\blambda'}\otimes \Lambda_{\blambda''}))$ of the based $\tilde{\mathbf{U}}^{n-2}$-module $(\Lambda_{\tilde{\lambda}'}\otimes \Lambda_{\tilde{\lambda}''},\mathbf{B}(\Lambda_{\tilde{\lambda}'}\otimes \Lambda_{\tilde{\lambda}''}))$, and based $\mathbf{U}$-module isomorphisms
$$\xymatrix@C=2cm{(\Lambda_{\blambda'}\otimes \Lambda_{\blambda''},\mathbf{B}(\Lambda_{\blambda'}\otimes \Lambda_{\blambda''})) \ar@<.5ex>[r]^-{\overline{\varphi_{\blambda'}}\otimes \overline{\varphi_{\blambda''}}} &(\Lambda(\blambda')\otimes \Lambda(\blambda''),\mathbf{B}(\Lambda(\blambda')\otimes \Lambda(\blambda''))). \ar@<.5ex>[l]^-{\overline{\psi_{\blambda'}}\otimes \overline{\psi_{\blambda''}}}}$$
Since the simple highest weight module $\Lambda_0$ of $\mathbf{U}$ is $\mathbb{Q}(v)$ such that $F_i,E_i$ acts by $0$ and $K_{\mu}$ acts by $1$ for any $i\in I$ and $\mu\in Y$, we can naturally identify $(\Lambda(\blambda''),\mathbf{B}(\Lambda(\blambda'')))=(\Lambda_{\lambda_n},\mathbf{B}(\Lambda_{\lambda_n}))$, and then $(\Lambda(\blambda')\otimes \Lambda(\blambda''),\mathbf{B}(\Lambda(\blambda')\otimes \Lambda(\blambda'')))=(\Lambda(\blambda),\mathbf{B}(\Lambda(\blambda)))$. 

Applying Lemma \ref{morphism underlinepsi of case 2} for the tensor product $(\Lambda_{\tilde{\lambda}'}\otimes \Lambda_{\tilde{\lambda}''},\mathbf{B}(\Lambda_{\tilde{\lambda}'}\otimes \Lambda_{\tilde{\lambda}''}))$ of based $\tilde{\mathbf{U}}^{n-2}$-modules, there exist a based $\tilde{\mathbf{U}}^{n-2}$-submodule $(\Lambda_{\tilde{\lambda}', \tilde{\lambda}''},\mathbf{B}(\Lambda_{\tilde \lambda', \tilde \lambda''})$ of the based $\tilde{\mathbf{U}}^{n-1}$-module $(\Lambda_{\lambda_1\odot\cdots\odot\lambda_n},\mathbf{B}(\Lambda_{\lambda_1\odot\cdots\odot\lambda_n}))$ and based $\tilde{\mathbf{U}}^{n-2}$-modules isomorphisms
\begin{equation*}
\xymatrix@C=2cm{(\Lambda_{\tilde \lambda', \tilde \lambda''},\mathbf{B}(\Lambda_{\tilde \lambda', \tilde \lambda''})) \ar@<.5ex>[r]^-{\overline{\varphi_{\tilde \lambda', \tilde \lambda''}}} &(\Lambda_{\tilde \l'}\otimes \Lambda_{\tilde \l''},\mathbf{B}(\Lambda_{\tilde \l'}\otimes \Lambda_{\tilde \l''})). \ar@<.5ex>[l]^-{\overline{\psi_{\tilde \lambda', \tilde \lambda''}}}}
\end{equation*}
Let $(\Lambda_{\blambda},\mathbf{B}(\Lambda_{\blambda}))=\overline{\psi_{\tilde \l', \tilde \l''}}((\Lambda_{\blambda'}\otimes \Lambda_{\blambda''},\mathbf{B}(\Lambda_{\blambda'}\otimes \Lambda_{\blambda''})))$. Then it is a based $\mathbf{U}$-submodule of $(\Lambda_{\lambda_1\odot\cdots\odot\lambda_n},\mathbf{B}(\Lambda_{\lambda_1\odot\cdots\odot\lambda_n}))$, and 
\begin{align*}
\overline{\varphi_{\blambda}}=(\overline{\varphi_{\blambda'}}\otimes \overline{\varphi_{\blambda''}})(\overline{\varphi_{\tilde \l',\tilde \l''}}|_{\Lambda_{\blambda}}), \
\overline{\psi_{\blambda}}=\overline{\psi_{\tilde \l',\tilde \l''}}(\overline{\psi_{\blambda'}}\otimes \overline{\psi_{\blambda''}})
\end{align*}
are the desired based $\mathbf{U}$-module isomorphisms.
\end{proof}

\subsection{Positivity of canonical basis of tensor product}

\begin{theorem}\label{Positivity in general case}
Let $m,n\in \mathbb{N}$ and $\lambda_1,\ldots,\lambda_{m},\lambda_{m+1},\ldots,\lambda_{m+n}\in X^+$. Then we have 
\begin{enumerate}
\item (Positivity of transition matrix) 
$$\mathbf{B}(\Lambda(\blambda))\subset \mathbb{N}[v^{-1}][\mathbf{B}({^{\omega}\Lambda_{\lambda_1}})\otimes\cdots\otimes \mathbf{B}({^{\omega}\Lambda_{\lambda_m}})\otimes \mathbf{B}(\Lambda_{\lambda_{m+1}})\otimes\cdots\otimes \mathbf{B}(\Lambda_{\lambda_{m+n}})].$$
\item (Positivity of action) Let $\dot{b}\in \dot{\mathbf{B}}$. If $\dot{b}$ is spherical parabolic or $\Lambda(\blambda)$ is the tensor product of simple highest (resp. lowest) weight modules, that is, $m=0$ (resp. $n=0$), then 
\begin{align*}
\dot{b}(\mathbf{B}(\Lambda(\blambda)))\subset \mathbb{N}[v,v^{-1}][\mathbf{B}(\Lambda(\blambda))].
\end{align*}
\end{enumerate}
\end{theorem}
\begin{proof}
Without loss of generality, we may assume that $m=n$ by taking the tensor products ${{^\omega}\Lambda_{0}}\otimes\ldots\otimes {{^\omega}\Lambda_{0}}\otimes \Lambda(\blambda)$ or $\Lambda(\blambda)\otimes \Lambda_0\otimes\ldots\otimes \Lambda_0$ if necessary. We denote $\blambda'=(-\lambda_1,\ldots,-\lambda_n)$ and $\blambda''=(\lambda_{n+1},\ldots,\lambda_{2n})$ so that $\blambda=\blambda'\blambda''$. If $n\geqslant 2$, by Proposition \ref{thickening realization of tensor product of simple based modules}, there exist a based $\mathbf{U}$-submodule $(\Lambda_{\blambda''},\mathbf{B}(\Lambda_{\blambda''}))$ of the based $\tilde{\mathbf{U}}^{n-1}$-module $(\Lambda_{\lambda_{n+1}\odot\cdots\odot\lambda_{2n}},\mathbf{B}(\Lambda_{\lambda_{n+1}\odot\cdots\odot\lambda_{2n}}))$, and a based $\mathbf{U}$-module isomorphism 
\begin{equation}\label{larger highest}
(\Lambda_{\blambda''},\mathbf{B}(\Lambda_{\blambda''}))\cong (\Lambda(\blambda''),\mathbf{B}(\Lambda(\blambda''))).   
\end{equation}
Similarly, there exist a based $\mathbf{U}$-submodule $({^{\omega}\Lambda_{\blambda'}},\mathbf{B}({^{\omega}\Lambda_{\blambda'}}))$ of the based $\tilde{\mathbf{U}}^{n-1}$-module $({^{\omega}\Lambda_{\lambda_1\odot\cdots\odot\lambda_{n}}},\!\mathbf{B}({^{\omega}\Lambda_{\lambda_1\odot\cdots\odot\lambda_{n}}}))$, and a based $\mathbf{U}$-module isomorphism 
\begin{equation}\label{larger lowest}
({^{\omega}\Lambda_{\blambda'}},\mathbf{B}({^{\omega}\Lambda_{\blambda'}}))\cong (\Lambda(\blambda'),\mathbf{B}(\Lambda(\blambda'))).
\end{equation}

Taking the tensor product of based modules, we obtain a based $\mathbf{U}$-submodule $({^{\omega}\Lambda_{\blambda'}}\otimes \Lambda_{\blambda''},\mathbf{B}({^{\omega}\Lambda_{\blambda'}}\otimes \Lambda_{\blambda''}))$ of the based $\tilde{\mathbf{U}}^{n-1}$-module 
\begin{equation}\label{larger lowest tensor larger highest}
({^{\omega}\Lambda_{\lambda_1\odot\ldots\odot \lambda_{n}}}\otimes \Lambda_{\lambda_{n+1}\odot\cdots\odot \lambda_{2 n}},\mathbf{B}({^{\omega}\Lambda_{\lambda_1\odot\cdots\odot \lambda_{n}}}\otimes \Lambda_{\lambda_{n+1}\odot\ldots\odot \lambda_{2 n}})),
\end{equation}
and based $\mathbf{U}$-module isomorphisms
\begin{equation}\label{thickening realization of based modules in general case}
({^{\omega}\Lambda_{\blambda'}}\otimes \Lambda_{\blambda''},\mathbf{B}({^{\omega}\Lambda_{\blambda'}}\otimes \Lambda_{\blambda''}))\cong (\Lambda(\blambda),\mathbf{B}(\Lambda(\blambda)))
\end{equation}

(1) If $n=1$, the statement is Theorem \ref{thm:trans-pos}. If $n\geqslant 2$, applying Theorem \ref{thm:trans-pos} for the tensor product of $\tilde{\mathbf{U}}^{n-1}$-modules in \eqref{larger lowest tensor larger highest}, we have 
$$\mathbf{B}({^{\omega}\Lambda_{\lambda_1\odot\ldots\odot \lambda_{n}}}\otimes \Lambda_{\lambda_{n+1}\odot\cdots\odot \lambda_{2 n}})\subset \mathbb{N}[v^{-1}][\mathbf{B}({^{\omega}\Lambda_{\lambda_1\odot\ldots\odot \lambda_{n}}})\otimes \mathbf{B}(\Lambda_{\lambda_{n+1}\odot\cdots\odot \lambda_{2 n}})],$$ 
and then $\mathbf{B}({^{\omega}\Lambda_{\blambda'}}\otimes \Lambda_{\blambda''})\subset \mathbb{N}[v^{-1}][\mathbf{B}({^{\omega}\Lambda_{\blambda'}})\otimes \mathbf{B}(\Lambda_{\blambda''})]$. By the based $\mathbf{U}$-module isomorphisms in \eqref{larger highest}, \eqref{larger lowest} and \eqref{thickening realization of based modules in general case}, $\mathbf{B}(\Lambda(\blambda))\subset \mathbb{N}[v^{-1}][\mathbf{B}(\Lambda(\blambda'))\otimes \mathbf{B}(\Lambda(\blambda''))]$. 

By Proposition \ref{thickening realization of tensor product of simple based modules}, and applying \cite[Corollary 4.8]{Fang-Lan-2025} for the tensor product of $\tilde{\mathbf{U}}^{n-1}$-modules inductively, we have $\mathbf{B}(\Lambda(\blambda''))\subset \mathbb{N}[v^{-1}][\mathbf{B}(\Lambda_{\lambda_{n+1}})\otimes \cdots \otimes\mathbf{B}(\Lambda_{\lambda_{2n}})]$. Similarly, $\mathbf{B}(\Lambda(\blambda'))\subset \mathbb{N}[v^{-1}][\mathbf{B}({^{\omega}\Lambda_{\lambda_1}})\otimes\cdots\otimes \mathbf{B}({^{\omega}\Lambda_{\lambda_n}})]$. Therefore, we have $$\mathbf{B}(\Lambda(\blambda))\subset \mathbb{N}[v^{-1}][\mathbf{B}({^{\omega}\Lambda_{\lambda_1}})\otimes\cdots\otimes \mathbf{B}({^{\omega}\Lambda_{\lambda_m}})\otimes \mathbf{B}(\Lambda_{\lambda_{m+1}})\otimes\cdots\otimes \mathbf{B}(\Lambda_{\lambda_{m+n}})].$$

(2) We first prove the case that $\dot{b}\in \dot{\mathbf{B}}$ is spherical parabolic. If $n=1$, for any $b_1^+\xi_{-\lambda_1}\diamondsuit b_2^-\eta_{\lambda_2}\in \mathbf{B}({^{\omega}\Lambda_{\lambda_1}}\otimes \Lambda_{\lambda_2})$, by Theorem \ref{thm:mult} and Theorem \ref{25.2.1}, we have $\dot{b}(b_1\diamondsuit_{\lambda_2-\lambda_1}b_2)\in \mathbb{N}[v,v^{-1}][\dot{\mathbf{B}}]$ and 
$$\dot{b}(b_1^+\xi_{-\lambda_1}\diamondsuit b_2^-\eta_{\lambda_2})=\dot{b}(b_1\diamondsuit_{\lambda_2-\lambda_1}b_2)(\xi_{-\lambda_1}\otimes \eta_{\lambda_2})\in \mathbb{N}[v,v^{-1}][\mathbf{B}({^{\omega}\Lambda_{\lambda_1}}\otimes \Lambda_{\lambda_2})].$$
Hence we have 

(a) $\dot{b}(\mathbf{B}({^{\omega}\Lambda_{\lambda_1}}\otimes \Lambda_{\lambda_2}))\subset \mathbb{N}[v,v^{-1}][\mathbf{B}({^{\omega}\Lambda_{\lambda_1}}\otimes \Lambda_{\lambda_2})]$.

If $n\geqslant 2$, let $\dot{\tilde{\mathbf{U}}}^{n-1}$ be the modified quantum group  associated with the root datum $(\tilde{Y}^{n-1},\tilde{X}^{n-1},\langle\,,\,\rangle,\ldots)$, and $\dot{\tilde{\mathbf{B}}}^{n-1}$ be its canonical basis. Then there are natural embeddings $\dot{\mathbf{U}}\subset \dot{\tilde{\mathbf{U}}}^{n-1}$ and $\dot{\mathbf{B}}\subset \dot{\tilde{\mathbf{B}}}^{n-1}$. Applying (a) for $\dot{b}\in \dot{\tilde{\mathbf{B}}}^{n-1}$ and the tensor product of $\tilde{\mathbf{U}}^{n-1}$-modules in \eqref{larger lowest tensor larger highest}, we obtain 
$$\dot{b}(\mathbf{B}({^{\omega}\Lambda_{\lambda_1\odot\cdots\odot \lambda_{n}}}\otimes \Lambda_{\lambda_{n+1}\odot\ldots\odot \lambda_{2 n}}))\subset \mathbb{N}[v,v^{-1}][\mathbf{B}({^{\omega}\Lambda_{\lambda_1\odot\cdots\odot \lambda_{n}}}\otimes \Lambda_{\lambda_{n+1}\odot\ldots\odot \lambda_{2 n}})],$$
and so $\dot{b}(\mathbf{B}({^{\omega}\Lambda_{\blambda'}}\otimes \Lambda_{\blambda''}))\subset \mathbb{N}[v,v^{-1}][\mathbf{B}({^{\omega}\Lambda_{\blambda'}}\otimes \Lambda_{\blambda''})]$. By the based $\mathbf{U}$-module isomorphism in \eqref{thickening realization of based modules in general case}, we have $\dot{b}(\mathbf{B}(\Lambda(\blambda)))\subset \mathbb{N}[v,v^{-1}][\mathbf{B}(\Lambda(\blambda))]$.

Next we prove the case that $\Lambda(\blambda)$ is the tensor product of simple highest weight modules. The case that $\Lambda(\blambda)$ is the tensor product of simple lowest weight modules can be proved in a similar way. Let $\dot{b}\in \dot{\mathbf{B}}$ and $b^-\eta_{\lambda_1}\in \mathbf{B}(\Lambda_{\lambda_1})$. Notice that $1\diamondsuit_{\lambda}b\in \dot{\mathbf{B}}$ is spherical parabolic. By Theorem \ref{thm:mult} and Theorem \ref{25.2.1}, we have $\dot{b}(1\diamondsuit_{\lambda}b)\in \mathbb{N}[v,v^{-1}][\dot{\mathbf{B}}]$ and
$$\dot{b}(b^-\eta_{\lambda_1})=\dot{b}b^-1_{\lambda_1}\eta_{\lambda_1}=\dot{b}(1\diamondsuit_{\lambda}b)\eta_{\lambda_1}\in \mathbb{N}[v,v^{-1}][\mathbf{B}(\Lambda_{\lambda_1})].$$
Hence 

(b) $\dot{b}(\mathbf{B}(\Lambda_{\lambda_1}))\subset \mathbb{N}[v,v^{-1}][\mathbf{B}(\Lambda_{\lambda_1})].$

If $n\geqslant 2$, applying (b) for $\dot{b}\in \dot{\tilde{\mathbf{B}}}^{n-1}$ and the $\tilde{\mathbf{U}}^{n-1}$-modules $\Lambda_{\lambda_1\odot\cdots\odot\lambda_n}$, we obtain 
$$\dot{b}(\mathbf{B}(\Lambda_{\lambda_1\odot\cdots\odot\lambda_n}))\subset \mathbb{N}[v,v^{-1}][\mathbf{B}(\Lambda_{\lambda_1\odot\cdots\odot\lambda_n})],$$
and so $\dot{b}(\mathbf{B}(\Lambda_{\lambda_1,\ldots,\lambda_n}))\subset \mathbb{N}[v,v^{-1}][\mathbf{B}(\Lambda_{\lambda_1,\ldots,\lambda_n})]$. By the based $\mathbf{U}$-module isomorphisms in Proposition \ref{thickening realization of tensor product of simple based modules}, we have $\dot{b}(\mathbf{B}(\Lambda(\blambda)))\subset \mathbb{N}[v,v^{-1}][\mathbf{B}(\Lambda(\blambda))]$.
\end{proof}

\printbibliography

\end{document}